\documentclass[11pt]{article}         
\usepackage{amsmath,amsthm,amsfonts,graphicx}
\newtheorem{assumption}{Assumption}
\newtheorem{theorem}{Theorem}         
\newtheorem{lemma}[theorem]{Lemma}

\theoremstyle{definition}   
 
\newtheorem{definition}[theorem] {Definition}     
     
\newtheorem{algorithm}[theorem]{Algorithm}
\newtheorem{corollary}[theorem]{Corollary}
\newtheorem{example}[theorem]{Example}
\newtheorem{remark}[theorem]{Remark}
   
\DeclareMathOperator{\diag}{diag}

\DeclareMathOperator{\rank}{rank}
\DeclareMathOperator{\nrank}{nrank}

\usepackage{multirow}   
\begin{document}
\title{Superfast Accurate Low Rank Approximation\thanks {Some results of this paper have been  presented at the  
Workshop on Fast Direct Solvers, November 12--13, 2016, Purdue University, West Lafayette, Indiana; the
SIAM Conference on Computational Science and Engineering, February--March 2017,
 Atlanta, Georgia, USA, and the   
 INdAM Meeting --
Structured Matrices in Numerical Linear Algebra:
Analysis, Algorithms and Applications,
Cortona, Italy, September 4--8, 2017.
 Also see \cite{PLSZ16} and \cite{PLSZ17}.}} 
\author{Victor Y. Pan}      
   
\author{Victor Y. Pan$^{[1, 2, 3],[a]}$,
Qi Luan$^{[2],[b]}$, 
John Svadlenka$^{[3],[c]}$, and  Liang Zhao$^{[1, 3],[d]}$
\\
\and\\
$^{[1]}$ Department of Computer Science \\ 
Lehman College of the City University of New York \\
Bronx, NY 10468 USA \\
$^{[2]}$ Ph.D. Program in Mathematics \\
The Graduate Center of the City University of New York \\
New York, NY 10036 USA \\
$^{[3]}$ Ph.D. Program in Computer Science \\
The Graduate Center of the City University of New York \\
New York, NY 10036 USA \\
$^{[a]}$ victor.pan@lehman.cuny.edu \\
http://comet.lehman.cuny.edu/vpan/  \\
$^{[b]}$ qi-luan@yahoo.com \\
$^{[c]}$ jsvadlenka@gradcenter.cuny.edu \\ $^{[d]}$ lzhao1@gc.cuny.edu \\
} 
\date{} 
\maketitle  
  
%------------------------------------------------------------------------------
  
\begin{abstract} 
  
%------------------------------------------------------------------------------
 
Low rank approximation of a matrix   (hereafter we use the acronym {\em LRA}) is a fundamental subject of numerical linear algebra and data mining and analysis and is a hot research area. Nowadays modern massive data sets, commonly called  Big Data, are routinely represented with immense matrices having billions of  entries. Realistically one can  access and process only a small fraction of them and therefore  is restricted to {\em superfast} LRA  algorithms --  using
 sublinear time and  memory space,
that is,  using
much fewer arithmetic operations and memory cells than the input matrix has entries. The customary LRA algorithms  that involve SVD of an input matrix, its rank-revealing factorization, or its 
  random projections are not  superfast, 
but the {\em cross-approximation} algorithms (hereafter we use the acronym {\em C-A}) are superfast and for more than a decade 
have been routinely computing accurate LRAs -- in the form of CUR approximations, which
%uses sublinear memory space and moreover 
 preserves sparsity and structure of an input matrix. No proof, however, has appeared so far that the output LRAs of these or any other superfast algorithms are accurate for the worst case input, and this is no surprise for us -- we  specify a small family of matrices of rank one whose close LRAs cannot be computed by any superfast algorithm.
We prove, however, that  with a high  
probability (hereafter  we use the acronym  {\em whp})
C-A as well as  some other superfast algorithms compute close CUR LRAs 
 of random and random  sparse  matrices allowing their close LRAs. Hence they
 compute close CUR LRAs of 
  average and average sparse matrices allowing their LRAs. 
  
 These results prompt us to apply the same superfast methods to 
 any matrix allowing its LRA and  pre-processed with random multipliers, and
 then  we prove that the  output CUR LRA  is accurate whp in the case of pre-processing with a {\em
 Gaussian random}, {\em SRHT}, or {\em SRFT} multiplier.\footnote{Here and hereafer we use the customary acronyms SRHT and SRFT for                                           subsampled randomized
 Hadamard or Fourier transforms.} Such pre-processing itself is not superfast,
but we replace the above  multipliers with  our sparse and structured ones, arrive at superfast algorithms, and then observe no deterioration of the accuracy of output CUR LRAs in our  extensive tests for real world inputs. 

 Our study provides new insights  into  LRA and  demonstrates the power of C-A and some other superfast CUR LRA algorithms
 as well as of our sparse and  structured  randomized  pre-processing. Our random and average case analysis and our other auxiliary techniques may be of independent interest and may motivate their 
 further exploration.  Two groups od
 distinct LRA techniques have been  proposed independently by researchers in the communities of Numerical Linear Algebra and Computer Science; their 
 synergy enables us to refine a crude but reasonably close LRA superfast and to enhance the efficiency of a C-A step.  
  
 Superfast LRA opens new  opportunities, not available for fast LRA, as we demonstrate by achieving  dramatic acceleration  -- from quadratic to nearly linear arithmetic time -- of 
the bottleneck stage 
  of  the Fast Multipole Method.
 
%------------------------------------------------------------------------------
   
\end{abstract}
 
%------------------------------------------------------------------------------
 

%------------------------------------------------------------------------------

\section{Introduction}\label{sintro}

%------------------------------------------------------------------------------

We expand the abstract in this section and
 supply full details in the main body of the paper. We  keep using the acronyms LRA, C-A,  and whp.

%------------------------------------------------------------------------------

\subsection{LRA, its  definition,  superfast LRA, hard inputs, and our goal}\label{smtv}

%------------------------------------------------------------------------------

LRA of a 
matrix
 is a fundamental subject of
  Numerical Linear Algebra 
 and 
Computer Science. It has enumerable applications to data mining and analysis,
image processing,
noise reduction, seismic inversion, latent
semantic indexing, principal component analysis, machine learning, regularization for ill-posed problems, web search models, tensor decomposition, system 
identification, signal processing,  neuroscience, computer vision, social network analysis, antenna array processing, electronic design automation, telecommunications and mobile       communication, chemometrics, 
psychometrics, biomedical engineering,
and so on  \cite{CML15}, \cite{HMT11},
\cite{M11}, \cite{KS16}, \cite{KB09},  \cite{DMM08}, \cite{MMD08},  \cite{MD09},  \cite{OT10}, \cite{ZBD15}.  

Recall that an $m\times n$ matrix $W$ has {\em numerical rank} at most $r$ (and then we write 
 $\nrank(W)\le r$) if
there exists a
 rank-$r$ matrix $W'$  such that
 (see  Figure \ref{fig1}) 
\begin{equation}\label{eqlrk}  
W=W'+E,~W':=AB,~|E|\le \Delta,
\end{equation}
for a pair of matrices $A$ of size 
 $m\times r$ and $B$ of size 
$r\times n$, a~fixed~matrix~norm $|\cdot|$, and~a~fixed~small tolerance $\Delta$.
\begin{figure}[h] 
\centering
\includegraphics[scale=0.25]{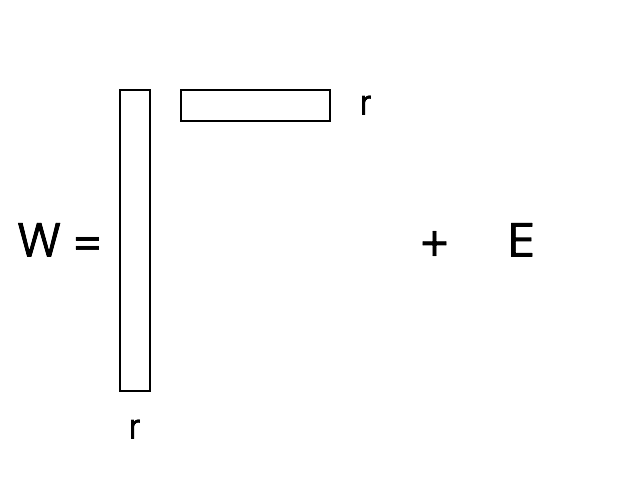}
\caption{Rank-$r$ approximation of a matrix}\label{fig1}
\end{figure} 
 The matrix $W'$ is
 % said to be 
 an {\em LRA} of the matrix $W$ if $r\ll\min\{m,n\}$.

 Generally we store an $m\times n$ matrix $W$ by using $mn$ memory cells  and multiply it by a vector  by using $(2n-1)m$ {\em  flops}.\footnote{Here and hereafter flop stands for
floating point arithmetic operation.}
 These upper bounds are optimal for $W$
\cite[Section 2.3]{BM75}, 
but for the matrix $AB$ of (\ref{eqlrk})
represented by the pair of matrices $a$ and$B$ they decrease to 
  $(m+n)r$ memory cells and  $2(m+n)r-m-r$
   flops. 
 For $r\ll \min\{m,n\}$ 
 this means using {\em sublinear computational time and  memory space}, that is, using
 much fewer  flops and memory cells than the input matrix has entries, and  we call such computations  {\em superfast}. 
 Recall that a flop can access only two memory cells, and so any algorithm uses  sublinear memory space if it runs in 
 sublinear time.
 
 Superfast algorithms can compute
LRAs  of a matrix having billions of entries,
  which is too immense to  access and to handle otherwise, but the earlier LRA algorithms based on computing SVD or rank-revealing factorization and even the popular randomized algorithms surveyed in \cite{HMT11}
 are not superfast. In particular the latter algorithms compute LRA of a matrix $W$ by the products 
$QQ^TW$ for appropriate unitary matrices $Q$ of low rank, 
use linear space exceeding $mn$ in order to  represent an $m\times n$ matrix $W'$, and only support its
approximate multiplication by a vector using  $(2n-1)m$ flops. 

Furthermore no superfast algorithm can compute accurate LRA of all matrices allowing close LRAs and even of all matrices of the following small family of rank-1 matrices.

\begin{example}\label{exdlt}
Define $\pm \delta$-{\em   matrices} filled with zeros except  for a  single entry filled with 1 or $-1$.
There are exactly $2mn$ such $m\times n$ matrices of rank 1,
e.g.,  
 eight matrices of size $2\times 2$: 
$$ \begin{pmatrix}1&0\\0&0\end{pmatrix},~
 \begin{pmatrix}0&1\\0&0\end{pmatrix},~ \begin{pmatrix}0&0\\1&0\end{pmatrix},~ \begin{pmatrix}0&0\\0&1\end{pmatrix},~
 \begin{pmatrix}-1&0\\0&0\end{pmatrix},~
 \begin{pmatrix}0&-1\\0&0\end{pmatrix},~ \begin{pmatrix}0&0\\-1&0\end{pmatrix},~ \begin{pmatrix}0&0\\0&-1\end{pmatrix}.$$ 
 \end{example}
The output matrix  of any  superfast  algorithm approximates
nearly $50\%$ of all these matrices   
as poorly as the trivial matrix filled with zeros does.
%, that is, with relative error 1. 
Indeed a superfast algorithm only depends on a small subset of all $mn$ input entries, and so its output is invariant in the input values at all the other entries.
In contrast nearly $mn$ pairs of $\pm \delta$-matrices vary on  these entries by 2. Hence the approximation by a single value is off by at least 1 for one or both of the matrices of such a pair, that is, it is off at least as much as  the trivial approximation by the matrix filled with zeros. Likewise if a superfast LRA algorithm is randomized and  accesses an input entry with a probability $p$, then it fails 
with probability $1-p$ on at least one of the two 
$\pm \delta$-matrices that do not vanish at that entry.
 
Furthermore  we cannot verify correctness of an output superfast unless we restrict the input class with some additional assumptions
such as those in Sections \ref{scurhrd}  and \ref{spstr}.
Indeed
if a superfast algorithm applied to a slightly perturbed   $\pm \delta$-matrix only accesses its nearly vanishing entries, then it would optimize LRA over 
such entries and would  never detect its failure to approximate the only entry close to 1 or $-1$. 
\begin{example}\label{exdltdns}
$\pm \delta$-matrices are
sparse, but add the rank-1 matrix filled with ones to every
$\pm \delta$-matrix and obtain 
 the set of $2mn$ {\em dense} matrices of  rank 2
that are not close to sparse matrices but are as hard for any superfast LRA algorithm as
$\pm \delta$-matrices.
\end{example}
Such examples have not stopped
the authors of \cite{T96}, \cite{T00},
\cite{B00}, and \cite{BR03},  whose   C-A iterations (see Section \ref{sprmrnd}) routinely compute close LRAs
superfast
in computational practice.
Formal support for this empirical
behavior has been missing  so far, and
%\footnote{We assume seeking LRA for all %matrices allowing LRA. The problem may be %more tractable for more restricted matrix %classes or perhaps for closely related %task of approximation of functions in %\cite{B00},  \cite{B11},  \cite{KS16}. }   
 like any superfast algorithm,  C-A iterations fail on the family of 
 $\pm \delta$-matrices. We are going to prove, however, that  C-A and  some other superfast algorithms are reasonably  accurate whp on random and random sparse inputs allowing LRA  and hence on such  average inputs as well.   
  
\subsection{Random, sparse, and average LRA inputs and our first main result}\label{srndav}

%------------------------------------------------------------------------------

 Next we 
define  random, random sparse, average, and average sparse matrices
allowing LRA.   
 Hereafter we call a standard Gaussian  random variable ``{\em Gaussian}" for  short and call an $m\times n$ matrix ``{\em Gaussian}" if it is 
filled with independent identically distributed 
Gaussian variables.\footnote{The information flow in the real world data quite typically contains Gaussian noise.} We  define the average  $m\times n$ matrix
by taking 
the average over all its entries, and now we are going to define 
random and average matrices of a fixed rank $r$. 

Represent an $m\times n$ matrix of rank $r$ as the product $AB$ of two matrices, $A$ of size $m\times r$ and $B$  
of size $r\times n$ (cf. (\ref{eqlrk})),  and call the product a {\em factor-Gaussian} matrix of expected rank $r$ if both of the factors $A$ and $B$ are Gaussian  
or if one of them is Gaussian and  another has full rank $r$ and is well-conditioned;  also call factor-Gaussian of expected rank $r$ the products 
 $AD B$
where  $A$ and $B$ are Gaussian matrices
of sizes $m\times r$ and  
 $r\times n$, respectively, and 
$D=\diag(d_i)_{i=1}^r$ is a 
well-conditioned 
 diagonal matrix, 
such that $\max_{i,j}|d_i/d_j|$ is not large
 (see Section \ref{srmcs}).
Define sparse factor-Gaussian matrices of  expected numerical rank $r$
by replacing some Gaussian entries of the factors $A$ and $B$ with 
zeros so that the expected ranks of the factors $A$ and $B$ equal $r$.
Extend this definition to define
factor-Gaussian and sparse factor-Gaussian matrices 
of expected numerical rank at most $r$. 

Then  define the $m\times n$ average
and average sparse matrices of numerical rank $r$ and at most $r$
as the average over small norm perturbations of  $m\times n$ factor-Gaussian and sparse factor-Gaussian matrices of expected numerical rank $r$
and at most $r$,  respectively.  
 
Now we state our {\bf first main result}: we prove that {\em primitive, C-A, and some other superfast 
algorithms compute  whp accurate CUR LRAs 
of small norm perturbations of  factor-Gaussian 
and  sparse factor-Gaussian 
 matrices of expected low rank; therefore they 
compute accurate CUR LRAs of the average and average sparse matrices of low numerical rank.}  

  In  our tests, in good accordance with the results of our formal analysis,  C-A and some other superfast algorithms have output quite accurate CUR LRAs to a large  class of real world matrices (see Tables \ref{tb_ranrc},    
  \ref{tb_Lp},  and \ref{tb_lr}).
  
%------------------------------------------------------------------------------

\subsection{CUR LRA; primitive, cynical, and cross-approximation algorithms}\label{sprmrnd} 
 
%------------------------------------------------------------------------------
 
We seek  LRA in a convenient compressed form of {\em CUR LRA}\footnote{The pioneering  
papers \cite{GZT95}, \cite{GTZ97}, 
\cite{GTZ97a},  \cite{GT01}, \cite{GT11},
\cite{GOSTZ10}, \cite{M14}, and \cite{O16} define CGR  approximations having 
 nuclei  $G$; ``G" can
 stand, say, for
``germ". We use the acronym CUR, more customary in the West. 
 ``U" can stand, say, for ``unification factor", and we notice the alternatives of CNR, CCR, or CSR with    
$N$, $C$, and $S$ standing for {\em ``nucleus",
``core", and ``seed"}.} and begin with a simple superfast algorithm  for computing
it.   Let the $r\times r$
leading principal block $W_{r,r}$ of an $m\times n$ matrix $W$
be nonsingular and well-conditioned. 
Call it a {\em CUR generator},
call its inverse  $U=W_{r,r}^{-1}$ a
{\em nucleus},
write $C$ and $R$ for the submatrices made up of the first $r$ columns and the first $r$ rows of the matrix $W$, respectively, and compute its  approximation by the  rank-$r$ matrix $CUR$. We call this superfast algorithm {\em primitive}. 
 
We can turn such a  {\em CUR LRA} into  LRA of (\ref{eqlrk}), say, by   computing
 $A:=CU$ and writing $B:=R$.
 Conversely, given an LRA  of (\ref{eqlrk}), our Algorithms \ref{alglratpsvd} and \ref{algsvdtocur}
 combined 
as well as  Algorithms \ref{alglratpsvd} and \ref{algsvdtocur}a  combined compute its CUR LRA superfast.
We call  a CUR LRA
 a low rank {\em CUR decomposition} if the approximation errors vanish
(see Figure \ref{fig2}).

\begin{figure}
[h]
\centering
\includegraphics[scale=0.25]{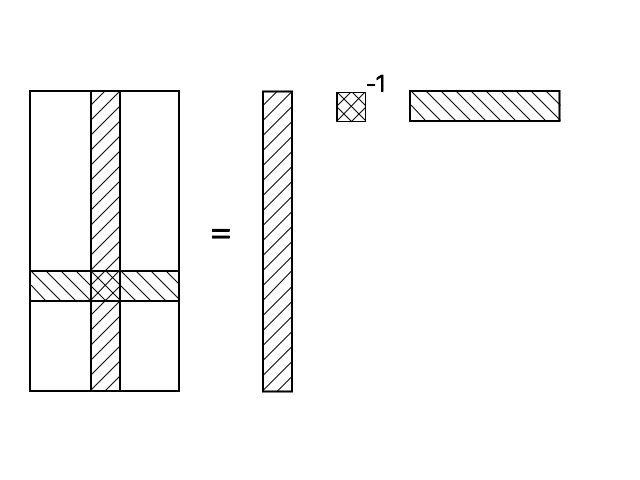}
\caption{CUR decomposition with a nonsingular generator}
\label{fig2}
\end{figure}
      
%------------------------------------------------------------------------------

 We can fix or choose at random
 any other $r\times r$ block or submatrix of $W$, and if it is nonsingular  we can similarly build on it a primitive CUR LRA algorithm. Moreover we can  compute CUR LRA  based on any rectangular $k\times l$
 CUR generator of rank at least $r$ if 
 \begin{equation}\label{eqklmnr}  
0<r\le k\le m~{\rm and}~r\le l\le n.
\end{equation}
Namely, given the factors $R$ of size $k\times n$ and $C$ of size $m\times l$, made up of
 $k$ random rows and  $l$ random
columns of an $m\times n$ matrix $W$,
 define their common
 $k\times l$ submatrix   $W_{k,l}$, call 
 it a {\em sketch}\footnote{This is a very special  sketch; generally sketch is any compact randomized data structure 
that enables approximate computation in low dimension. See  various algorithms involving sketches in \cite{S06} and \cite{W14}.} of $W$, and use it as a CUR  generator.
 Compute a nucleus $U$ from its SVD
 as follows.  First obtain a $k\times l$ matrix
  $W_{k,l,r}$  by setting to 0 all singular values of  $W_{k,l}$ except for the $r$ largest ones. Then let $U=W_{k,l,r}^+$ be the Moore--Penrose generalized inverse of
  this matrix. For $k\ll m$ or $l\ll n$ its computation is superfast, involving $O(kl)$ memory cells and
  $O((k+l)kl)$ flops, which means 
   $r^2$ memory cells and $O(r^3)$ 
flops.\footnote{We can use just $O(r^{\omega})$  flops for $\omega<2.38$
based on asymptotically fast matrix multiplication algorithms, but they supersede the straightforward $r\times r$ matrix multiplication only for immense  
$r$ \cite{P17}.}
  We still call this  algorithm
{\em  primitive}; it varies
  with the choice of dimensions 
  $k$ and $l$. 
   
%------------------------------------------------------------------------------

%\subsection{Cynical and cross-%approximation (C-A) algorithms}
%\label{ssktchc-a}
 
%------------------------------------------------------------------------------
    
Next we generalize  primitive algorithms. For a pair of integers 
$q$ and $s$ such that   
\begin{equation}\label{eqklmnqsr}  
0<r\le k\le q\le m~{\rm and}~r\le l\le s\le n,
\end{equation}
 fix a random $q\times s$   
 sketch of an input matrix $W$. 
 Then by applying any auxiliary LRA algorithm compute a  $k\times l$ CUR generator $W_{k,l}$
 of the sketch, 
 consider it  
 a CUR generator of the  matrix $W$ as well, 
  and build on it a CUR LRA of that matrix. 
  
 For $q=k$ and $s=l$ this is a
primitive
algorithm again. Otherwise 
the algorithm 
 is  still quite primitive; we call it {\em cynical} \footnote{We
allude to  the benefits of the austerity and simplicity of primitive life,  
advocated by Diogenes the Cynic, and not to shamelessness and  distrust associated with modern  cynicism.} (see Figure \ref{fig3}).
 %For the computation of a generator
  %$W_{k,l}$ one can apply any auxiliary %LRA %algorithm,
%such as deterministic algorithms of %\cite{GE96}, \cite{P00}, \cite{GOSTZ10} or %our Appendix \ref{scrsorth},  our  %randomized Algorithms \ref{alglratpsvd1} %or \ref{algrhtcur},
%or randomized algorithms of \cite{DMM08} %or \cite{BW17}.
%For $qs\ll mn$ these randomized %algorithms are superfast. 
%Even the classical computation of LRA via %computing SVD is superfast in this case %if $(q+s)^3\ll mn$.
Its cost depends on the choice of an auxiliary algorithm. Let it be a deterministic algorithm from \cite{GE96} or \cite{P00}. Then  both primitive and cynical algorithms
 involve  $qs$  memory cells, but cynical algorithms use fewer flops, that is, 
   $O(qs\min\{q,s\})$ versus 
   $O(qs\max\{q,s\})$.
In our tests 
  cynical algorithms
incorporating the algorithms from \cite{P00}
 succeeded more consistently than the primitive algorithms
 (see Tables  
\ref{tb_ranrc} and \ref{tb_lr}).

Next enhance the power of cynical algorithms by  recursively
  alternating  their application to vertical  $m\times s$
  and horizontal   $q\times n$
  sketches. By following \cite{T00}
we call such loops  {\em C-A}
iterations
(see Figure \ref{fig4}).
Each C-A step amounts to application of a cynical algorithm to a $q\times s$ sketch
for $q=m$ or $s=n$.
Suppose that we incorporate
       the deterministic  algorithms of \cite{GE96}
       or \cite{P00} into the
         C-A iterations. Then 
$\alpha$ C-A iteration loops involve $(mq+ns)\alpha$ memory cells and $O((m+n)(q+s)^2\alpha)$ flops. 
 For  $\alpha \ll \min\{m,n\}/(q+s)^2$ the computation is superfast. 
Empirically  much fewer  
 C-A loops were always  sufficient in our tests for computing accurate LRAs, and in Section \ref{svlmxmz} we prove that
 already a single two-step C-A loop outputs a rather accurate LRA on the average input and whp on a random input.

\begin{figure}[h] 
\centering
\includegraphics[scale=0.25]{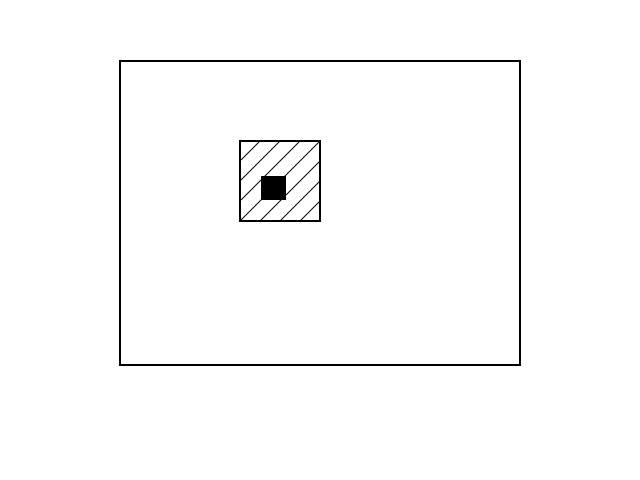}
\caption{A cynical CUR algorithm (the strips mark a sketch; the CUR generator is black)
}\label{fig3} 
\end{figure}

\begin{figure} 
[h]  
\centering
\includegraphics[scale=0.25]{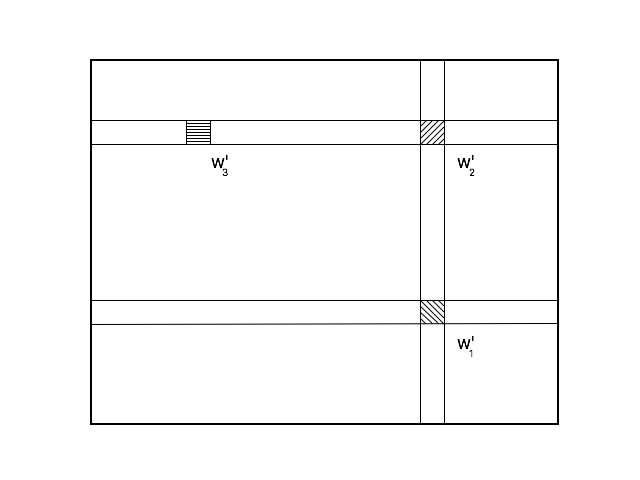}
\caption{The first three recursive steps of a C-A algorithm output three
striped matrices $W_0'$, $W_1'$, and $W_2'$.}
\label{fig4}
\end{figure}
\begin{remark}\label{rec-atnsr}
Given a  $q\times s$ submatrix of an 
$m\times n$ matrix $W$, we can embed
it into  an $m\times s$
vertical or into $q\times n$
horizontal
sketch and then begin C-A iterations
with a sketch that minimizes the output 
error norm or has maximal volume (cf.  CUR criteria 1 and 2  of the next subsection).
In the extension of C-A
iterations to $d$-dimensional tensors,  one would minimize the error norm among the $d$ directions given by {\em fibers} or, say, among  about $0.5d^2$
directions given by  {\em slices}.   
\end{remark}
 
%------------------------------------------------------------------------------

\subsection{Three criteria of the accuracy  of  CUR LRA}\label{scrtaccr}

%------------------------------------------
 
Each of the following three {\em CUR criteria} is sufficient for obtaining accurate CUR LRA, although they imply distinct upper bounds on the errors  of LRA,
whose comparison  with each other is not straightforward (see Remark \ref{rechb}):
\begin{enumerate} 
\item
a CUR generator and an input matrix share numerical rank $r$ (see  Corollary  \ref{cocrtr1}),\footnote{This CUR criterion is also necessary in a small neighborhood of
 inputs of
rank $r$ (see Theorem \ref{thrnkgen}).}  
\item
a $k\times l$ CUR generator has maximal volume among all  $k\times l$ submatrices
of  an input matrix\footnote{A $k\times l$ matrix $M$ has volume  $(\max\{|\det(MM^*)|,|\det(M^*M)|\})^{1/2}$, which is
 $|\det(M)|$ for $k=l$.} (see Sections  
\ref{svlmprv} -- \ref{svlmxmz}), and
\item
the factors  $C$ and $R$ of a CUR approximation have been formed by two
sufficiently large column and row sets,
respectively, sampled at random from an input matrix according to properly 
pre-computed {\em leverage scores} 
     (see Section \ref{slrasmp}).
\end{enumerate}

%------------------------------------------------------------------------------

\subsection{Superfast computation of CUR generators satisfying CUR criteria}\label{ssfstcgr}

%------------------------------------------
 
The first two CUR criteria
had been guiding the study of LRA performed by researchers
in Numerical Linear Algebra,
and we 
 prove that whp the CUR generators computed  by the primitive, cynical, and C-A superfast algorithms satisfy  the first two CUR  criteria 
 for a small norm perturbation of a factor-Gaussian input matrix of low expected numerical rank; hence they satisfy 
 these criteria on the average of such inputs.

 More recently some researchers in Computer Science\footnote{Hereafter we use the acronyms {\em NLA} for Numerical Linear Algebra 
and {\em CS} for Computer Science.} introduced the third CUR
criterion,  dramatically different from the first two (see more details in Section \ref{sextrw}). 
   In particular
%------------------------------------------------------------------------------
 randomized algorithms of \cite{DMM08}
based on the third CUR criterion (see  Section \ref{slrasmpav}) first compute a CUR generator and then a nearly optimal CUR LRA; the algorithms  are 
 superfast except for their initial stage of computing leverage scores.
 
We make this stage and the entire algorithms superfast as well
 by  choosing the {\em uniform  scores}. 
 Then we
 prove that whp the output CUR LRA is still accurate  for small norm perturbations of $m\times n$ factor-Gaussian  
  matrices of numerical rank $r$,
 for $r\ll \min\{m,n\}$, as well as for their averages.       
%------------------------------------------------------------------------------
The restriction that $r\ll \min\{m,n\}$
 limits the value of our proof, but empirically the algorithms of  \cite{DMM08} compute nearly optimal LRAs superfast even without such a restriction. 
%------------------------------------------
 
  We also prove that whp
the CUR generators computed  by the primitive, cynical, and C-A superfast algorithms  satisfy the first CUR criterion for a small norm perturbation of  a sparse factor-Gaussian input matrix of low expected numerical rank and  hence they satisfy 
 this criterion on the average of such inputs, although all our upper bounds on the output errors a little increase in the transition from general to sparse input
 (see Remark \ref{resprs}). Proving any similar results for small norm  perturbations of sparse inputs by using the second or the third CUR criterion  is  a research challenge. 
 
 In sum, given a little perturbed factor-Gaussian or sparse factor-Gaussian  matrix having low numerical rank, we prove that C-A and  some other superfast algorithms compute its accurate CUR LRAs whp and consequently compute 
 accurate CUR LRAs of the average of such matrices. 

%------------------------------------------------------------------------------

\subsection{Superfast computation of CUR LRAs of  matrices pre-processed with random generators}\label{ssfstcurprpr}

%------------------------------------------

Knowing that C-A and some other superfast algorithms compute whp CUR  LRAs
of random inputs allowing their close LRA,
we are motivated to apply these algorithms 
 to any input allowing its LRA and pre-processed with appropriate random multipliers. Indeed we  
prove  our {\bf second main result} that  for a fixed $r$ these {\em  superfast algorithms 
compute  
accurate CUR LRA whp
for  any  matrix of numerical
 rank $r$ 
  pre-processed with  a Gaussian, SRHT, or SRFT multiplier.}
Generally multiplication by such a multipliers is not superfast.
(Otherwise the superfast algorithms  would have computed close LRAs of all rank-$r$ matrices superfast, which is impossible, as we can readily  show by extending Example \ref{exdlt}.) 
Empirically, however, {\em we routinely
arrive at the same output accuracy} when we apply
the same algorithms to the same  
 input matrices pre-processed 
 superfast with  our 
random sparse and structured multipliers
replacing Gaussian, SRHT, and SRFT multipliers. We also specify superfast
algorithms for the transition 
from CUR LRAs of the pre-processed matrix to CUR LRAs of the original matrix. 
           
%------------------------------------------------------------------------------

\subsection{Synergy of the methods from Numerical Linear Algebra  and Computer Science: enhanced accuracy  of LRA
and accelerated C-A steps}\label{saccur}

%------------------------------------------
  
The classical SVD-based computation of an LRA achieves optimal output accuracy but
is not fast. Its acceleration based on rank-revealing factorization sacrifizes about a factor of $\sqrt{(q-r)r}$ in the output error bound. The fast algorithms of \cite{DMM08}, based on CS techniques for LRA and  the third CUR criterion,
achieve a nearly optimal randomized bound
on the Frobenius error norm of the  output. We apply the latter algorithms to
superfast refinement of a crude but reasonably close LRA,  computed, say, by an algorithm from NLA based on CUR criterion 1 or 2.  

We also achieve NLA--CS synergy when
  we apply the algorithms of \cite{DMM08} in order to compute CUR LRAs of the vertical and horizontal sketches processed in C-A iterations. 

%------------------------------------------------------------------------------
 
\subsection{Our auxiliary results of interest and some extensions}\label{saixext}

%------------------------------------------------------------------------------  
  
    Our techniques and auxiliary results can be of independent interest.
We  describe
 superfast algorithms for the transition from any LRA of (\ref{eqlrk}) to CUR LRA,    estimate the norm of the inverse of a sparse    Gaussian matrix (see Remark \ref{resstimp} and Theorem \ref{thgssnrm} and compare a challenge stated in   
\cite{SST06}),    present some novel advanced 
   pre-processing techniques for fast and superfast computation of LRA,
   propose memory efficient decomposition of a Gaussian matrix into the product of random bidiagonal and random permutation matrices, 
      nontrivially prove convergence of that  decomposition, and improve a decade-old estimate for the norm of the inverse of a Gaussian matrix. 
    
    We also analyze and extend some known efficient LRA techniques and algorithms.
     In Section \ref{smd09} we recall 
    the algorithm of \cite{MD09}, extending that of   \cite{DMM08}
   and having valuable applications,
   and observe that it   
  is not superfast. 
  In Section \ref{srprpr} we extend our techniques to
  superfast computation of the Lewis weights, involved in \cite{SWZ17}. 
       In Appendix \ref{scrsorth} we explore Osinsky's numerical approach  of \cite{O16} to enhancing the output accuracy of numerical algorithms for LRA. 
  
            In  Section  \ref{sextpr}
  we demonstrate new LRA application based on using superfast LRA: we  
  accelerate by order of magnitude -- from quadratic to nearly linear time bounds --
  the bottleneck stage of the 
   construction of low rank generators for the Fast Multipole Method.\footnote{Hereafter we use the acronym ``FMM".} FMM 
    is among the most important in the 20th century (see \cite{C00}, \cite{BY13}) and is highly and increasingly popular. 

   Empirically we can extend our study to a larger area.  Whenever our analysis covers  a
  small neighborhood of
rank-$r$
matrices, 
 we can tentatively apply the algorithms to  a larger 
 neighborhood and then
  estimate
  a posteriori errors 
   superfast. 
         
%------------------------------------------------------------------------------
      In Section \ref{ssmr}
     we list some other natural research directions.
  
%------------------------------------------------------------------------------

\subsection{Related Works}\label{sextrw} 
 
%------------------------------------------------------------------------------
 
The reader can access
extensive bibliography for 
 LRA and CUR LRA via
   \cite{HMT11}, \cite{M11},  \cite{W14}, \cite{CBSW14},    
\cite{O16},  \cite{KS16}, \cite{BW17},
\cite{SWZ17}, and the references therein
(also see our Section \ref{sacccmpl}). 
Next we briefly comment on the items 
 most relevant
to our present work.
 
The study of CUR (aka CGR and pseudo-skeleton approximation) can be  traced back to the skeleton decomposition
in \cite{G59} and  QRP factorization
  in \cite{G65}  and \cite{BG65}, redefined and refined as rank-revealing factorization
  in \cite{C87}. 
  
 The LRA algorithms in  \cite{CH90},
   \cite{CH92},
 \cite{HP92},  \cite{HLY92}, 
\cite{CI94},\footnote{Here are some relevant dates for these papers: \cite{CH90} was submitted on 15  December  1986, accepted on 01 February 1989;
\cite{CH92} was submitted on 14  May  1990,
accepted on 10 April 1991;
 \cite{HP92} was submitted on December 1, 1990, revised on February 8, 1991; 
 \cite{HLY92} was submitted on October 8,
1990, accepted on
July 9, 1991, and  \cite{CI94}  appeared as Research Report YALEU/DCS/RR-880, December 1991.}  \cite{GE96}, and \cite{P00} largely rely on the
 maximization of the volumes of CUR generators. This fundamental idea goes back to  \cite{K85} and has been developed  in 
  \cite{GZT95}, \cite{T96}, \cite{GTZ97}, \cite{GTZ97a}, 
\cite{GT01},  \cite{GOSTZ10},
\cite{GT11},  \cite{M14},  and
\cite{O16}.  
In particular our work on the subjects of Part II and Appendix \ref{scntrp} of the present paper was prompted by the significant
 progress of Osinsky in \cite{O16}
 in  improving the accuracy of numerical 
algorithms for CUR LRA based on volume maximization.

The study in 
 \cite{GZT95}, \cite{T96}, \cite{GTZ97}, and \cite{GTZ97a}  towards volume maximization
 revealed the crucial property that the computation of a CUR generator requires no factorization of the input matrix but just  proper selection of its row and column sets.  
  
 C-A iterations were  a natural extension of this observation preceded by the Alternating Least Squares method of \cite{CC70} and 
 \cite{H70} and leading to 
dramatic empirical decrease of 
quadratic and cubic 
memory  space and arithmetic time
used by LRA algorithms, respectively.
 The concept of C-A
was coined in \cite {T00}, and  
we credit \cite {B00},  \cite {BR03}, 
 \cite{MMD08}, \cite{MD09}, \cite {GOSTZ10}, \cite{OT10},  \cite {B11}, and \cite{KV16} for devising some efficient 
 C-A algorithms.
 
 The random sampling approach to LRA 
surveyed in \cite{HMT11} and \cite{M11}  can be traced back to 
the  breakthrough of applying pre-computed leverage scores for random sampling of the rows and columns of an input matrix 
 in \cite{FKW98} and \cite{DK03}. See also early works \cite{FKW04}, \cite{S06}, \cite{DKM06}  and  \cite{MRT06} and  further research progress
   in the papers \cite{DMM08}, \cite{BW14},  \cite{BW17}, and \cite{SWZ17}, which also 
   survey the related work.\footnote{LRA has long been the domain and a major subject of NLA;  then Computer Scientists  proposed powerful randomization techniques, but even earlier  D.E. Knuth, a foremost Computer Scientist, published his pioneering paper  \cite{K85}, which became the springboard for LRA 
   using the criterion of volume maximization.}
 
%------------------------------------------------------------------------------

\subsection{Our previous work and publications}\label{sprvwk}

%------------------------------------------
     
Our  paper  extends
 our study in the papers \cite{PQY15}, \cite{PZ16}, \cite{PLSZ16}, \cite{PZ17}, and \cite{PZ17a},
devoted to

(i) the efficiency of heuristic sparse and structured multipliers for LRA (see \cite{PZ16}, \cite{PLSZ16}), 

(ii) the  approximation of  trailing singular spaces 
associated with the $\nu$ smallest singular values 
of a matrix having numerical nullity $\nu$ (see  \cite{PZ17}), and 

(iii)  using
   random multipliers instead of  pivoting 
  in Gaussian elimination 
(see \cite{PQY15}, \cite{PZ17a}).

In particular we extend the duality approach 
from  \cite{PQY15}, \cite{PZ16}, \cite{PLSZ16}, \cite{PZ17}, and
\cite{PZ17a} for the design and analysis of LRA and other fundamental matrix   computations, e.g., for numerical Gaussian elimination pre-processed with
 sparse and structured random multipliers that replace pivoting.\footnote{Pivoting, that is, row or column interchange, is intensive in data movement, which is costly nowadays.} With this replacement Gaussian elimination has consistently been efficient empirically;  in the papers \cite{PQY15} and \cite{PZ17} we first formally supported that empirical 
 observation whp for Gaussian multipliers
 and any well-conditioned input; then we extended
 this support whp to  any nonsingular well-conditioned multiplier and a  Gaussian input and observed that this covers the average inputs. The papers \cite{PZ16} and \cite{PLSZ16}  extended the approach to LRA;
the report  \cite{PLSZ16} contains 
the main results of our present paper except for those linked to the third CUR criterion, but including the dramatic acceleration of Fast Multipole  Method.  
   
%------------------------------------------------------------------------------
 
\subsection{Organization of the paper}\label{sorgp}

%------------------------------------------------------------------------------  
 
 PART I of our paper, made up of the 
next five sections, is devoted to  
basic definitions and preliminary results,
including the definitions of CUR LRA and its canonical version,  to estimating their output errors, in particular for  random and average input matrices, and to superfast transition from any LRA of a matrix to its top SVD and CUR LRA.
 
In Section \ref{svlm} we recall some basic
definitions and auxiliary results for our study  of general, random, and average matrices. 
 
In Section \ref{sacclb}
we recall the definitions of a CUR LRA and its canonical version. 

In Section \ref{sprmcn} we estimate  a priori and a posteriori output error norms of a CUR LRA of any input matrix.

In Section \ref{serralg1} we estimate such a priori error norms in the case of random and  average input matrices that allow LRA. 
 
In Section \ref{slracur} we describe  
superfast algorithms for the transition from an LRA (\ref{eqlrk}) of a matrix  $W$ at first to its top SVD  and then to its CUR LRA.

\medskip

In PART II of our paper, made up of Sections  \ref{svlmprv} -- \ref{svlmxmz},
we 
study CUR LRA based on the concept of matrix volume and the second CUR criterion for supporting accurate LRAs.

In Section \ref{svlmprv} we recall that concept  and 
estimate the volumes of a slightly perturbed matrix 
and  a matrix product.

In Section \ref {simpvlmxmz} we recall 
 the second CUR criterion (of the maximization of the volume of a CUR generator) and the
upper bounds on the errors of CUR LRA 
implied by this criterion.  
 
In Section \ref{svlmxmz}   we study the impact of C-A iterations on the
 maximization of the volumes of CUR generators. 
\medskip
\medskip
 
In PART III of our paper, made up of Sections  \ref{slrasmp} and \ref{sextgqg},
we study randomized techniques and algorithms for CUR LRA
 based on the third CUR criterion for supporting accurate LRAs
and on randomized pre-processing.

In Section \ref{slrasmp}   we
recall fast  randomized algorithms of \cite{DMM08},
which compute highly accurate CUR LRAs
defined by SVD-based leverage scores
 (thus fulfilling the third CUR criterion), and show their simple superfast modifications that
 refine a crude LRA of any input allowing closer LRA and compute a close CUR LRA of a slightly perturbed factor-Gaussian inputs 
having low numerical rank;  consequently they do this also for the average input
allowing its LRA.
% We also discuss some directions for the %improvement of the efficiency of that %algorithm.

In Section \ref{sextgqg} we prove that fast randomized multiplicative  pre-processing enables superfast computation of accurate  LRA.

\medskip
\medskip

In PART IV of our paper, made up of Sections  \ref{sexp} -- \ref{sextpr},
we  cover the results of our numerical tests, recall  the known estimates for the accuracy and complexity of LRA algorithms,  summarize our study,  and  discuss its extensions.

 In Section \ref{sextpr} we  dramatically accelerate the bottleneck stage of the FMM. 
  
In Section \ref{sexp}  we present the results of our numerical experiments,
which are in good accordance with our  formal study and moreover show that sparse randomized pre-processing  supports  superfast computation of an accurate CUR LRA of a matrix allowing LRA.
 
In Section \ref{sacccmpl} we recall the known estimates for the accuracy and complexity of LRA algorithms.

In Section \ref{ssmr} we summarize the results of our study and list our technical  novelties and some natural research directions. 
 
%------------------------------------------

\medskip
\medskip

In the Appendix we cover various auxiliary results,  some of independent interest. 

In Appendix \ref{srnrmcs} we recall 
the known estimates for the
ranks of random matrices and the norms of a Gaussian matrix and its Moore--Penrose pseudo inverse; we improve the known estimates                                                                                                                                                                                                                                                                                                                                                                                                                                                                                                                                                                                                                                                            for the norm of its pseudo inverse  and extend them to a sparse Gaussian  matrix. 
 
In Appendix \ref{sahaf} we recall the   Hadamard and Fourier multipliers and
describe their  abridged variations.

In Appendix \ref{ssrcs} we recall 
the sampling and rescaling algorithms of \cite{DMM08}
and some results showing their efficiency.

In Appendix \ref{scntrp} we explore  a distinct numerical approach by Osinsky in \cite{O16} to enhancing the output accuracy of numerical algorithms for LRA. In particular
we recall some  efficient CUR algorithms that can be used as subalgorithms of cynical algorithms, 
 revisit the
 greedy cross-approximation iterations
 of \cite{GOSTZ10} based on LUP  factorization, and devise their QRP counterparts based on some technical novelties in \cite{O16}.  
 
\medskip
%\medskip
\medskip

{\bf \Large PART I. 
 DEFINITIONS, AUXILIARY RESULTS, ERROR $~~~~~~$BOUNDS OF LRA, AND
%\medskip
%~~~~~~~~~~~
 SUPERFAST ALGORITHMS} 
% \medskip

%-------------------------------------------------------------------------------

\section{Basic Definitions and Properties}\label{svlm}
  
%------------------------------------------------------------------------------  
%------------------------------------------------------------------------------
  
Hereafter  the concepts ``large", ``small", ``near", ``close", ``approximate",   
``ill-con\-di\-tioned" and 
``well-con\-di\-tioned" are usually 
quantified in  context.
``$\gg$" and ``$\ll$" mean ``much greater than" and ``much less than", respectively.
  
%------------------------------------------------------------------------------

\subsection{General matrix computations}\label{sgnm}

%------------------------------------------------------------------------------  

Recall some basic definitions for matrix computations  (cf. \cite{ABBB99}, \cite{GL13}).

$\mathbb C^{m\times n}$ is the class of $m\times n$
matrices with complex entries.

$I_s$ denotes the $s\times s$ identity matrix. $O_{q,s}$ denotes the $q\times s$  matrix filled with zeros.

$\diag (B_1,\dots,B_k)=\diag(B_j)_{j=1}^k$  denotes a $k\times k$ block diagonal matrix   
with diagonal blocks $B_1,\dots,B_k$.  
  
$(B_1~|~\dots~|~B_k)$ and $(B_1,\dots,B_k)$ 
denote a $1\times k$ block matrix with blocks $B_1,\dots,B_k$.  

$W^T$ and $W^*$ denote the transpose and the Hermitian transpose 
of an $m\times n$ matrix $W=(w_{ij})_{i,j=1}^{m,n}$, respectively.
 $W^*=W^T$ if the  matrix $W$ is real.

$\mathcal {NZ}_W$ denotes the number of  nonzero entries of a matrix $W$.

For two sets $\mathcal I\subseteq\{1,\dots,m\}$  
and $\mathcal J\subseteq\{1,\dots,n\}$   define
the submatrices
\begin{equation}\label{eq111}
W_{\mathcal I,:}:=(w_{i,j})_{i\in \mathcal I; j=1,\dots, n},  
W_{:,\mathcal J}:=(w_{i,j})_{i=1,\dots, m;j\in \mathcal J},~{\rm and}~ 
W_{\mathcal I,\mathcal J}:=(w_{i,j})_{i\in \mathcal I;j\in \mathcal J}.
\end{equation}
%$\sigma_j(W)$ denotes its $j$th largest singular value.

 $||W||=||W||_2$, $||W||_F$, and $||W||_C$
 denote 
spectral, Frobenius, and Chebyshev norms
of a matrix $W$, respectively,
$$ ||W||_F^2:=\sum_{i,j=1}^{m,n}|w_{ij}|^2=\sum_{j=1}^{\rank(W)}\sigma_j^2(W),~ 
||W||_C:=\max_{i,j=1}^{m,n}|w_{ij}|,$$
\begin{equation}\label{eq12inf}
||W||_h:=
\sup_{{\bf v}\neq 
{\bf 0}}||W{\bf v}||_h/||{\bf v}||_,
~{\rm for}~h=1,2,\infty,
\end{equation}
such that (see \cite[Section 2.3.2 and Corollary 2.3.2]{GL13})  
\begin{equation}\label{eq0}
||W||_C\le ||W||\le ||W||_F\le \sqrt {mn}~||W||_C,~
%{\rm and}~ 
 ||W||_F^2\le
 \min\{m,n\}~||W||^2,
\end{equation}
\begin{equation}\label{eq1inf}
||W||_1=||W^T||_{\infty}=
\max_{j=1,\dots,n}\sum_{i=1}^{m}|w_{ij}|,
\end{equation}
\begin{equation}\label{eq01}
\frac{1}{\sqrt m}||W||_1\le ||W||\le \sqrt {n}~||W||_1,~||W||^2\le ||W||_1||W||_{\infty}.
\end{equation} 
Represent a matrix $W=(w_{i,j})_{i,j}^{m,n}$  as a vector ${\bf w}$ and 
define the  $l_1$-norm $$||W||_{l_1}:=
||{\bf w}||_1=
\sum_{i,j}|w_{i,j}|$$
(cf. \cite{SWZ17}).
Observe that
$||W||_F=||{\bf w}||$, and so
 \begin{equation}\label{eqnrmswz}
 ||W||_F\le ||W||_{l_1}\le \sqrt {mn}~||W||_F. 
  \end{equation}
  
An $m\times n$ matrix $W$ is {\em unitary}  
(also {\em orthogonal} when real)
if $W^*W=I_n$ or $WW^*=I_m$.

$U=W^{(I)}$ is a left inverse of $W$ if $UW=I_n$ (and then $\rank(W)=n$) and its right inverse if $WU=I_m$  (and then 
$\rank(W)=m$). $W^{(I)}=W^{-1}$ if a matrix $W$ is nonsingular.  
\begin{equation}\label{eqsvd}
W=S_W\Sigma_WT_W^*
\end{equation}
denotes its {\em compact SVD}, hereafter referred to just as {\em SVD}, such that  
$$S_W^*S_W=
T_W^*T_W=I_{\rank (W)},~ 
\Sigma_W:=\diag(\sigma_j(W))_{j=1}^{\rank (W)}.$$
$\sigma_j(W)$
denotes the $j$th largest singular value of $W$ for $j=1,\dots,\rank (W)$, 
$$\sigma_j(W)=0~{\rm for}~j>\rank (W),~ 
||W||=\sigma_1(W)~{\rm and}~
||W||_F^2=\sum_{j=1}^{\rank(W)}\sigma_{j}^2(W)$$ (see \cite[Corollary 2.4.3]{GL13}).

$|\cdot|$ can denote spectral or Frobenius norm, depending on context. 

 Write  
\begin{equation}\label{eqnsgmr0}
\tilde\sigma_{r+1}:=\tilde\sigma_{r+1}(W):=\min_{\rank(W')\le r}|W-W'|~{\rm for~any}~r,
\end{equation}
 so that  
\begin{equation}\label{eqnsgmrsp}
\tilde \sigma_{r+1}=\sigma_{r+1}(W)~{\rm under~the~spectral~norm~|\cdot|}
 \end{equation}
(this is the Eckart--Young theorem, see \cite[page 79]{GL13})
and
 \begin{equation}\label{eqnsgmrfr}
(\tilde \sigma_{r+1})^2=\sum_{j=r+1}^{\rank(W)}\sigma_{j}^2(W)\le (\rank(W)-r)~\sigma_{r+1}^2(W)~{\rm under~the~Frobenius~norm~|\cdot|.}
 \end{equation}  

Set to 0 all but the $r$ largest singular values of 
a matrix $W$ and arrive at 
the
 {\em rank-$r$ truncation $\tilde W$ of the matrix} $W$ and at
 its {\em top SVD of rank} $r$,
 given by $\tilde W=\tilde S\tilde \Sigma \tilde T^*$ where
 $\tilde S$, $\tilde \Sigma$, and $\tilde T$ are submatrices
of the matrices $S$, $\Sigma$, and $T$, respectively (see Figure \ref{fig5}).

$W^+:=T_W\Sigma_W^{-1}S_W^*$ is the Moore--Penrose 
pseudo inverse of an $m\times n$ matrix $W$, which is its left inverse  if 
$\rank(W)=n$ and its right inverse if 
$\rank(W)=m$.
If a matrix $W$ has full rank, then
\begin{equation}\label{eqsgm}
||W^+||\sigma_{r}(W)=1.
\end{equation}

 $\kappa(W):=\sigma_1(W)/\sigma_{\rho}(W)=|W|~|W^+|\ge 1$
 denotes its  {\em condition number}. 
 
A matrix $W$ is unitary if and only if $\kappa(W)=1$;
%\item%12
%The  $\xi$-rank of a matrix, for a fixed positive  $\xi$,
%is the minimum rank of its approximations within 
%the norm bound  $\xi$. 
%The   {\em numerical rank} of a matrix is its  $\xi$-rank
%for $\xi$ being small in context. 
it is {\em ill-conditioned} if 
$\kappa(W)$ 
 is large in context, and
%or equivalently if its rank exceeds its numerical rank. 
it is {\em well-conditioned}
if  
$\kappa(W)$ is reasonably bounded.  
It has 
$\epsilon$-{\em rank} at most $r>0$
for a fixed tolerance $\epsilon>0$ if there is a matrix $W'$ 
 of rank $r$ such that 
 $|W'-W|/|W|\le \epsilon$.
 
We write $\nrank(W)=r$ and say that a
matrix $W$ has {\em numerical rank} $r$
if it has $\epsilon$-rank $r$
for a small $\epsilon$.
(A matrix is ill-conditioned
if and only if it has a matrix of a smaller rank nearby. A matrix is well-conditioned
 if and only if its rank is equal to 
its numerical rank.)
 
%------------------------------------------

\subsection{A bound on the norm of pseudo inverse of a matrix product}\label{slmmsng}

%------------------------------------------------------
 
\begin{lemma}\label{lehg}
Let $G\in\mathbb C^{k\times r}$, 
$\Sigma \in\mathbb C^{r\times r}$ and
$H\in\mathbb C^{r\times l}$
and let  the matrices $G$, $H$ and 
$\Sigma$ have full rank 
$r\le \min\{k,l\}$.
Then 
$||(G\Sigma H)^+||
\le ||G^+||~||\Sigma^+||~||H^+||$. 
\end{lemma}
\begin{proof}
Let $G=S_G\Sigma_GT_G$ and $H=S_H\Sigma_HT_H$ be SVDs
where $S_G$, $T_G$, $D_H$, and $T_H$ are unitary matrices,
$\Sigma_G$ and $\Sigma_H$ are the $r\times r$ 
nonsingular diagonal matrices of the singular values,
and $T_G$ and $S_H$ are  $r\times r$ matrices. 
Write $$M:=\Sigma_GT_G\Sigma S_H\Sigma_H.$$ 
Then 
$$M^{-1}=\Sigma_H^{-1} S_H^*\Sigma^{-1}T_G^*\Sigma_G^{-1},$$ and consequently
$$||M^{-1}||\le ||\Sigma_H^{-1}||~|| S_H^*||~||\Sigma^{-1}||~||T_G^*||~||\Sigma_G^{-1}||.$$ 
Hence
$$||M^{-1}||\le||\Sigma_H^{-1}||~
||\Sigma^{-1}||~|\Sigma_G^{-1}||$$ because $S_H$ and $T_G$
are unitary matrices. It follows from 
(\ref{eqsgm}) for $W=M$ that

$$\sigma_r(M)\ge\sigma_r(G)\sigma_r(\Sigma)\sigma_r(H).$$

Now let  $M=S_M\Sigma_MT_M$ be SVD
where $S_M$ and $T_M$ are $r\times r$ unitary matrices. 

Then $S:=S_GS_M$ and $T:=T_MT_H$ are unitary matrices,
and so $G\Sigma H=S\Sigma_M T$ is SVD. 

Therefore 
$\sigma_r(G\Sigma H)=\sigma_r(M)\ge  \sigma_r(G)\sigma_r(\Sigma)\sigma_r(H)$.
Combine this bound with (\ref{eqsgm})
for $W$ standing for $G$, $\Sigma$,
$H$, and $G \Sigma H$.
\end{proof}
 
%------------------------------------------------------------------------------

\subsection{Random and average matrices}\label{srmcs} 

%------------------------------------------------------------------------------
%------------------------------------------------------------------------------
 
Hereafter  
``{\em i.i.d.}" stands for ``independent identically distributed", 
$\mathbb E(v)$ for the expected value of 
a random variable $v$, and
${\rm Var}(v)$ for its variance.

\begin{definition}\label{defgsg} {\em Gaussian and 
${\mathcal {NZ}}$-Gaussian matrices.}

(i) Call an $m\times n$ matrix $W$ {\em  Gaussian} and 
write $W\in \mathcal G^{m\times n}$
if its entries  are  i.i.d. Gaussian variables.   

(ii) Call a matrix $W$
${\mathcal {NZ}}$-{\em Gaussian} and
 write   
$W \in
\mathcal G_{\mathcal{NZ}}^{m\times n}$
if it is filled with zeros except for at most
$\mathcal {NZ}$  entries,  
filled with i.i.d. Gaussian variables.
Call such a matrix {\em nondegenerate}
if it has neither rows nor columns filled with zeros. 
\end{definition}

%------------------------------------------------------------------------------

\begin{assumption}\label{asnndg}
Hereafter we assume by default dealing only with nondegenerating ${\mathcal {NZ}}$- Gaussian matrices.
\end{assumption}
  
%------------------------------------------------------------------------------
 
\begin{lemma}\label{lepr3} ({\rm Orthogonal invariance of a Gaussian matrix.})
Suppose $k$, $m$, and $n$  are three  positive integers,
$G$ is an  
 $m\times n$  Gaussian matrix, 
$S$ and $T$ are $k\times m$ and 
$n\times k$
orthogonal matrices, respectively, and $k\le \min\{m,n\}$.
Then $SG$ and $GT$ are Gaussian matrices.
\end{lemma}
 
%------------------------------------------

\begin{definition}\label{defgsfc} {\em Factor-Gaussian and 
$\mathcal {NZ}$-fac\-tor-Gaus\-sian matrices.}
(i) Call a matrix $W=G\Sigma H$ a
 {\em diagonally scaled factor-Gaussian
matrix of expected rank} $r$ and write $W \in \mathcal G^{m\times n}_r(\Sigma)$
if   $G\in \mathcal G^{m\times r}$, $H\in \mathcal G^{r\times n}$,
   $\Sigma=\diag(\sigma_i)_{i=1}^r$, $\sigma_1\ge \sigma_2\ge\dots\ge\sigma_r> 0$
and unless the ratio $\sigma_1/\sigma_r$ is large (see Figure \ref{fig5}).
 If
$\Sigma=\sigma_1 I_r$  or equivalently if $\sigma_1=\sigma_r$, then 
call the matrix $W$ a
 {\em scaled factor-Gaussian
matrix of  expected rank} $r$
and write 
  $W \in \mathcal G_{r}^{m\times n}$.  
  
\begin{figure}
[h]
\centering
\includegraphics[scale=0.25]{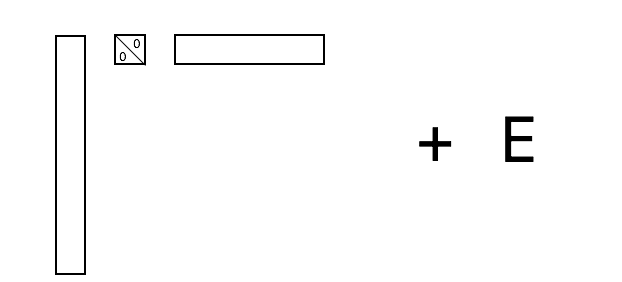}
\caption{The top SVD of a matrix
and a slightly perturbed factor-Gaussian matrix}
\label{fig5}
\end{figure} 

(ii)  Call an $m\times n$ matrix $W=GH$  {\em a left fac\-tor-Gaussian matrix
of expected rank} $r$ and write 
$W \in \mathcal G_{r,H}^{m\times n}$
if $G\in \mathcal G^{m\times r}$, $H\in\mathbb C^{r\times n}$, and $rank(H)=\nrank(H)=r$, that is, the matrix $H$ has full rank and is well-conditioned. 

(iii) Call an   $m\times n$ matrix $W=GH$  a
{\em right fac\-tor-Gaussian matrix of  expected rank} $r$  and write 
$W \in \mathcal G_{G,r}^{m\times n}$
if $G\in \mathbb C^{m\times r}$, 
$H\in\mathcal G^{r\times n}$,
and $rank(G)=\nrank(G)=r$. 
%the matrix $G$ is well-conditioned. 

(iv) Call the matrices $W$ of parts (i)--(iii) 
{\em diagonally scaled, scaled, left, and right $\mathcal {NZ}$-fac\-tor-Gaus\-sian of expected rank} $r$ and write  
$$W \in \mathcal G^{m\times n}_{r,\mathcal {NZ}_G,\mathcal {NZ}_H}(\Sigma),~W \in \mathcal G^{m\times n}_{r,\mathcal {NZ}_G,\mathcal {NZ}_H},~ 
W \in \mathcal G^{m\times n}_{G,r,\mathcal {NZ}_H},~{\rm and}~ 
W \in \mathcal G^{m\times n}_{\mathcal {NZ}_G,r,H},$$
respectively, if they are defined by 
 $\mathcal {NZ}$-Gaussian rather than Gaussian factors $G$ and/or $H$ of expected rank $r$. 
\end{definition}
 
A submatrix of  an $\mathcal {NZ}$-Gaussian 
matrix is $\mathcal {NZ}$-Gaussian, 
and we  readily
verify the following results.
 
%------------------------------------------------------------------------------

\begin{theorem}\label{thsbmgr} 
(i) A submatrix of a  diagonally scaled
 (resp. scaled)  
factor-Gaussian matrix of expected rank 
$r$ is a  diagonally scaled (resp. scaled) 
factor-Gaussian matrix of expected rank  
$r$, (ii) a $k\times n$ (resp. 
$m\times l$) submatrix of 
an $m\times n$  left  (resp. right)
fac\-tor-Gaussian matrix of expected rank $r$ is a left (resp. right)
 fac\-tor-Gaussian matrix of expected rank $r$, and 
 (iii) Similar properties hold for the submatrices of $\mathcal {NZ}$-fac\-tor-Gaus\-sian matrices.
\end{theorem}
  
\begin{definition}\label{deffgsatmst} 
By combining the matrices of expected rank $i$ for $i\le r$ define matrices of expected rank at most $r$. Write
$\mathcal G^{m\times n}_{\le r,G,H}(\Sigma)$,~ $\mathcal G^{m\times n}_{\le r,G,H}$,~ 
 $\mathcal G^{m\times n}_{G,\le r,H}$, and  
$\mathcal G^{m\times n}_{G,\le r,H}$ to denote the classes of  {\em diagonally scaled}, {\em scaled}, {\em left}, and {\em right} {\em fac\-tor-Gaus\-sian matrices of expected rank at most} $r$, respectively. Write
$\mathcal G^{m\times n}_{\le r,\mathcal {NZ}_G,\mathcal {NZ}_H}(\Sigma)$,~ $\mathcal G^{m\times n}_{\le r,\mathcal {NZ}_G,\mathcal {NZ}_H}$,~ 
 $\mathcal G^{m\times n}_{G,\le r,\mathcal {NZ}_H}$, and 
$\mathcal G^{m\times n}_{\mathcal {NZ}_G,\le r,H}$ to
denote the classes of  {\em diagonally scaled}, {\em scaled}, {\em left}, and {\em right} $\mathcal {NZ}$-{\em fac\-tor-Gaus\-sian matrices of expected rank at most} $r$, respectively.  
\end{definition}\label{defatmst}  
An $m\times n$ matrix having a numerical rank
at most  $r$ 
is a small norm perturbation of  the product 
$AB$ of two 
matrices $A\in \mathbb C^{m\times r}$ and  $B\in \mathbb C^{r\times n}$.
Together with  
Definitions \ref{defgsfc} and \ref{deffgsatmst} this motivates the following  definitions 
of  the  $m\times n$ average   matrices
of rank $r\le \min\{m,n\}$.

%------------------------------------------------------------------------------
%------------------------------------------------------------------------------

\begin{definition}\label{deaverg} {\em The average matrices allowing LRA.} 
Define the average $m\times n$ matrices $W$ 
 of a rank $r$ in four ways -- as  the {\em diagonally scaled average,  scaled average,                                                                                                                                                                                                                                                                                                                                    left average}, and {\em right average} -- by  taking
the  average over the matrices of the classes
    $\mathcal G^{m\times n}_{r}(\Sigma)$,
   $\mathcal G^{m\times n}_{r}(\sigma I_r)$,
  $\mathcal G^{m\times n}_{r,H}$, and 
   $\mathcal G^{m\times n}_{G,r}$
       of Definition \ref{defgsfc},
   respectively. Namely fix
   (up to scaling by a constant) all
    non-Gaussian
  matrices $\Sigma$,
  $\sigma I$, $G$, 
  and $H$ involved in the definition
 (see Remark \ref{repr3} below) and
   take the average under the Gaussian probability distribution  over the                                                                                                                                                                                                                                                                                                                                                                                                                                                                                                                                                      
 i.i.d. entries of the 
   Gaussian factors $G$ and/or $H$. Similarly define the four classes of 
   {\em  average  matrices} of rank at most $r$ and the eight  classes of
   $\mathcal {NZ}$-{\em  average  matrices} of rank $r$ and of rank at most $r$. Call perturbations
   of such matrices (within a fixed norm bound $\Delta$)
 the {\em average} and
  $\mathcal {NZ}$-{\em average}
   matrices allowing 
  their  approximations of rank $r$
  or at most $r$, respectively. 
 \end{definition}  
  
 Hereafter we occasionally refer to diagonally scaled, left, and right factor-Gaussian matrices 
 of expected rank $r$ or at most $r$ just as to 
 {\em factor-Gaussian matrices}, dropping 
 some attributes as long as they are clear from context. 
 Likewise we refer to small norm perturbations of factor-Gaussian matrices just as to {\em perturbed factor-Gaussian matrices}, and refer to 
 $\mathcal {NZ}$-{\em factor-Gaussian matrices, average matrices, and 
 $\mathcal {NZ}$-average matrices},
  dropping 
 their  attributes that are clear from context.
 
%------------------------------------------------------------------------------

\begin{definition}\label{defnrm} {\em Norms and expected values} of matrices (see the estimates of Appendix \ref{srnrmcs}).  
Write 
$\nu_{p,q}:=||G||$,
$\nu_{p,q,C}:=||G||_C$,
$\nu_{p,q}^+:=||G^+||$
 and
$\nu_{p,q,C}^+:=||G^+||_C$
for a $p\times q$ Gaussian matrix $G$
and a pair of positive integers 
$p$ and $q$
and notice that
$\nu_{p,q}=\nu_{q,p}$ and 
$\nu_{p,q}^+=\nu_{q,p}^+$.
Write 
$\nu_{p,q,\mathcal {NZ}}:=||G||$ and
$\nu_{p,q,\mathcal {NZ}}^+:=||G^+||$
for a $\mathcal {NZ}$-Gaussian $p\times q$ matrix $G$, for a fixed integer $\mathcal {NZ}$ 
and all pairs of $p$ and $q$. 
\end{definition}

%------------------------------------------------------------------------------

\begin{remark}\label{repr3}  Gaussian matrices tend to have  small norms, but this is no obstacle for using the
matrices of Definitions 
 \ref{defgsfc},  \ref{deffgsatmst}, and \ref{deaverg}
 because we can scale 
the matrices $W=GH$ and $W=G\Sigma H$  
 at will by scaling the non-Gaussian factors $G$,  $H$, or $\Sigma$ involved. 
\end{remark}
 
%------------------------------------------------------------------------------

\section{CUR LRA  and Canonical CUR LRA} \label{sacclb} 
 
%------------------------------------------------------------------------------

\subsection{CUR LRA -- definitions and a necessary   criterion} 
\label{scurdef} 
 
%------------------------------------------------------------------------------  
  
  We first restate some definitions.
For an $m\times n$
  matrix $W$ of numerical rank $r$  fix two integers $k$ and $l$ satisfying 
  (\ref{eqklmnr}),   
fix a tolerance $\Delta$, and seek an LRA $W'$ of $W$ 
satisfying bound (\ref{eqlrk}) and restricted to
the  form of
 {\em CUR} (aka {\em CGR} and 
 {\em Pseudo-skeleton}) {\em approximation}, 
\begin{equation}\label{eqcur0}  
W':=CUR,~W=W'+E,~|E|\le \Delta. 
\end{equation} 
 Here  
 $C$ and $R$ are 
  $m\times l$   
 and $k\times n$ submatrices of the matrix  $W$,
 made up of its $l$  columns and $k$
 rows, respectively,  and 
$U$ is an $l\times k$
matrix, which we call {\em nucleus}  
 of a CUR 
  approximation $W'$. We
  call equation (\ref{eqcur0}) a {\em  CUR decomposition} 
if
\begin{equation}\label{eqcureq}  
 W=W'=CUR. 
\end{equation}
    We call the $k\times l$ submatrix $W_{k,l}$ 
  shared by the matrices $C$ and $R$
   a {\em CUR generator}.
   \begin{theorem}\label{thrnkgen}  
(i)   Equation (\ref{eqcureq})  implies that
   \begin{equation}\label{eqrnkcureq} 
 \rank(C)=\rank(R)=\rank(W)=\rank(W_{k,l}).
  \end{equation}
(ii) For a sufficiently small positive 
$\Delta$
equation (\ref{eqcur0}) implies that
   \begin{equation}\label{eqrnkwcr} 
 \nrank(C)=\nrank(R)=\nrank(W)=\nrank(W_{k,l}).
 \end{equation} 
    \end{theorem}
   \begin{proof}
Let us prove claim (i).  Write $r:=\rank(W)$ and $r':= \rank(W_{k,l})$. 
  
   Deduce from  (\ref{eqcureq}) that
  $\rank(C)=\rank(R)=r\ge r'$.
  
  It remains to show that
  $r'\ge r$
  in order to deduce 
 equation (\ref{eqrnkcureq}). 
  
Apply Gauss--Jordan elimination to the matrix $W_{k,l}$, transforming it into a diagonal matrix $W_{k,l}'$ with $r'$ nonzero entries. 

Extend this elimination to the matrix $W$.
The images $W'$,  $C'$, and  $R'$ 
of the matrices $W$, $C$, and  $R$ 
keep rank $r$, and so 
the matrix $C$ has at least $r$ nonzero columns and the matrix $R$ has 
at least $r$ nonzero rows versus 
$r'$  nonzero columns and rows
of  $W_{k,l}'$. 

It follows that
$r=\rank(W_{k,l}')\ge r'+2(r-r')$,
and so $r'=r$,
which proves claim (i).

Deduce claim (ii) by applying claim (i) to the SVD-truncation for the matrix $W$.
\end{proof} 

The theorem implies that a matrix 
$W_{k,l}$ is a CUR generator
only if  equation (\ref{eqrnkwcr})
holds. In Section \ref{spstr1} we prove that, 
conversely,  a CUR generator $W_{k,l}$   defines 
 a close CUR approximation
 if (\ref{eqrnkwcr}) holds (even when we restrict the CUR approximation to its canonical version
of the next subsection). Hence
(\ref{eqrnkwcr}) is 
{\em a necessary and sufficient CUR criterion}.
  
%------------------------------------------------------------------------------
 
\subsection{Computation of a nucleus. Canonical CUR LRA}\label{scanon}

%------------------------------------------------------------------------------
   
Theorem \ref{thwc} bounds the error norm 
$|W-CUR|$ in terms of the value
$\Delta$ and the norms 
$|C|$, $|R|$, $|\tilde U|$, and  
 $|\tilde U-U|$ where 
 $\max\{|C|,|R|\}\le |W|$. 
  Next we cover generation of the nucleus  $U$  and in Section \ref{spstr1} estimate the norms  
 $|\tilde U-U|$ and  $|\tilde U|$
in terms  of $|U|$ and $\Delta =\tilde \sigma_{r+1}$.
   
    First consider the case of CUR  decomposition
   (\ref{eqcureq}) for $k=l=r$
   (see Figure \ref{fig3}). 
   
\begin{theorem}\label{thdcmp} 
Suppose that $k=l=r$ and the matrices
 $W$, $C=W_{:,\mathcal J}$, 
and $R=W_{\mathcal I,:}$  
have rank $r$ (cf. (\ref{eqrnkcureq})).  Then
 the following two  matrix equations 
 imply one another: 
 \begin{equation}\label{eqCURr}
W=CUR
\end{equation}
and
 \begin{equation}\label{eqnucinv}
U:=W_{r,r}^{-1}.
\end{equation}
 \end{theorem}
 %\end{lemma}
 \begin{proof}
Let (\ref{eqnucinv}) hold and 
without loss of generality let  
$W_{r,r}=W_{\mathcal I,\mathcal J}$ for
$\mathcal I=\mathcal J=\{1,\dots,r\}$
(see (\ref{eq111})).  Then   
$CU=W_{:,\mathcal J}U=(I_r~|~X)^T$ 
for a $r\times (m-r)$ matrix $X$, and so
$$CUR=W_{:,\mathcal J}UW_{\mathcal I,:}=
(I_r~|~X)^TW_{\mathcal I,:}.$$
Hence the matrices $W$ and
$CUR$, both of rank $r$, share their
 first $r$ rows. 
 Likewise 
they share their first $r$ columns,
and hence 
$W=CW_{r,r}^{-1}R=CUR$. 
This implies (\ref{eqCURr}).

Conversely, equation
(\ref{eqCURr}) implies that
$CU=WR^{(I)}$. 
Since $\rank(C)=r$,
this matrix equation
has unique solution $U$, which is
$W_{r,r}^{-1}$, as we can 
readily verify.
\end{proof} 

%------------------------------------------------------------------------------

Next we extend expression (\ref{eqnucinv}) to any
triple of $k$, $l$, and $r$ satisfying
(\ref{eqklmnr}).

Unless $k=l=r$
a nucleus of a CUR decomposition is
not uniquely defined because
 we can add to 
it  any $l\times k$ matrix orthogonal to any or both of 
the factors $C$ and $R$.
Hereafter (except for Section \ref{slrasmp}) we narrow our study to 
 {\em canonical CUR decomposition}
 (and similarly to
 {\em canonical CUR approximation}) by 
generating its nucleus from 
a candidate CUR generator $W_{k,l}$ in two steps as follows.

\begin{algorithm}\label{algcnnncl}
{\em Canonical computation of a nucleus.}
 
%------------------------------------------------------------------------------

\begin{description}

%------------------------------------------------------------------------------

\item[{\sc Input:}]
Two positive integers $k$ and  $l$
 and a $k\times l$ CUR generator
$W_{k,l}$ of unknown numerical rank $r$.

%------------------------------------------------------------------------------

\item[{\sc Output:}]
$r=\nrank (W_{k,l})$ and
an  $l\times k$
nucleus $U$.

%------------------------------------------------------------------------------
%-------------------------------------------------------------------------------

\item[{\sc Computations:}]

\begin{enumerate}
\item
Compute and output numerical rank $r$ of 
the matrix $W_{k,l}$ and its SVD-truncation $W_{k,l,r}$.
 \item
  Compute and output the nucleus 
\begin{equation}\label{eqnuc+svd}
U:=W_{k,l,r}^+.
\end{equation}
\end{enumerate}
%------------------------------------------------------------------------------
\end{description}
%------------------------------------------------------------------------------
\end{algorithm}

%------------------------------------------------------------------------------

 \begin{remark}\label{recnnncl1} 
In this paper we only use Algorithm \ref{algcnnncl} for the computation  
of the nucleus $U$ and can omit its computation of numerical rank $r$.
\end{remark}
   
%------------------------------------------------------------------------------

 \begin{remark}\label{recnnncl}
Algorithm \ref{algcnnncl} involves 
 $O((k+l)kl)$ flops and $O(kl)$
 memory  cells.
\end{remark}

%------------------------------------------------------------------------------

 \begin{theorem}\label{thcandec}
 Suppose that a canonical CUR LRA is defined by  equation
 (\ref{eqnuc+svd})
 and that  equations (\ref{eqrnkcureq}) hold.
 Then equations (\ref{eqcureq}) hold.
  \end{theorem}
 
  \begin{proof}
  Combine Corollary   
 \ref{cowc1} for $\tilde\sigma_{r+1}=0$
 and   equations 
 (\ref{eqrnkcureq}) and (\ref{eqnuc+svd}).
  \end{proof}
  
\begin{remark}\label{reklr_kl} 
 $W_{k,l,r}=W_{k,l}$  if  $r=\min\{k,l\}$;
  (\ref{eqnuc+svd}) turns into 
 (\ref{eqnucinv}) if $k=l=r$.
 \end{remark} 
 
  \begin{remark}\label{reubsc} 
Instead of computing the SVD-truncation we can faster compute  another rank-$r$
approximation of the matrix $W_{k,l,r}$
within a small tolerance 
$\Delta\ge \tilde\sigma_{r+1}$ to the error   norm and then 
replace $\tilde\sigma_{r+1}$ by
 $\Delta$
in our error analysis.
We consistently arrived at
   about the same output accuracy of    CUR LRA when we
computed a nucleus by performing 
SVD-truncation and by
applying a rank-reveal\-ing
LUP  factorization of a generator $W_{k,l}$,
which involved just
  $O(mn\min\{m,n\})$ flops
  (see Section \ref{sprdv_via_v}). 
 \end{remark}

%------------------------------------------------------------------------------

\section{Errors of LRA}\label{sprmcn}

%------------------------------------------------------------------------------     
%------------------------------------------------------------------------------
 
\subsection{Some basic  estimates}\label{scgrbckg}

%------------------------------------------------------------------------------

Our next task is the  
  estimation of the 
 spectral and Frobenius  norms $|E|$ of the error matrix
of a  CUR LRA in terms of the norms of the factors $C$, $U$, and $R$
and of the perturbation of the factor $U$
caused by the perturbation of an  input
matrix $W$ within a norm bound $\Delta$.

  We can only satisfy (\ref{eqlrk})
   for
    $\Delta\ge \tilde \sigma_{r+1}$ (see (\ref{eqnsgmr0})), and we   
 assume that the ratios 
 $\frac{\Delta}{||W||}$,
 $\frac{\tilde \sigma_{r+1}(W)}{||W||}$,
  $\frac{r}{m}$, and
 $\frac{r}{n}$ are small, which is typically the case in various applications of LRA, e.g., to numerical integration,  PDEs, and integral equations.
 
%------------------------------------------------------------------------------
    
 Next 
assume that a rank-$r$ matrix 
$\tilde W$ approximates a matrix $W$ 
of numerical rank $r$ within a fixed small tolerance $\Delta$ and let 
$\tilde W=\tilde C\tilde U\tilde R$
be any canonical 
CUR decomposition 
(cf. Theorem \ref{thcandec}). Then clearly
$|W-CUR|\le \Delta + |\tilde C\tilde U\tilde R-CUR|$, for a fixed spectral or Frobenius matrix norm $|\cdot|$. 
Here $\Delta$ takes on its minimum 
value $\tilde \sigma_{r+1}$
for the matrix $\tilde W$
being a rank-$r$ truncation of matrix $W$.

Our  remaining goal is the estimation of the norm $|\tilde C\tilde U\tilde R-CUR|$.

%------------------------------------------------------------------------------
   
\begin{theorem}\label{thwc} 
Fix the spectral or Frobenius matrix norm $|\cdot|$, five integers  
$k$, $l$, $m$, $n$, and $r$ satisfying (\ref{eqklmnr}),  an  $m\times n$ matrix $W$,  its rank-$r$ approximation $\tilde W$ within a norm bound   
 $\Delta$ 
not exceeded by the value $\tilde \sigma_{r+1}$  of equations (\ref{eqnsgmr0})--(\ref{eqnsgmrfr}),  
two index sets 
 $\mathcal I$ and $\mathcal J$,
   and  $l\times k$ matrices
 $U$
 and $\tilde U$. Write
 \begin{equation}\label{eqsvd1}
C:=W_{:,\mathcal J},~
R:=W_{\mathcal I,:},
~\tilde C:=\tilde W_{:,\mathcal J},
~{\rm and}~ 
\tilde R:=\tilde W_{\mathcal I,:}
\end{equation}
 Then  
\begin{equation}\label{eqnrm1}
|\tilde C\tilde U\tilde R-CUR|\le 
(|R|+|C|+\Delta)~|\tilde U|~\Delta+  
|C| ~|R|~|U-\tilde U|+\Delta.
\end{equation}
\end{theorem}
\begin{proof}
Notice that   
$$\tilde C\tilde U\tilde R-CUR=(\tilde C-C) \tilde U\tilde R+
C \tilde U(\tilde R-R)+ 
C(\tilde U-U)R.$$
%\begin{equation}\label{eqcur'}
%\end{equation}
Therefore
$$|\tilde W-CUR|\le
|\tilde C-C|~|\tilde U|~|\tilde R|+  
|C|~|\tilde U|~|\tilde R-R|+
|C|~|\tilde U-U|~|R|.$$

Substitute the bound 
$\max\{|C-\tilde C|,|R-\tilde R|\}\le
|W-\tilde W|=\Delta$ 
and  obtain  (\ref{eqnrm1}). 
\end{proof}

%------------------------------------------

\subsection{Errors of  a canonical CUR LRA}\label{spstr1}

%------------------------------------------------------------------------------
 
 Next we extend  Theorem \ref{thwc} to estimating the norms $|W-CUR|$
 and  $|\tilde U-U|$ for
a canonical CUR LRA.
%and consider non-canonical ones in %Section \ref{snncncl}.
 We begin with two lemmas.
  
\begin{lemma}\label{leprtpsdinv} 
 Under the assumptions of 
Theorem \ref{thwc}, write 
 $W_{k,l}=W_{\mathcal I,\mathcal J}$,
 define the SVD-truncation $W_{k,l,r}$ of
 the matrix $W_{k,l}$, and
 let
$U:=W_{k,l,r}^+$ and
$\tilde U:=\tilde W_{k,l}^+$
(cf. (\ref{eqnuc+svd})).
Then 
$$|\tilde U-U|\le 
\mu|\tilde U|~|U|~ \Delta_{k,l,r}.$$
Here  
 $\Delta_{k,l,r}=||\tilde W_{k,l} -W_{k,l,r}||\le  
  \eta\tilde\sigma_{r+1}$, 
  $\eta=1$ if $r=\min\{k,l\}$ 
  and $\eta=2$
  if $r<\min\{k,l\}$. Furthermore
$\mu=1$ for $|\cdot|=||\cdot||_F$, 
 
$\mu=(1+\sqrt 5)/2$ 
if $|\cdot|=||\cdot||$
 and if
$r<\min\{k,l\}$, 

 $\mu=\sqrt 2$ if $|\cdot|=||\cdot||$
 and if $r=\min\{k,l\}$.
\end{lemma}

\begin{proof}
Notice that $\rank(\tilde W_{k,l})=\rank(W_{k,l,r})=r$ and apply
 \cite[Theorem 2.2.5]{B15}.
\end{proof}

\begin{lemma}\label{leprtinv}
Under the assumptions and definitions of 
Lemma \ref{leprtpsdinv}, 
write $\theta:=||U||~\Delta_{k,l,r}$ and let 
$\theta<1$.
Then 
 $||\tilde U||\le \frac{1}{1-\theta}||U||$.
\end{lemma}

\begin{proof} 
This is  \cite[Theorem 2.2.4]{B15}.
\end{proof}
 Next assume that $\tilde W$ is a rank-$r$ truncation of $W$, so that
$\Delta=\tilde \sigma_{r+1}$, 
 combine the two lemmas with Theorem \ref{thwc}, and estimate the output errors of a CUR LRA solely in terms  of the tolerance $\Delta=\tilde \sigma_{r+1}$ and the norms of the factors $C$,  $U$, and $R$. 
 
 \begin{corollary}\label{cowc1}
(i) Under the assumptions and definitions of Lemma \ref{leprtpsdinv}, 
it holds that
\begin{equation}\label{eqnrm10}
\frac{1}{\tilde \sigma_{r+1}}|W-CUR|\le 
(|R|+|C|+\tilde \sigma_{r+1}+
\mu\eta
|C| ~|R|~|U|)|\tilde U|+1.
\end{equation}
(ii) Under the assumptions and definitions of Lemma \ref{leprtinv}, 
let $(1-\theta)\xi =1$ for 
$|\cdot|=||\cdot||$
and  $(1-\theta)\xi =\sqrt{kl}$ for 
$|\cdot|=||\cdot||_F$. Then
\begin{equation}\label{eqnrm11}
\frac{1}{\tilde \sigma_{r+1}}|W-CUR|\le 
(|R|+|C|+\tilde \sigma_{r+1}+
\mu\eta 
|C| ~|R|~|U|)\xi|U|+1,
\end{equation}
and so   
\begin{equation}\label{eqnrmO}
|W-CUR|=
O((\theta+1)\theta \tilde \sigma_{r+1})
~{\rm for}~\theta=|W|~|U|.
\end{equation}
\end{corollary}

%------------------------------------------------------------------------------
   
\begin{remark}\label{reerrdlt} 
For an input matrix $W$ being a perturbed $\pm\delta$-matrix and for $r>1$, the 
value $\tilde \sigma_{r+1}$ is small,
but also the value $1/|U|$ is small unless 
a CUR generator includes the 
$(i,j)$th  entry at which 
$|w_{ij}|\approx 1$.
\end{remark}

%-------------------------------------------------------------------------------
%------------------------------------------

\subsection{Superfast a posteriori error  
  estimation for a  CUR LRA
 and sufficiency of the first CUR criterion}\label{scurhrd}

%------------------------------------------------------------------------------
 
Based on Corollary \ref{cowc1} we can obtain {\em superfast a posteriori error  
  estimation for a
 CUR LRA} under (\ref{eqrnkwcr}). (Assumptions like  (\ref{eqrnkwcr}) are necessary
% (cf. randomized algorithms
%of \cite[Sections 4.3 and 4,4]{HMT11} for
 % a posteriori error estimation). 
 in order to counter the problems arising  for the hard inputs, such as those of
 Examples \ref{exdlt}  and \ref{exdltdns} of Section \ref{smtv}.)
  
Given an $m\times n$ matrix $W$, its CUR factors $U\in \mathbb C^{l\times k}$ of 
(\ref{eqnuc+svd}),
  $C\in \mathbb C^{m\times l}$ and  
  $R\in \mathbb C^{k\times n}$,  and an upper bound on the value  $\tilde \sigma_{r+1}$ of
 equations 
 (\ref{eqnsgmr0})--(\ref{eqnsgmrfr}),
 we can readily estimate at first the norms 
 of these factors and then the norm of the matrix 
 $W-CUR$ of Corollary \ref{cowc1}.
 Even if we are not given the factor $U$, we can compute it  superfast
 (see Remark \ref{recnnncl}). 
 
%------------------------------------------------------------------------------

The error norm bounds of the 
corollary
are proportional to
$\tilde \sigma_{r+1}$ and are converging to 0 as 
$(\theta+1))\theta\sigma_{r+1} 
\rightarrow 0$. 
By combining these observations  with Theorem \ref{thrnkgen} we deduce that  the {\em first CUR criterion is necessary and sufficient} for computing accurate CUR LRA.  

\begin{corollary}\label{cocrtr1}
A submatrix  $W_{k,l}$ of a matrix $W$
is a generator of a close CUR
 LRA  of a matrix $W$ if and only if $\nrank(W_{k,l})=\nrank(W)$. 
\end{corollary}
%\begin{remark}\label{realgcvrt} 
%A submatrix $W_{k,l}$ satisfies the first %CUR criterion unless this submatrix is a %small norm perturbation of a matrix of %rank less than $\nrank(W)$, that is, of a %matrix from an algebraic variety of a %smaller dimension in the variety of %matrices of rank 
%$\nrank(W)$.  
%\end{remark} 
 
%-----------------------------------------------------

\subsection{Superfast a posteriori error  estimation for LRA of a matrix filled with i.i.d. values of a single variable}\label{spstr}

%-----------------------------------------------------
   
  The above error  estimates     
   involve the norm of the matrices $C$, $U$, and $R$ and the smallest
 error norm of  rank-$r$ approximation.
% \begin{remark}\label{reestnrk}
% The estimates enable us to
%  bound the numerical rank of an input %matrix. 
% \end{remark}
 In our distinct  superfast randomized a
 posteriori error estimation below
 we do not assume that we 
 have this information, but suppose that
 the error matrix  $E$ of an LRA 
  has enough  entries, say,  100 or more, and that they are the 
observed i.i.d. values of a single random variable. This is realistic, for example, where the deviation 
of the matrix $W$ from  its rank-$r$ approximation is due to the errors of
 measurement or rounding.

 In this case the Central Limit Theorem implies that the distribution of the variable is close to Gaussian
(see \cite{EW07}). 
Fix a pair of integers $q$ and $s$
such that $qs$ is large enough (say, exceeds 100), but $qs=O((m+n)kl)$ 
and hence $qs\ll mn$; then  
 apply our tests just to
a random $q\times s$ submatrix
of the $m\times n$ error matrix.

  Under this policy we compute the error matrix 
 at a dominated arithmetic cost in $O((m+n)kl)$ 
but still verify correctness  
 with high confidence, 
by applying the customary rules of {\em hypothesis 
testing for the variance of a Gaussian variable.}

%------------------------------------------------------------------------------

Namely suppose that we have observed the values  
$g_1,\dots,g_K$
of a Gaussian random variable $g$ with a mean value $\mu$
and a variance $\sigma^2$ and that we have computed 
 the observed average value and variance 
 $$\mu_K=\frac{1}{K}\sum_{i=1}^K |g_i|~
{\rm and }~\sigma_K^2=\frac{1}{K}\sum_{i=1}^K |g_i-\mu_K|^2,$$
respectively.
Then, for a fixed reasonably large $K$,
both 
$${\rm Probability}~\{|\mu_K-\mu|\ge t|\mu|\}~{\rm and~Probability}\{|\sigma_K^2-\sigma^2|\ge t\sigma^2\}$$ 
converge to 0 exponentially fast as $t$ grows to the infinity
(see \cite{C46}).

%------------------------------------------------------------------------------
  
\section{Accuracy of Superfast CUR LRAs of Random and Average  
 Matrices: Direct Error Estimation}\label{serralg1} 
 
%------------------------------------------
 
Next we estimate the norms of the matrices $C$, $R$, and 
$U=W^+_{k,l}$ of (\ref{eqnuc+svd})
for a perturbed 
 diagonally scaled  factor-Gaussian matrix $W$ of Definition
\ref{defgsfc} and then substitute our
 probabilistic estimates into Corollary \ref{cowc1}. 
 %By virtue of Theorem  \ref{thrnkgen}, %the computation fails at the stage of %computing the nucleus $U=W^+_{k,l}$ of  (\ref{eqnuc+svd})
% unless a submatrix $W_{k,l}$
% has numerical rank $r$;
%such a failure, however, occurs with a low probability
% (see Theorem \ref{thsiguna}). 
 Moreover we arrive at a reasonably small 
  probabilistic upper bounds
 on the error norm $||W-CUR||$ 
 {\em for any choice} of a $k\times l$ CUR generator
 unless the integer $(k-r)(l-r)$ is small.
 We estimate just the spectral norm 
 $||W-CUR||$; bounds (\ref{eq0})
 enable  extension to the estimates for
 the Frobenius norm  $||W-CUR||_F$.
In this and the next sections we only  cover matrices of an expected rank $r$ and  omit the straightforward extension to the  matrices of an expected rank at most $r$.
 
Our results can be readily extended to the average matrices allowing LRA 
(see Definition \ref{deaverg}) and also to
the perturbed  left (resp. right) factor-Gaussian inputs
provided that $l=n$
(resp.  $k=m$).    
 We need this provision  because  
 the upper bounds $g_+\ge \sigma_r(G)$ and $h_+\ge \sigma_r(H)$
 for  non-Gaussian factors $G$
 and $H$  in parts (ii) and (iii) of 
 Definition \ref{defgsfc} do not
 generally 
 hold for the submatrices of these factors.                                                                                                                                                                                                                                                                                                                                                                                                                                                                                                                                                                                                                                                                                                                                                                                                                                                                                                                                                                                                                                                                                                                                                                                                                                                                                                                                                                                                                                                                                                                                                                                                                                                                                                                                                                                                                                                                                                                                                                                                                                                                                                                                                                                                                                                                                                                                                                                                                                                                                                                                                                                                                                                                                                                                                                                                                                                                                                                                                                                                                                                                                                                                                                                                                                                                                                                                                                                                                                                                                                                                                                                                                                                                                                                                                                                                                                                                                                                                                                                                                                                                                                                                                                                                                     
 We relax this assumption in Section \ref{svlmprvlrsp} for 
 LRAs output by C-A iterations.
   
  We deduce our estimates for a rank-$r$
  input by means  of bounding the condition number of a CUR generator. If this number is nicely bounded we can
  readily extend the estimates to nearby  input matrices, having numerical rank $r$. 
   
%------------------------------------------

\subsection{The norm bounds for CUR factors in a CUR decomposition of a  perturbed  factor-Gaussian matrix}\label{savdsc}

%------------------------------------------------------------------------------
 
  Consider a CUR decomposition of a
  diagonally scaled factor-Gaussian matrix $W=CUR$ having an expected rank $r$,
   a generator $W_{k,l}$, and  a 
  nucleus $U=W_{k,l,r}^+$ of  (\ref{eqnuc+svd}).                                                                                                                                                                                                                                                                                   
  By virtue of Theorem \ref{thsbmgr} 
  the submatrices $C$, $R$, and
  $W_{k,l}$ are also diagonally scaled factor-Gaussian matrices having an expected rank $r$.   
  By virtue of Theorem \ref{thrnd} they have full rank $r$ with probability 1, 
  and in
   our probabilistic analysis
  we assume that these matrices do have full rank $r$. 
    
  Next observe that
 \begin{equation}\label{eqnrcr} 
||C||\le\nu_{m,r}\nu_{r,r}\sigma_1~{\rm and}~
||R||\le\nu_{r,r}\nu_{r,n}\sigma_1,
\end{equation}
for $\sigma_1$ of Definition
\ref{defgsfc}  and $\nu_{p,q}$ denoting the random variables 
of Definition \ref{defnrm}, whose probability distributions and expected values have been estimated in 
  Theorem \ref{thsignorm}. In particular  obtain the following 
   upper bounds on the
  expected values, 
  \begin{align}\label{eqexpv}
\frac{\mathbb E(||C||)}{\sigma_1}\le \mathbb E(\nu_{m,r})\mathbb E(\nu_{r,r})\le 2(\sqrt {mr}+ r),~ 
\frac{\mathbb E(||R||)}{\sigma_1}\le E(\nu_{r,r})E(\nu_{r,n})\le 2(\sqrt {nr}+ r).
\end{align}
  
Recall that $W_{k,l,r}=W_{k,l}$ 
for $W_{k,l}\in \mathcal G^{k\times l}_{r}$, and so $U=W_{k,l}^+$ (cf. Remark 
\ref{reubsc}). 
%recall bound (\ref{eqsgmr}), and
 %estimate the norm
%\begin{equation}\label{eqij} 
%||U||\le  \beta||W_{k,l}^+||,~{\rm for}~%%\beta:= \sqrt{((k-r)r+1)((l-r)r+1)}.
%\end{equation}    
 Complete our estimates  with
  the following theorem and corollary, to be proven in Section \ref{sprffcg}. 
  
  Hereafter we write  
 $e:=2.7182828\dots$, and so  $e^2<7.4$.  
 
\begin{theorem}\label{thbscsb}
(i) Let $W'$ be an $m\times n$ diagonally scaled factor-Gaussian matrix with an expected rank $r$ (this covers a scaled factor-Gaussian matrix as a special case). 
Let $W_{k,l}$ 
be its $k\times l$ submatrix  for $\min\{k,l\}>r$ and define $U=W_{k,l,r}^+$ 
(cf. (\ref{eqnuc+svd})). Then  
\begin{equation}\label{eqnrm+}
\mathbb E(||U||)\le
\frac{e^2\sqrt {kl}}{(k-r)(l-r)\sigma_r}<\frac{7.4~\sqrt {kl}}{(k-r)(l-r)\sigma_r},
\end{equation}
for  $\sigma_r$ of Definition
\ref{defgsfc}.  
Furthermore, for any fixed $\zeta>1$,
\begin{equation}\label{eqcnv}
{\rm Probability}\Big \{||U||>
\zeta\frac{e^2\sqrt {kl}}{(k-r)(l-r)\sigma_r}\Big \}
\end{equation}  
converges to 0 exponentially fast as $\min\{k-r,l-r\}$ grows to the infinity.

(ii) Consider an $m\times n$ left  factor-Gaussian matrix $W$ with an expected rank $r$ (see Definition \ref{defgsfc}). Let $W_{\mathcal I,:}:=W_{k,n}$
 denote its $k\times n$
 submatrix. Then
\begin{equation}\label{eqnrm+l}
\mathbb E(||U||)=
\mathbb E(||G_{\mathcal I,:}^+||~||H^+||)=
\mathbb E(||G_{\mathcal I,:}^+||)~h_+=\mathbb E(\nu_{k,r}^+)h_+\le
\frac{e\sqrt {k}~h_+}{k-r}<\frac{2.72\sqrt {k}~h_+}{k-r}.
\end{equation}

(iii)  Consider an $m\times n$ right factor-Gaussian matrix $W'$ with an expected rank $r$ (see Definition \ref{defgsfc}). Let $W_{:,\mathcal J'}:=W_{m,l}'$ denote its $m\times l$ submatrix. Then
\begin{equation}\label{eqnrm+r}
\mathbb E(||U||)= 
\mathbb E(||G^+||~||H_{:,\mathcal J}^+||)=
g_+\mathbb E(||H_{:,\mathcal J}^+||)=g_+\mathbb E(\nu_{l,r}^+)\le
\frac{e\sqrt {l}~g_+}{l-r}<\frac{2.72\sqrt {l}~g_+}{l-r}.
\end{equation} 

%------------------------------------------------------------------------------

\end{theorem}
 
\begin{corollary}\label{cobscsb}  
For the average matrices $W$ of Definition
\ref{deaverg}   
upper bound  (\ref{eqnrm11}) on the ratio 
$||W-CUR||/\tilde\sigma_{r+1}$ turns into
$$((4(\sqrt{mr}+r)(\sqrt{nr}+r)\sigma_1^2\mu\eta\mathbb E(||U||)+4r+
2\sqrt{mr}+2\sqrt{nr})\sigma_1+\tilde\sigma_{r+1})\xi
 \mathbb E(||U||)+1$$
 for $\mu$, $\eta$, $\xi$, and $\sigma_1$
 of Theorem \ref{thbscsb} and
  $\mathbb E(||U||)$
 bounded according to 
 (\ref{eqnrm+}), (\ref{eqnrm+l}),
and (\ref{eqnrm+r}).
\end{corollary} 
\begin{remark}\label{resprs} 
Based on Theorem \ref{thNZnrm+} we can extend the estimates of the corollary to a little weaker bounds in the case of sparse 
factor-Gaussian matrices (see Remark \ref{reNZnrm}).
\end{remark}

%------------------------------------------

\subsection{The proof of 
the norm bounds for CUR factors}\label{sprffcg}
     
%------------------------------------------------------
%------------------------------------------------------------------------------
     
 Since 
$W'$ is a diagonally scaled factor-Gaussian   
matrix with an expected rank $r$, write
$$W'=G\Sigma H$$ where
$G\in \mathcal G^{m\times r},~
\Sigma \in\mathbb C^{r\times r},~
~H\in \mathcal G^{r\times n}$, 
and 
$||\Sigma^+||=1/\sigma_r$
for $\sigma_r$ of Definition \ref{defgsfc}.

It follows that  
%\begin{equation}\label{eqghw}
$W_{\mathcal I, \mathcal J}'=G_{\mathcal I,:}\Sigma H_{:,\mathcal J}$,
$~{\rm for}~G_{\mathcal I,:}\in \mathcal G^{k\times r}~
{\rm and}~H_{:,\mathcal J}\in \mathcal G^{r\times l}.$ 
%\end{equation} 

%------------------------------------------------------------------------------
  
Hereafter 
\begin{equation}\label{eqgmma}
\Gamma(x):=
\int_0^{\infty}\exp(-t)t^{x-1}dt
\end{equation}
denotes the {\em Gamma function}.  Apply 
Theorem \ref{thsiguna} 
and obtain  that
\begin{equation}\label{eqprbnrmg}
{\rm Probability}\{||G_{\mathcal I,:}^+||\ge x\}\le 
\frac{(x/k)^{-0.5(k-r+1)}}{\Gamma(k-r+2)},
\end{equation}
\begin{equation}\label{eqprbnrmh}
{\rm Probability}\{||H_{:,\mathcal J}^+||\ge x\}\le 
\frac{(x/k)^{-0.5(l-r+1)}}{\Gamma(l-r+2)},\end{equation}
 $$\mathbb E(||G_{\mathcal I,:}^+||)=\mathbb E(\nu_{k,r}^+)\le \frac{e\sqrt {k}}{k-r},~{\rm and}~
\mathbb E(||H_{:,\mathcal J}^+||)=\mathbb E(\nu_{r,l}^+)\le \frac{e\sqrt l}{l-r}.$$ 

Recall that the matrix $\Sigma$ is nonsingular and the matrices $G$ and $H$ have full rank  with probability 1 (see Theorem \ref{thrnd}), apply   
Lemma \ref{lehg} to the matrix $W=G\Sigma H$, and obtain
$$
||W_{\mathcal I,\mathcal J}^+||
=||(G_{\mathcal I,:}\Sigma H_{:,\mathcal J})^+||\le ||G_{\mathcal I,:}^+||~||H_{:,\mathcal J}^+||~||\Sigma^+||.$$
Substitute equation $||\Sigma^+||=1/\sigma_r$
and bounds (\ref{eqprbnrmg}) and 
(\ref{eqprbnrmh})
and obtain bound (\ref{eqnrm+}).

Similarly complete the derivation of 
bounds (\ref{eqcnv})--(\ref{eqnrm+r}). 
 
%------------------------------------------------------------------------------
  
\subsection{CUR LRA of perturbed left, right, and sparse factor-Gaussian
 matrices}\label{svlmprvlrsp}

%------------------------------------------------------------------------------

Next estimate the norms
$||W_{\mathcal I,\mathcal J}^+||$ of some submatrices $W_{\mathcal I,\mathcal J}$ of 
 left and right factor-Gaussian
and  $\mathcal {NZ}$-factor-Gaussian  matrices  $W=GH$. 

In Section \ref{slmmsng} we reduced this task to the estimation of the norms  $||G_{\mathcal I,:}^+||$ and $||H_{:,\mathcal J}^+||$
of some submatrices $G_{\mathcal I,:}$ and $H_{:,\mathcal J}$
of the factors $G$ and $H$, respectively.
Now we are going to bound those norms by
extending 
 the  upper estimates 
 for the 
 norms $||G^+||$ and $||H^+||$ implied by 
 Theorems \ref{thsiguna} and
 \ref{thNZnrm+}. We apply these theorems 
 assuming  that the submatrices
 $G_{\mathcal I,:}$ and $H_{:,\mathcal J}$
 have been output in two successive C-A steps that use the algorithms of 
 \cite{GE96} as the subalgorithms.
 In this case  
 $||G_{\mathcal I,:}^+||\le 
 t_{n,r,h}~||G_+||$ and
 $||H_{:,\mathcal J}^+||\le 
 t_{m,r,h}~||H_+||$ for 
 $t_{p,r,h}^2=(p-r)rh^2+1$ of (\ref{eqtkrh}) and any fixed $h>1$,
 by virtue of Corollary \ref{cocndprt}.
 
 Here and hereafter  
 $t_{i,j,h}^d$ denotes $(t_{i,j,h}^d)$
 for all $i$, $j$, $h$, and $d$.
 
  The bounds on 
  the norms $||G^+||$ and $||H^+||$
   increase by the factors
 $t_{n,r,h}$ and $t_{m,r,h}$ in the transition to the 
 bounds on the norms
  $||G_{\mathcal I,:}^+||$
  and $||H_{:,\mathcal J}^+||$, 
  but
   we can bound the factors
 $t_{n,r,h}$ and $t_{m,r,h}$ by constants
 by applying the recipe of Remark
   \ref{recurgge}.
  Combine these bounds with those of (\ref{eqnrcr}) and (\ref{eqexpv}) and with Corollary \ref{cowc1}  
  and arrive at the following result. 
  
    \begin{theorem}\label{thc-aleft_right} 
Perform  two 
successive steps of a C-A algorithm with subalgorithms from the paper 
\cite{GE96}.  Use the output of the second step 
as a CUR generator and build on it 
a CUR LRA of an input matrix $W$.
 Then this approximation is accurate whp
 as long as the input matrix $W$ is 
 a perturbed left, right, or  
 $\mathcal {NZ}$-factor-Gaussian matrix  or the average of such matrices. 
\end{theorem}
 
%------------------------------------------------------------------------------
  
\section{Superfast Transition from an LRA to a CUR LRA}\label{slracur}

%------------------------------------------
 
Next we devise a superfast algorithm
based on Corollary \ref{cowc1} for
the transition from any LRA of (\ref{eqlrk}) to CUR LRA. 

%------------------------------------------------------------------------------
  
\subsection{Superfast transition from an LRA to top SVD}\label{slrasvd}

%------------------------------------------
 
 We begin with an auxiliary transition of independent interest from an LRA of a matrix to its top SVD
(see again Figure \ref{fig5}); moreover the algorithm works for a little more general LRA than that of (\ref{eqlrk}). 
    
\begin{algorithm}\label{alglratpsvd}
{\em Superfast transition from an LRA to top SVD.}
 
%------------------------------------------------------------------------------

\begin{description}

%------------------------------------------------------------------------------

\item[{\sc Input:}]
Three matrices 
$A\in \mathbb C^{m\times l}$, 
$V\in \mathbb C^{l\times k}$, 
$B\in \mathbb C^{k\times n}$, and  
$W\in \mathbb C^{m\times n}$
such that 
$$W=AVB+E,~||E||= O(\tilde\sigma_{r+1}),~
r\le\min\{k,l\},~k\ll m,~{\rm and}~l\ll n~
{\rm for}~\tilde\sigma_{r+1}~{\rm of~ (\ref{eqnsgmr0})}.$$
(This turns into ref{eqlrk}) for $k=l$ and $V=I-n$.)
%------------------------------------------------------------------------------

\item[{\sc Output:}]
Three matrices  
$S\in \mathbb C^{m\times r}$ (unitary),  
$\Sigma\in \mathbb C^{r\times r}$ 
(diagonal), and  
$T^*\in \mathbb C^{r\times n}$  (unitary)
such that $W=S\Sigma T^*+E'$ for 
$||E'||= O(\tilde\sigma_{r+1})$.                                                                                                                                                    

%------------------------------------------------------------------------------
%-------------------------------------------------------------------------------

\item[{\sc Computations:}]
\begin{enumerate}
\item
 Compute QRP rank-revealing factorization  of the matrices $A$ and $B$:
$$A=(Q~|~E_{m,l-r})RP~{\rm and}~B=P'R'
 \begin{pmatrix}Q'\\E_{k-r,n}\end{pmatrix}
 $$ where  $Q\in \mathbb C^{m\times r}$,
 $Q'\in \mathbb C^{r\times n}$ and 
 $||E_{m,l-r}||+ 
 ||E_{k-r,n}||=O(\tilde\sigma_{r+1})$. Substitute the expressions for $A$ and $B$
 into the matrix equation
$W=AVB+E$ and obtain
$W=                                                                           QUQ'+E'$ where 
$U=RPVP'R'\in \mathbb C^{r\times r}$ and
 $||E'||=O(\tilde\sigma_{r+1})$. 
 \item
Compute 
 SVD $U=\bar S\Sigma \bar                                                                                    T^*$.
Output the $r\times r$  diagonal matrix $\Sigma$.
 \item
Compute and output the unitary matrices
$S=Q\bar S$ and  $T^*=\bar T^                                                                                                                                                                                                                                      *Q'$.
\end{enumerate}
%------------------------------------------------------------------------------
\end{description}
%------------------------------------------------------------------------------
\end{algorithm}

%------------------------------------------------------------------------------

This  algorithm 
uses $ml+lk+kn$ memory cells and $O(ml^2+nk^2)$ flops. 
 
%------------------------------------------------------------------------------
  
\subsection{Superfast transition from top SVD to a CUR LRA}\label{ssvdcur}
 
%------------------------------------------
   
Given top SVD of a matrix allowing LRA we are going to compute its CUR LRA  superfast.

%------------------------------------------------------------------------------
    
\begin{algorithm}\label{algsvdtocur}
{\em Superfast transition from  top SVD to CUR LRA.}
 
%------------------------------------------------------------------------------

\begin{description}

%------------------------------------------------------------------------------

\item[{\sc Input:}]
Five integers $k$, $l$, $m$, $n$, and $r$ satisfying (\ref{eqklmnr}) and matrices  
$W\in \mathbb C^{m\times n}$,
$\Sigma\in \mathbb C^{r\times r}$ 
(diagonal),
$S\in \mathbb C^{m\times r}$, and  
$T\in \mathbb C^{n\times r}$  (both unitary)
such that $1\le r\ll\min\{m,n\}$, 
$W=S\Sigma T^*+E$, and 
$||E||= O(\tilde\sigma_{r+1})$ for
$\tilde\sigma_{r+1}$ of
(\ref{eqnsgmr0}).                                                                                                                                                    

%------------------------------------------------------------------------------

\item[{\sc Output:}]
Three matrices 
$C\in \mathbb C^{m\times l}$, 
$U\in \mathbb C^{l\times k}$, and
$R\in \mathbb C^{k\times n}$   
such that 
$$W=CUR+E'~{\rm and}~||E'||= O(\tilde\sigma_{r+1}).$$

%------------------------------------------------------------------------------

\item[{\sc Computations:}]
\begin{enumerate}
\item
By applying the algorithms of \cite{GE96} or  \cite{P00} to the matrices $S$ and $T$
compute their  $k\times r$
and $r\times l$ submatrices
$S_{\mathcal I,:}$
 and $T^*_{:,\mathcal J}$,
respectively.
Output the CUR factors 
$C=W_{:,\mathcal J}=
S\Sigma T^*_{:,\mathcal J}$
and $R=W_{\mathcal I,:}=
S_{\mathcal I,:}\Sigma T^*$.
\item 
Compute and output a nucleus 
$U=W_{k,l,r}^+=W_{k,l}^+$
for the CUR generator  

$W_{k,l}=W_{\mathcal I,\mathcal J}=
S_{\mathcal I,:}\Sigma T^*_{:,\mathcal J}$.
\end{enumerate}
%------------------------------------------------------------------------------
\end{description}
%------------------------------------------------------------------------------
\end{algorithm}

%------------------------------------------------------------------------------

The algorithm uses $kn+lm+kl$ memory cells and 
$O(ml^2+nk^2)$ flops, and so it is superfast for $k\ll m$ and $l\ll n$.

Clearly $||W_{k,l}||\le ||W||$.
Recall that 
$||S_{\mathcal I,:}^+||\le t_{m,l,h}$ and
$||T_{:,\mathcal J}^{*+}||\le t_{n,k,h}$
by virtue of Theorem \ref{thcndprt}
(we can choose $h\approx 1$, say h=1.1)
and that 
$||\Sigma^+||=||W^+||=1/\sigma_r(W)$. Combine these bounds with
  Lemma \ref{lehg} and deduce that

%------------------------------------------------------------------------------
    
\begin{equation}\label{eqsvdtocur}
||U||=||W_{k,l}^+||\le 
||S_{\mathcal I,:}^+||~||\Sigma^+||~|| T_{:,\mathcal J}^{*+}||\le
t_{m,l,h}t_{n,k,h}/\sigma_r(W). 
\end{equation}
Now Corollary \ref{cocrtr1} bounds 
the output errors  of the algorithm, 
which shows its correctness.

 Hereafter we refer to the following variant of  Algorithm \ref{algsvdtocur}
as {\bf  Algorithm \ref{algsvdtocur}a}:
choose two integers $k$ and $l$
not exceeded by $r$
%such that $\min\{k,l\} \ge r$ 
and 
instead of deterministic algorithms of 
 \cite{GE96} or  \cite{P00}
apply randomized Algorithm \ref{algsmplexp} with leverage scores computed from equations (\ref{eqsmpl}).

In this variation
 the algorithm 
only involves $mk+nl+kl$ memory cells
and
 $O(mk+nl+(k+l)kl)$ flops;
 moreover the upper bounds on the norms  
$||S_{\mathcal I,:}^+||$  and 
$||T_{:,\mathcal J,}^{*+}||$ 
decrease. 

Indeed
fix any positive $\Delta\le 0.4$ 
and by combining
Theorems \ref{thsngvsmpl} and \ref{thnrmsmpl} deduce that with a probability at least
$0.8-2\Delta$ it holds that
$$10/||S_{\mathcal I,:}^+||^2\ge 
1-\sqrt{4r\ln (2r/\Delta)/k}~
{\rm and}~10/||T_{:,\mathcal J,}^{*+}||^2\ge 
1-\sqrt{4r\ln (2r/\Delta)/l};$$
  consequently $||S_{\mathcal I,:}^+||~
||T_{:,\mathcal J,}^{*+}||\le 12.5$
for $k=l=100r\ln (2r/\Delta)$, say.

%------------------------------------------------------------------------------

\medskip       
       \medskip        
   
{\bf \Large PART II. CUR LRA VIA VOLUME MAXIMIZATION}         
         
%------------------------------------------------------------------------------

\section{Volumes of Matrices and Matrix Products}\label{svlmprv}

%------------------------------------------------------------------------------
   
   In the next section we study CUR LRA based on the second CUR criterion, that is, on the maximization of the volume of a CUR generator. 
   
In the next subsection we
define and estimate the
basic concept of matrix volume.

%------------------------------------------------------------------------------ 
   
In Section \ref{svlmprtrb}  we  estimate the   impact of a perturbation of a matrix on its volume. 
 
 In Section 
  \ref{svlmprd} we bound the volume of a matrix product via the volumes of the factors. 
 
%------------------------------------------------------------------------------

\subsection{Volume of a matrix: definitions and the known upper  estimates}\label{svlm1}
  
%------------------------------------------------------------------------------

For  
a triple of integers $k$,  $l$, and $r$ such that $1\le r\le \min\{k,l\}$,
the {\em volume} $v_2(M)$ and the $r$-{\em projective volume}
$v_{2,r}(M)$
of a  $k\times l$ matrix $M$
 % \cite[Corollary 2.3]{GT01} has 
%introduced 
%a basic concept of the {\em volume} of a square matrix and 
 %and  
%provided a fundamental upper bound  
% on the Chebyshev norm of $\Delta$  
%in terms of the  {\em volume} of the   square matrix $W_{I,J}$, said 
 % to be $|\det (W)|$.
% \cite{M14} has extended the concept of volume and  
%\cite[Corollary 2.3]{GT01}
are defined 
as follows:

\begin{equation}\label{eqvol}
v_2(M):=\prod_{j=1}^{\min\{k,l\}} \sigma_j(M),~~
v_{2,r}(M):=\prod_{j=1}^r \sigma_j(M),
\end{equation}
\begin{equation}\label{eqvlms}
v_{2,r}(M)=v_{2}(M)~{\rm if}~r=\min\{k,l\},
\end{equation}
$$v_2^2(M)=\det(MM^*)~{\rm if}~k\ge l;~
v_2^2(M)=\det(M^*M)~{\rm if}~k\le l,~v_2^2(M)=|\det(M)|~{\rm if}~k=l.$$

%------------------------------------------------------------------------------
  
  We use these concepts for devising efficient algorithms for LRA; they also have distinct applications (see \cite{B-I92}).
 
Here are some upper bounds 
on the matrix volume.

 For a $k\times l$ matrix  $M=(m_{ij})_{i,j=k,l}^r$
write ${\bf m}_j:=(m_{ij})_{i=1}^k$ and  ${\bf\bar m}_i:=((m_{ij})_{j=1}^l)^*$ for all $i$ and $j$. For $k=l=r$
 recall {\em Hadamard's bounds}

\begin{equation}\label{eqhdm} 
v_2(M)=|\det(M)|\le \min~\{\prod_{j=1}^r||{\bf m}_j||,~
\prod_{i=1}^r||{\bf \bar m^*}_j||,~r^{r/2}\max_{i,j=1}^r |m_{ij}|^r\}.
\end{equation}

 For $k\le l\le q$, 
a $k\times q$ matrix $M$ and its  any $k\times l$ submatrix $M_{k,l}$
observe that

\begin{equation}\label{eqvmsb}
v_2(M)=v_{2,k}(M)\ge v_2(M_{k,l})=v_{2,k}(M_{k,l}).
\end{equation}

%------------------------------------------------------------------------------

\subsection{The volume and $r$-projective volume of a perturbed matrix}\label{svlmprtrb}

%-----------------------------------------------------------------------------
   
\begin{theorem}\label{thprtrbvlm} 
Suppose that 
$W'$ and $E$ are $k\times l$ 
matrices, 
$\rank(W')=r\le\min\{k,l\}$,
$W=W'+E$, and  
$||E||\le \epsilon$. Then
\begin{equation}\label{eqrts}
\Big (1-\frac{\epsilon}{\sigma_r(W)}\Big )^{r}
\le \prod_{j=1}^r\Big (1-\frac{\epsilon}{\sigma_j(W)}\Big )
\le\frac{v_{2,r}(W)}{v_{2,r}(W')}\le
 \prod_{j=1}^r\Big (1+\frac{\epsilon}{\sigma_j(W)}\Big )\le  
\Big (1+\frac{\epsilon}{\sigma_r(W)}\Big )^r.
\end{equation}
If $\min\{k,l\}=r$, then $v_2(W)=v_{2,r}(W)$,  $v_2(W')=v_{2,r}(W')$, and 

\begin{equation}\label{eqrts1}
\Big (1-\frac{\epsilon}{\sigma_r(W)}\Big )^{r}
\le\frac{v_2(W)}{v_2(W')}=
\frac{v_{2,r}(W)}{v_{2,r}(W')}\le
\Big (1+\frac{\epsilon}{\sigma_r(W)}\Big )^r.
\end{equation}
\end{theorem}
\begin{proof}
Bounds (\ref{eqrts}) follow because a perturbation of a matrix within 
a norm bound $\epsilon$ changes its singular values by at most $\epsilon$ (see 
\cite[Corollary 8.6.2]{GL13}).
Bounds (\ref{eqrts1}) follow because
$v_2(M)=\prod_{j=1}^r\sigma_j(M)$
for any $k\times l$ matrix $M$ with $\min\{k,l\}=r$, in particular 
for $M=W$ and $M=W+E$.
\end{proof}

If the ratio 
$\frac{\epsilon}{\sigma_r(W)}$ is small,
then $\Big (1-\frac{\epsilon}{\sigma_r(W)}\Big )^{r}=1-O\Big (\frac{r\epsilon}{\sigma_r(W)}\Big )$ and  
$\Big (1+\frac{\epsilon}{\sigma_r(W)}\Big )^r=
1+O\Big (\frac{r\epsilon}{\sigma_r(W)}\Big )$,
which shows that the volume perturbation is amplified by at most a factor of $r$ in comparison to 
the perturbation of singular values.

%---------------------------------------------------------------------------- 

\begin{remark}\label{repert2nd}
Theorem \ref{thprtrbvlm} implies that a 
small norm perturbation of a matrix $W$
of rank $r$
 changes its volume little unless the matrix is ill-conditioned.
\end{remark}

%--------------------------------------------

\subsection{The volume and $r$-projective volume of a matrix product}\label{svlmprd}

%---------------------------------------------------------------------------- 

\begin{theorem}\label{thvolfctrg} 
(See Examples \ref{ex1} and  \ref{ex2} and Remark \ref{reO16} below.)

Suppose that $W=GH$ for an $m\times q$ matrix $G$ and a $q\times n$
matrix $H$.
Then 

(i) $v_2(W)=v_2(G)v_2(H)$
if $q=\min\{m,n\}$; 
$v_2(W)=0\le v_2(G)v_2(H)$ if 
$q<\min\{m,n\}$.

(ii) $v_{2,r}(W)\le v_{2,r}(G)v_{2,r}(H)$ for $1\le r\le q$, 

 (iii) $v_2(W)\le v_2(G)v_2(H)$ if $m=n\le q$.
\end{theorem}

The following examples show some limitations on the extension of the theorem.   

\begin{example}\label{ex1}
If $G$ and $H$ are unitary matrices and if $GH=O$, 
then $v_2(G)=v_2(H)=v_{2,r}(G)=v_{2,r}(H)=1$ and 
$v_2(GH)=v_{2,r}(GH)=0$ for all $r\le q$.
\end{example}

\begin{example}\label{ex2}
If $G=(1~|~0)$ and $H=\diag(1,0)$, 
then $v_2(G)=v_2(GH)=1$ and 
$v_2(H)=0$.
\end{example}

\begin{remark}\label{reO16}
  See distinct proofs of
claims (i)  
and (iii) in \cite{O16}.
\end{remark}

\begin{proof} 
First prove claim (i). 

Let $G=S_G\Sigma_GT^*_G$ and $H=S_H\Sigma_HT^*_H$ be SVDs
such that $\Sigma_G$, $T^*_G$, $S_H$, $\Sigma_H$,  and $U=T^*_GS_H$ 
are $q\times q$ matrices and
$S_G$, $T^*_G$, $S_H$, $T^*_H$,  and $U$ are unitary matrices.

Write $V:=\Sigma_GU\Sigma_H$. Notice that 
$\det (V)=\det(\Sigma_G)\det(U)\det(\Sigma_H)$. Furthermore
 $|\det(U)|=1$  because 
$U$ is a square unitary matrix.
Hence $v_2(V)=|\det (V)|=|\det(\Sigma_G)\det(\Sigma_H)|=v_2(G)v_2(H)$.

Now let $V=S_V\Sigma_VT_V^*$ be SVD where $S_V$, $\Sigma_V$, and $T_V^*$ 
are $q\times q$ matrices and where $S_V$ and $T_V^*$ are unitary matrices.

Observe that $W=S_GVT^*_H=S_GS_V\Sigma_VT_V^*T^*_H
=S_W \Sigma_VT^*_W$ where $S_W=S_GS_V$ and $T^*_W=T_V^*T^*_H$
are unitary matrices. 
Consequently
$W=S_W \Sigma_VT^*_W$ is SVD, and so $\Sigma_W=\Sigma_V$.

Therefore $v_2(W)=v_2(V)=v_2(G)v_2(H)$
unless $q<\min\{m,n\}$.
This proves claim (i) because clearly 
$v_2(W)=0$ if $q<\min\{m,n\}$.

Next prove claim (ii).

First assume that $q\le \min\{m,n\}$ as in claim (i)
and let $W=S_W \Sigma_WT^*_W$ be SVD.

In this case we have proven that
 $\Sigma_W=\Sigma_V$ for $V=\Sigma_GU\Sigma_H$, 
$q\times q$ diagonal matrices $\Sigma_G$ and $\Sigma_H$,
and a $q\times q$ unitary matrix $U$.
Consequently $v_{2,r}(W)=v_{2,r}(\Sigma_V)$.

In order to prove claim (ii) in the case where
$q\le \min\{m,n\}$, it remains to deduce that
\begin{equation}\label{eqghv}
v_{2,r}(\Sigma_{V})\le v_{2,r}(G) v_{2,r}(H).
\end{equation} 

Notice that
$\Sigma_V=S_V^*VT_V=S_V^*\Sigma_GU\Sigma_HT_V$
for $q\times q$ unitary matrices $S_V^*$ and $H_V$.

Let $\Sigma_{r,V}$ denote the $r\times r$
leading submatrix of $\Sigma_V$, and so
$\Sigma_{r,V}=\widehat G \widehat H$
where
$\widehat G:=S_{r,V}^*\Sigma_GU$ and $\widehat H:=\Sigma_HT_{r,V}$ and
where $S_{r,V}$ and $T_{r,V}$ denote the 
$r\times q$ leftmost unitary submatrices 
of the matrices $S_V$
and $T_V$, respectively.

Observe that $\sigma_j(\widehat G)\le \sigma_j(G)$ for 
all $j$ because $\widehat G$ is a submatrix of the $q\times q$ matrix $S_{V}^*\Sigma_GU$,
and similarly 
$\sigma_j(\widehat H)\le \sigma_j(H)$ for all $j$. Therefore
$v_{2,r}(\widehat G)=v_{2}(\widehat G)\le v_{2,r}(G)$
and $v_{2,r}(\widehat H)=v_{2}(\widehat H)\le v_{2,r}(H)$.
Also notice that 
$v_{2,r}(\Sigma_{r,V})=v_{2}(\Sigma_{r,V})$.

Furthermore $v_{2}(\Sigma_{r,V})\le 
v_{2}(\widehat G)v_{2}(\widehat H)$
by virtue of claim (i) because $\Sigma_{r,V}=\widehat G \widehat H$.

Combine the latter relationships and obtain (\ref{eqghv}), 
which implies claim (ii) in the case where
$q\le \min\{m,n\}$.

Next we extend claim (ii)  to the general case of any  positive 
integer $q$.

Embed a matrix $H$ into a $q\times q$  matrix $H':=(H~|~O)$
banded by zeros if $q>n$. Otherwise write $H':=H$.
 Likewise 
embed a matrix $G$ into a $q\times q$  matrix $G':=(G^T~|~O)^T$
banded by zeros if $q>m$. Otherwise write $G':=G$.

Apply claim (ii) to the $m'\times q$ matrix $G'$ and $q\times n'$ matrix
 $H'$ where $q\le \min\{m',n'\}$. 

Obtain that $v_{2,r}(G'H')\le v_{2,r}(G')v_{2,r}(H')$.

Substitute equations $v_{2,r}(G')=v_{2,r}(G)$, $v_{2,r}(H')=v_{2,r}(H)$,  and
$v_{2,r}(G'H')=v_{2,r}(GH)$,
which hold because the embedding 
keeps invariant the singular values
and therefore keeps invariant the  volumes  of the matrices 
$G$, $H$,  and $GH$. 
This completes the  proof of claim (ii), which
implies claim (iii) because  
$v_2(V)=v_{2,n}(V)$ if $V$ stands for $G$, $H$, or $GH$
and if $m=n\le q$.
\end{proof}  
  
%------------------------------------------------------------------------------

\begin{corollary}\label{coinvmx}
Suppose that $BW=(BU|BV)$ for a nonsingilar
matrix $B$ and that the submatrix $U$ 
is $h$-maximal in the matrix $W=(U|V)$.
Then  the submatrix $BU$ 
is $h$-maximal in the matrix $BW$.
\end{corollary}

%------------------------------------------------------------------------------

\section{Criterion of Volume Maximization
(the Second CUR Criterion) and the Second Proof of the Accuracy of Superfast CUR LRA of Random and Average Matrices}\label{simpvlmxmz}
 
%-------------------------------------------------------

 In this section we estimate the errors of CUR LRAs based on the maximization of the volume of a CUR generator, which is our {\em second CUR  criterion}.
 The estimates are quite different from those of Sections \ref{sprmcn} and  \ref{slracur}, but also enable us to  
   prove that 
whp  primitive,  cynical, and C-A  superfast algorithms are accurate on a  perturbed factor-Gaussian  matrix and hence are accurate on the average input allowing LRA. 

%------------------------------------------------------------------------------

In the next subsection we recall  the error  bounds of \cite{O16} 
for a CUR LRA 
in terms of the volume of CUR generator:
the errors are small
if the volume
is close to maximal. 

In Section \ref{smxvorprv} we minimize the worst case norm bound depending on the size of a CUR generator.

In Section \ref{scndlrgv} we show that 
whp the volume 
 is nearly maximal for
a $k\times l$ submatrix of
a  perturbed $m\times n$  factor-Gaussian  matrix that
hasg expected rank $r$, where $k$, $l$, $m$, $n$, and $r$
satisfy (\ref{eqklmnr}).
 
Together the two subsections imply that
whp any $k\times l$ CUR generator defines 
 an accurate 
 CUR LRA
 of a matrix of the above classes as well as of its small norm perturbation 
 and hence defines an accurate 
 CUR LRA of the average matrix allowing its LRA. This yields alternative proofs of the main results of Section \ref{serralg1}, except for those proven for sparse input matrices; our study in this section applies to any pair  of integers $k$ and $l$ that satisfy (\ref{eqklmnr}).
  
%------------------------------------------------------------------------------

\subsection{Volume maximization as the
second CUR criterion}\label{svlm1imp}

%------------------------------------------------------------------------------

The following result is \cite[Theorem 6]{O16}; it is also
 \cite[Corollary 2.3]{GT01} in the  special case where $k=l=r$ and $m=n$.
  
\begin{theorem}\label{th12}  
Suppose that $r:=\min\{k,l\}$, 
 $W_{\mathcal I,\mathcal J}$
is the $k\times l$ CUR generator, 
$U=W_{\mathcal I,\mathcal J}^+$  is the
nucleus  defining
a canonical
 CUR LRA $W'=CUR$ of an $m\times n$ matrix $W$, $E=W-W'$  (see (\ref{eqlrk}), (\ref{eqcur0}),
 and  (\ref{eqnuc+svd})),  $h\ge 1$,
 and $$h~v_{2}(W_{\mathcal I,\mathcal J})=\max_B v_2(B)$$
where   the maximum is over all $k\times l$  submatrices $B$
 of the matrix $W$. 
Then 
$$||E||_C\le h~f(k,l)~\sigma_{r+1}(W)~~{\rm for}~~
f(k,l):=\sqrt{\frac{(k+1)(l+1)}{|l-k|+1}}.
$$ 
\end{theorem}     

\begin{theorem}\label{th3}  \cite[Theorem 7]{O16}.
Suppose that  
 $W_{k,l}=W_{\mathcal I,\mathcal J}$
is a $k\times l$ submatrix of 
 an $m\times n$ matrix $W$,
 $U=W_{k,l,r}^+$ 
 (cf.  (\ref{eqnuc+svd}))
 is the nucleus of
a canonical CUR LRA of $W$, $E=W-W'$, $h\ge 1$,
 and 
  $$h~v_{2,r}(W_{\mathcal I,\mathcal J})=\max_B v_{2,r}(B)$$                                                                                                                                                                                                                                                                                                                                                                                                                                                                                                                             
 where   the maximum is over all $k\times l$ submatrices $B$
 of the matrix $W$. 
Then 
$$||E||_C\le h~f(k,l,r)~\sigma_{r+1}(W)
~~{\rm for}~~f(k,l,r):=
\sqrt{\frac{(k+1)(l+1)}{(k-r+1)(l-r+1)}}.
$$
%and consequently (see (\ref{eq0}))  
%$$||\Delta(r)||\le (r+1)h\sqrt {mn}~\sigma_{r+1}(W).$$
\end{theorem}

\begin{remark}\label{rechb} 
The bounds of Theorems \ref{th12} and 
 \ref{th3} would increase by a factor of
 $\sqrt {mn}$ in the transition from 
 the Chebyshev to Frobenius norm
(see (\ref{eq0})),
but the Chebyshev error norm  is  most adequate where one deals with immense  matrices representing Big Data, so that the computation of its Frobenius or even spectral norm is unfeasible because only a very small fraction of all its entries can be accessed. 
\end{remark}
 
%------------------------------------------------------------------------------

\subsection{Optimization of the sizes of CUR generators}\label{smxvorprv}

%------------------------------------------------------------------------------

Let us optimize the size $k\times l$ of a CUR generator towards minimization of the  bounds of Theorems \ref{th12} and \ref{th3} on the error norm $|||E||_C$. 
 
 The bound of Theorem \ref{th12}
turns into 
$$||E||_C\le (r+1)~h~\sigma_{r+1}(W)$$
if $k=l=r$ and into
$$||E||_C\le \sqrt {(1+1/b)(r+1)}~h~\sigma_{r+1}(W)$$
if $k=r=(b+1)l-1$ or $l=r=(b+1)k-1$ and if $b>0$, that is, we decrease the output error bound  
 by a factor of $\sqrt {\frac{r+1}{1+1/b}}$ 
 in the latter case.
 
 The bound of Theorem \ref{th3}
turns into 
$$||E||_C\le (1+1/b)h~\sigma_{r+1}(W)$$
and is  minimized
for $k=l=(b+1)r-1$ and a positive $b$. 

This upper estimate shows that 
the volume is maximal where 
$\min\{k,l\}=r<\max\{k,l\}$  and that the $r$-projective volume is maximal where
$l=k>r$. Furthermore the  upper estimate of Theorem \ref{th3} for the norm $||E||_C$ converges to 
$\sigma_{r+1}(W)$ 
as $h\rightarrow 1$ and 
 $b\rightarrow \infty$.

%------------------------------------------------------------------------------
  
\subsection{Any submatrix of the average matrix allowing LRA has a nearly maximal volume}\label{scndlrgv}
  
%------------------------------------------------------------------------------

Claim (i) of Theorem \ref{thvolfctrg} 
implies the following result. 

\begin{theorem}\label{th33}
Suppose that $W=G\Sigma H$ is an $m\times n$ diagonally scaled factor-Gaussian matrix with an expected rank $r$, 
$W_{k,l}=G_k\Sigma H_l$  is its  $k\times l$ submatrix,  
 $G\in \mathcal G^{m\times r}$, 
$H\in \mathcal G^{r\times n}$,
 and 
 $\Sigma \in \mathbb C^{r\times r}$ is
a well-conditioned  matrix.
Then  
$$v_2(W_{k,l})=v_2(G_k\Sigma H_l)= 
v_2(G_k)v_2(\Sigma)v_2(H_l).$$
\end{theorem}

Now recall  the following  theorem.

\begin{theorem}\label{thjlcncntr} (See \cite{MZ08}.)
Let $G\in \mathcal G^{r\times l}$
for $1\le r\le l$ and let 
$\epsilon$ be a small positive number.
%$z=(\chi_k^2\chi_{k-1}^2\cdots \chi_{k-r%+1}^2)^{1/r}$. 
Then 
$${\rm Probability}\Big \{v_2(G)^{1/r}\ge (1+2\epsilon)\mathbb E(v_2(G)^{1/r})\Big\}=O(\exp(-l\epsilon^2))$$
and 
$${\rm Probability}\Big \{v_2(G)^{1/r}\le (1-\epsilon)\mathbb E(v_2(G)^{1/r})\Big \}=O(\exp(-l\epsilon^2)).$$
\end{theorem}

The  theorem implies that the volume of a Gaussian matrix is very strongly concentrated about its expected value.
Therefore  the volume of 
a $k\times l$ submatrix $W_{k,l}$ of $W$
is also very strongly concentrated about the expected value 
$\nu_{k,l,r,\Sigma}=\mathbb E(v_2(G_k\Sigma H_l))$
of the volume 
$v_2(G_k\Sigma H_l)= 
v_2(G_k)v_2(\Sigma)v_2(H_l)$ of the matrix
$G_k\Sigma H_l$. This value
 is invariant in the choice of a   
$k\times l$ submatrix of $W=G\Sigma H$ and
only depends on  the matrix $\Sigma$ 
and the integers $k$, $l$,  
and  $r$.
 
It follows that whp the volume of any such a submatrix is
close to the  value $\nu_{k,l,r,\Sigma}$,
which is within a factor 
close to 1 from the maximal volume; moreover this factor little depends on the choice of a triple of integers $k$, $l$, and $r$. 

The scaled factor-Gaussian inputs with expected rank $r$ are a special case, and the argument and the
results of our study are  readily extended
 to the case of the left and right factor-Gaussian  matrices $W$
with expected rank $r$ and then to the
  average inputs $W$ obtained over the scaled, diagonally scaled, left,
  and right  
factor-Gaussian inputs.
Theorem \ref{thprtrbvlm} implies extension 
of our results to inputs lying near factor-Gaussian and average matrices.

%------------------------------------------------------------------------------

\section{C-A Towards Volume Maximization}\label{svlmxmz}
  
%------------------------------------------------------------------------------
%------------------------------------------------------------------------------

\subsection{Section's outline}\label{soutl1} 

%------------------------------------------------------------------------------
  
In the next two subsections  we apply  two successive C-A steps to  {\em the worst case input matrix} $W$  of numerical rank $r$ and assume that they have been 
 initiated at a submatrix of numerical rank $r$ and have
  output a $k\times l$ matrix having
  $\bar h$-maximal
 volume among  $k\times l$ submatrices of the input matrix $W$. We  bound the value
  $\bar h$ in case of
small  ranks $r$ and then deduce from Theorem \ref{th12} that
 already the two C-A steps generate
 a CUR LRA of the matrix $W$ within a bounded error.
   
  In Sections \ref{sprdv_via_v} and
\ref{simpct}  we extend this result to the maximization of $r$-projective volume rather than volume of  a 
  CUR generator. (Theorem \ref{th3} shows benefits of
such  maximization.)

 In Section \ref{scurac} we
 summarize our study in this section and
comment on the estimated and empirical  performance of C-A algorithms. 

%------------------------------------------------------------------------------
   
\subsection{We only need to maximize 
 $k\times l$ volume in the input sketches of a two-step C-A  loop}\label{svlmcross}

%------------------------------------------------------------------------------
 
We begin with some   
 definitions and simple auxiliary results.

\begin{definition}\label{defmxv} 
The volume of a $k\times l$ submatrix $W_{\mathcal I,\mathcal J}$
of a matrix $W$ is $h$-{\em maximal} if 
over all its
$k\times l$ submatrices
this volume is maximal up to a factor of $h$. 
The volume $v_2(W_{\mathcal I,\mathcal J})$ is {\em column-wise} (resp. 
{\em row-wise}) $h$-maximal if it is $h$-maximal in the submatrix 
$W_{\mathcal I,:}$ (resp. $W_{:,\mathcal J})$. Such a volume is 
 {\em column-wise} (resp. 
{\em row-wise}) {\em locally} $h$-maximal if it is $h$-maximal over all submatrices of $W$
that differ from it by a single column  (resp. single row).
Call volume  $(h_c,h_r)$-{\em maximal}
 if it is  both
 column-wise $h_c$-maximal and row-wise $h_r$-maximal.
Likewise  define {\em locally}  $(h_c,h_r)$-{\em maximal} volume.
Call $1$-maximal and $(1,1)$-maximal volumes {\em maximal}. 
Extend all these definitions to the case of $r$-projective volumes.
\end{definition}

By comparing SVDs of two matrices $W$ and $W^+$ we obtain the following lemma.

\begin{lemma}\label{le0} 
$\sigma_j(W)\sigma_j(W^+)=1$ for 
 all matrices $W$ and all subscripts $j$, 
 $j\le \rank(W)$.
\end{lemma}

\begin{corollary}\label{co0} 
$v_2(W)v_2(W^+)=1$ 
and $v_{2,r}(W)v_{2,r}(W^+)=1$
for all  matrices $W$ of full rank and all integers $r$ such that $1\le r \le \rank(W)$.
\end{corollary}

%------------------------------------------------------------------------------

Now
we are ready to prove that for some specific constants $g$ and $h$
nonzero volume of a $k\times l$
  submatrix of the matrix $W$ is
  $g$-maximal globally, that is,
  over all its $k\times l$ submatrices, 
    if it is $h$-maximal locally, over the 
    $k\times l$ submatrices of 
    two input sketches of two successive  C-A steps.

\begin{theorem}\label{th111} 
Suppose that a nonzero volume 
 of a $k\times l$ submatrix  
$W_{\mathcal I,\mathcal J}$ 
is 
 $(h,h')$-maximal for $h\ge 1$ 
 and $h'\ge 1$
 in a  matrix $W$.
 Then this volume is
$hh'$-maximal  
 over all its $k\times l$ submatrices of
 the matrix $W$.
\end{theorem}

\begin{proof} 
  The matrix 
  $W_{\mathcal I,\mathcal J}$ has full rank because its volume is nonzero.
  
 Fix any $k\times l$ submatrix 
 $W_{\mathcal I',\mathcal J'}$ of the matrix $W$, apply Theorem \ref{thcandec} and
 obtain that 
$$W_{\mathcal I',\mathcal J'}=
W_{\mathcal I',\mathcal J}W_{\mathcal I,\mathcal J}^{+}W_{\mathcal I,\mathcal J'}.$$
 
If $k\le l$, then
first apply claim (iii) of Theorem \ref{thvolfctrg}  for $G:=W_{\mathcal I',\mathcal J}$  and $H:=W_{\mathcal I,\mathcal J}^{+}$;
 then apply claim (i)
of that theorem for $G:=W_{\mathcal I',\mathcal  J}W_{\mathcal I,\mathcal J}^{+}$
and $H:=W_{\mathcal I,\mathcal J'}$
and obtain that 
$$v_2(W_{\mathcal I',\mathcal J}W_{\mathcal I,\mathcal J}^{+}W_{\mathcal I,\mathcal J'})\le  
v_2(W_{\mathcal I',\mathcal J})v_2(W_{\mathcal I,\mathcal J}^{+})v_2(W_{\mathcal I,\mathcal J'}).$$

If $k>l$ deduce the same bound by applying
the same argument to the matrix equation
$$W_{\mathcal I',\mathcal J'}^T=
W_{\mathcal I,\mathcal J'}^TW_{\mathcal I,\mathcal J}^{+T}W_{\mathcal I',\mathcal J}^T.$$

Combine this bound with Corollary \ref{co0} for $W$ replaced by 
$W_{\mathcal I,\mathcal J}$ and deduce that
 \begin{equation}\label{eqloop1} 
 v_2(W_{\mathcal I',\mathcal J'})
 =v_2(W_{\mathcal I',\mathcal J}W_{\mathcal I,\mathcal J}^{+}W_{\mathcal I,\mathcal J'})\le v_2(W_{\mathcal I',\mathcal J})v_2 (W_{\mathcal I,\mathcal J'})/v(W_{\mathcal I,\mathcal J}).  
\end{equation}
Recall that the matrix $W_{\mathcal I,\mathcal J}$
is $(h,h')$-maximal
and conclude that
$$hv_2(W_{\mathcal I,\mathcal J})\ge v_2(W_{\mathcal I,\mathcal J'})~{\rm 
and}~ 
h'v_2(W_{\mathcal I,\mathcal J})\ge v_2(W_{\mathcal I',\mathcal J}).$$ 

Substitute these inequalities into the above bound on the volume $v_2(W_{\mathcal I',\mathcal J'})$ and obtain
that $v_2(W_{\mathcal I',\mathcal J'})\le hh' v_2(W_{\mathcal I,\mathcal J})$.
\end{proof}

%------------------------------------------------------------------------------
  
\subsection{From locally to globally $h$-maximal 
 volumes of full rank submatrices in 
 sketches of the same full rank}\label{sc-avl} 
   
%------------------------------------------------------------------------------

\begin{theorem}\label{thlcl1}
Suppose that 
 $r\times l$ submatrix $U$  has 
 a nonzero column-wise 
   locally $h$-maximal volume in the matrix $W=(U~|~V)\in\mathbb C^{r\times n}$ for $h\ge 1$. Then this submatrix 
has $\widehat h$-maximal volume in the matrix $W$
for $\widehat h=t_{n,r,h}^r$ and $t_{p,r,h}^2=(p-r)rh^2+1$ of (\ref{eqtkrh}).
\end{theorem}

%------------------------------------------------------------------------------

\begin{proof} 
 By means of orthogonalization of the rows of the matrix $W$ obtain its  factorization $W=RQ$ where $R$ is a $r\times r$ nonsingular matrix and 
$Q=(R^{-1}U~|~R^{-1}V)$ is a  $r\times n$ unitary matrix and deduce from Corollary \ref{coinvmx}  that  the  volume  of the matrix $R^{-1}U$ is column-wise locally $h$-maximal in the matrix $Q$.

 Therefore $\sigma_r(R^{-1}U)\ge 
 \sigma_r(Q)/t_{n,r,h}$ by virtue of Theorem \ref{thcndprt}.
 
 Combine this bound with the
 relationships $\sigma_r(Q)=1$ 
 and $v_2(R^{-1}V)\ge (\sigma_r(R^{-1}V))^r$ and deduce that
$\widehat hv_2(R^{-1}U)\ge 1$ for
$\widehat h=t_{n,r,h}^r$.

Notice that $v_2(Q_l)\le v_2(Q)=1$ for any $r\times l$ submatrix $Q_l$ of $Q$.

Hence the volume $v_2(R^{-1}U)$ is $\widehat h$-maximal in $Q$.

 Now
Theorem \ref{thlcl1} 
follows from
Corollary \ref{coinvmx}.
\end{proof}
 
\begin{example}\label{exlcl1}
The bound of Theorem \ref{thlcl1}
is quite tight for $r=h=1$.
Indeed the unit row vector 
${\bf v}=\frac{1}{\sqrt n}(1,\dots,1)^T$ of dimension $n$ is a $r\times n$ matrix
for $r=1$. Its coordinates 
are  
$r\times r$ submatrices, all having volume 
$\frac{1}{\sqrt n}$. Now
 notice that $\sqrt n\approx \widehat h
 =((n-1)+1)^{1/2}=t_{n,1,1}$.
 \end{example}
 
%------------------------------------------------------------------------------

\begin{remark}\label{relcliden}
The theorem is readily extended to the case of a $k\times n$ matrix $W$ of rank
$r$,  $0<r\le k\le n$, where
 $r$-projective volume replaces volume. Indeed row ortogonalization reduces the extended claim precisely to Theorem \ref{thlcl1}. 
\end{remark}
 
%------------------------------------------------------------------------------

By following \cite{GOSTZ10} we decrease the upper bound 
$\widehat h=t_{n,l,h}^r$ of 
Theorem \ref{thlcl1} in the case 
where
$l=r$. We begin with a lemma.

%------------------------------------------------------------------------------

\begin{lemma}\label{lelcliden}
 (Cf. \cite{GOSTZ10}.)
Let   
$W=(I_r~|~V)\in\mathbb C^{r\times n}$ 
for $r\le n$ and let the submatrix $I_r$ have column-wise locally $h$-maximal volume in $W$ 
for $h\ge 1$.
Then $|| W||_C\le h$.
\end{lemma}
\begin{proof} Let 
$|w_{ij}|>h$ for an entry $w_{ij}$ of the matrix $W$, where, say, $i=1$. Interchange its first and $j$th columns. Then the leftmost block $I_r$ turns into
the matrix  
 $R=\begin{pmatrix} w_{1j} & {\bf u}^T\\
 {\bf 0}  & I_{r-1}
 \end{pmatrix}$. Hence 
 $v_2(R)=|\det (R)|=|w_{1j}|>h$. Therefore
  $I_r$ is not a 
 column-wise  locally $h$-maximal submatrix of $W$. The contradiction implies that 
 $||W||_C\le h$. 
\end{proof} 

%------------------------------------------------------------------------------

\begin{theorem}\label{thlclglb}
 (Cf. \cite{GOSTZ10}.)
Suppose that
 $r\times r$ submatrix $U$  has  a nonzero
column-wise locally $h$-maximal volume in a  matrix 
$W=(U~|~V')\in\mathbb C^{r\times n}$ for $h\ge 1$. Then this submatrix
has $\tilde h$-maximal volume in $W$
for $\tilde h=h^rr^{r/2}$.
\end{theorem}

%------------------------------------------------------------------------------

\begin{proof} Apply Lemma \ref{lelcliden} 
to the matrix $U^{-1}W=(I_r~|~V)$ 
for $V=U^{-1}V'$ and obtain that
$||U^{-1}W||_C\le h$. 
 Hadamard's bound (\ref{eqhdm}) for $M=V$  implies that the volume 1 of the
 submatrix $I$ is 
 $\tilde h$-maximal
  in the matrix $U^{-1}W$ for the claimed value of 
 $\tilde h$.  Now deduce from 
Corollary \ref{coinvmx} that 
 the submatrix $U$ has $\tilde h$-maximal volume in $W$. \end{proof}
 
\begin{remark}\label{rerkl}
Clearly the bound $\widehat h=t_{n,r,h}^r$ 
of Theorem \ref{thlcl1} is larger than  
the bound  $\tilde h=r^{r/2}$ of 
 Theorem \ref{thlclglb}, but how much larger?
 Substitute a slightly smaller expression
$((k-r)rh^2)^{1/2}$ for
 $t_{k,r,h}=((k-r)rh^2+1)^{1/2}$
into the equation $\widehat h=t_{n,r,h}^r$ 
and 
observe that the  resulting decreased value   
  is still larger  
than $\tilde h=h^rr^{r/2}$  by a
 factor of 
$(n-r)^{r/2}$.
\end{remark}

%------------------------------------------------------------------------------
%--  

\subsection {Extension of  the maximization of the
 volume of a full rank matrix
 to the maximization of its
 $r$-projective volume}\label{sprdv_via_v}

%-----------------------------------------------
  
Recall from Section \ref{smxvorprv} 
that the output error bounds of
 Theorems \ref{th12} and  \ref{th3}
are strengthened when we maximize 
$r$-projective volume for $r<k=l$. Next 
 we reduce such a task to the maximization of the volume of  
 $k\times r$ or $r\times l$ CUR
  generators of  full rank for
 $r=\min\{k,l\}$.
 
Corollary \ref{coinvmx} implies
  the following lemma.
 
 \begin{lemma}\label{leprwnmp} 
 Let $M$ and $N$ be a pair of $k\times l$
 submatrices of a $k\times n$ matrix
 and let $Q$ be a $k\times k$
 unitary matrix.
 Then 
 $v_2(M)/v_2(N)=v_2(QM)/v_2(QN)$,
and if $r\le \min\{k,l\}$, then
also 
$v_{2,r}(M)/v_{2,r}(N)=
v_{2,r}(QM)/v_{2,r}(QN)$.
\end{lemma} 
  
\begin{algorithm}\label{algprjvlm}
{\em Maximization of   
 $r$-projective volume via
 maximization of the volume of a full rank matrix.}
 
%------------------------------------------------------------------------------

\begin{description}

%------------------------------------------------------------------------------

\item[{\sc Input:}]
Four integers $k$,  $l$, $n$, and $r$ such that
$0<r\le k$ and $r\le l\le n$,
a $k\times n$ matrix $W$ of rank $r$ and a black box algorithm 
that computes a $r\times l$ submatrix of maximal volume in a
$r\times n$ matrix of full rank $r$. 

%------------------------------------------------------------------------------

\item[{\sc Output:}]
A column set $\mathcal J$
such that the $k\times l$
submatrix $W_{:,\mathcal J}$
has maximal $r$-projective volume in
the matrix $W$. 

%------------------------------------------------------------------------------
%-------------------------------------------------------------------------------

\item[{\sc Computations:}]

\begin{enumerate}
\item %1   
 Compute a rank-revealing QRP factorization $W=QRP$,
where  $Q$ is a unitary matrix,  $P$ is a permutation matrix, 
 $R=\begin{pmatrix} R' \\ O
\end{pmatrix}$ and 
 $R'$ is a $r\times n$ 
 matrix.\footnote{One can apply
other rank-revealing factorizations  instead.} (See \cite[Sections 5.4.3 and 5.4.4]{GL13} and \cite{GE96}.)
 
\item %2
 Compute a $r\times l$ submatrix 
 $R'_{:,\mathcal J}$ of $R'$ having maximal volume 
$v_2(R')$ and output
the matrix 
$W_{:,\mathcal J}$.  
 
\end{enumerate}

%------------------------------------------------------------------------------

\end{description}

%------------------------------------------------------------------------------

\end{algorithm}

%------------------------------------------------------------------------------

The 
 submatrices $U'$ and
$\begin{pmatrix} U' \\O
\end{pmatrix}$ have  maximal volume and
 maximal 
$r$-projective volume in  the matrix $R$,
respectively, by virtue of Theorem \ref{thvolfctrg} and  
because $v_2(\bar U)=v_{2,r}(\bar U)=v_{2,r}(U')$. Therefore
the submatrix $W_{:,\mathcal J}$
has maximal $r$-projective volume in  the matrix $W$ by virtue of Lemma  \ref{leprwnmp}. 

 \begin{remark}\label{reprvtov} 
 By transposing a horizontal input matrix $W$ and interchanging the integers $m$ with $n$ and $k$ with $l$ we extend the algorithm to computing
a $k\times l$ submatrix of maximal or nearly  maximal $r$-projective volume in  an $m\times l$ matrix of rank $r$.
\end{remark} 
 
%------------------------------------------------------------------------------
  
\subsection{From local to global $h$-maximization of the
 volume of a submatrix}\label{simpct}
  
%------------------------------------------------------------------------------

By combining 
Theorems \ref{thlcl1}
and \ref{thlclglb} deduce that
the volume of a $r\times l$ submatrix is 
$\bar h$-maximal in a $r\times n$ matrix
of rank $r$
for $\bar h=t_{n,r,h}^r$ if $l>r$ and 
for $\bar h=h^rr^{r/2}$ if $l=r$ 
provided that the volume of the submatrix is 
column-wise locally $h$-maximal.
 The theorems imply that
Algorithm \ref{algprjvlm}   
 computes a
$k\times l$ submatrix having maximal
 $r$-projective volume in
 an $m\times n$ matrix of rank $r$
 for any integers  $k$, $l$,   $m$, $n$, and $r$ satisfying (\ref{eqklmnr}). 
The following theorem summarizes these observations.

%------------------------------------------------------------------------------

\begin{theorem}\label{coc-avlm}
Given five integers $k$, 
 $l$,  $m$, $n$, and  $r$ satisfying 
 (\ref{eqklmnr}), write $p:=\max\{k,l\}$ and suppose that
  two successive C-A steps (say, based on the
  algorithms of 
\cite{GE96} or \cite{P00})
  have been applied  
to an $m\times n$ matrix $W$ of rank $r$ and  have  output  
   $k\times l$ submatrices 
  $W'_{1}$ and  
  $W'_{2}=W_{\mathcal I_2,\mathcal J_2}$
    having nonzero $r$-projective
column-wise locally $h$-maximal and
 row-wise locally  $h'$-maximal volumes, respectively.
Then the submatrix $W_2'$
 has $\bar h$-maximal $r$-projective volume in the matrix $W$
where either
$\bar h:=(t_{n,r,h}t_{m,r,h'}r)^r$ for 
$t_{p,r,h}^2=(p-r)rh^2+1$
of (\ref{eqtkrh}) if $p>r$
or 
$\bar h:=(hh'r)^r$ if $p=r$
and for real values $h$ and $h'$
slightly exceeding 1.
\end{theorem}
 
%------------------------------------------------------------------------------

\begin{proof} 
 By applying Algorithm \ref{algprjvlm}
 reduce the claim of the theorem
 to the case where $r$-projective volume
 is equal to the volume of a matrix of full rank $r$. Then combine 
 Theorems \ref{th111}, \ref{thlcl1}, and \ref{thlclglb}.
\end{proof}

\begin{remark}\label{rerkl1}
(Cf. Remark \ref{rerkl}.)
 Substitute a slightly smaller expression
$((k-r)rh^2)^{1/2}$ for
 $t_{k,r,h}=((k-r)rh^2+1)^{1/2}$
 into the product 
 $(t_{n,r,h}t_{m,r,h'}r)^r$.
Then its value decreases but   
still  
exceeds  the bound 
$(hh'r)^r$ by a
 factor of 
$((m-r)(n-r))^{r/2}$.
\end{remark}
    
%------------------------------------------------------------------------------
  
\subsection{Complexity, accuracy,   and extension of a two-step C-A loop}\label{scurac}

%------------------------------------------------------------------------------
   
In this section we arrived at a C-A algorithm that computes a CUR approximation of a rank-$r$ matrix $W$.
 Let us summarize our study by combining
 Theorems \ref{th12},
\ref{th3}, and \ref{coc-avlm}.  

\begin{corollary}\label{colclgl}
Under the assumptions of Theorem 
\ref{coc-avlm}
apply a two-step  C-A loop to an $m\times n$  matrix $W$ 
and suppose that both its C-A steps 
output $k\times l$  submatrices having
nonzero $r$-projective column-wise and row-wise locally $h$-maximal volumes.
 Build
 a canonical CUR LRA  on
a CUR generator $W_2'=W_{k,l}$
of rank $r$ output by the second C-A step. 
Then the error matrix $E$ of the output  
CUR LRA satisfies the bound
$||E||_C\le g(k,l,r)~\bar h~\sigma_{r+1}(W)$
for 
$\bar h$ of Theorem  
\ref{coc-avlm} and $g(k,l,r)$
denoting the functions $f(k,l)$
of Theorem \ref{th12}
or  $f(k,l,r)$
of Theorem \ref{th3}; 
 the computation of this LRA by using 
the auxilairy algorithms of \cite{GE96} or  \cite{P00}
involves $(m+n)r$ 
memory cells and $O((m+n)r^2)$ flops.
\end{corollary}

%------------------------------------------------------------------------------
 
\begin{remark}\label{reklprtr}
Theorem \ref{thprtrbvlm}  enables us to
 extend 
  Algorithm \ref{algprjvlm},
  Theorem \ref{coc-avlm}, and 
Corollary \ref{colclgl} 
 to the case where an input matrix $W$ of
 numerical rank $r$ is  closely approximated by a rank-$r$ matrix
 (see Remark \ref{repert2nd}).
\end{remark}
 
%------------------------------------------------------------------------------
 
\medskip
 
The factor $\bar h$
of Theorem \ref{coc-avlm} is large already for
 moderately large integers $r$, 
 but 
 
%\medskip
 
 (i) in many important applications the numerical rank $r$ is small,
 
 (ii) our upper bounds $\bar h$ are overly pessimistic on the average input in view of our study in Section \ref{scndlrgv},
 
 (iii) continued iterations
  of the C-A algorithm  may possibly strictly 
   decrease the volume of a  CUR generator and hence an upper bound on  Chebyshev's norm of the CUR output error matrix in every C-A step. 
  
The algorithm of \cite{GOSTZ10} 
 locally maximizes the volume and 
  strictly decreases it 
 in every C-A step. For the worst case input the decrease is slow, but empirically 
 the convergence is fast,
in good accordance  with our study in Section \ref{scndlrgv}. 
 
 Our Sections \ref{scurhrd} and  \ref{spstr} describe superfast  
  a posteriori errors estimation.  
\medskip
\medskip
\medskip

\noindent {\bf \Large PART III. RANDOMIZED ALGORITHMS FOR CUR LRA}
  
%------------------------------------------------------------------------------

\section{Computation of LRA with Random Subspace Sampling \\    
 Directed by Leverage Scores}\label{slrasmp} 

%----------------------------------------- 
  
In this section we study statistical approach to the computation of  CUR generators  
satisfying the
{\em third CUR criterion}. The CUR LRA algorithms of \cite{DMM08}, implementing this approach,  are superfast for the worst case input, except for  their stage of computing leverage scores.   We, however,  bypass that stage simply by assigning the uniform leverage scores  and then prove that in this case the algorithms of \cite{DMM08} compute accurate CUR LRA for 
the average input matrices of a low numerical rank and whp for a  perturbed factor-Gaussian input matrices. 
Moreover, given  a crude but reasonably  close LRA   
of an input allowing even closer LRA, we yield  superfast  refinement
by using this approach. 

%Empirically we can  
%sample much fewer rows and columns of an %input matrix than we need in the
%formal estimates for the output errors; %we discuss this issue in Section %\ref{srprpr}.
  
%------------------------------------------------------------------------------ 
%------------------------------------------------------------------------------

\subsection{Fast accurate LRA with leverage scores defined by a singular space}\label{slrasmpsngv}
   
%-------------------------------------------------------- 
%------------------------------------------------------------------------------
 
Given the top SVD of   
 an $m\times n$ matrix $W$ of numerical rank $r\ll \min\{m,n\}$, we can compute a close CUR LRA of 
  that matrix  whp by applying 
 randomized Algorithm  \ref{algsvdtocur} or  \ref{algsvdtocur}a. Moreover  by  following \cite{DMM08} and
sampling sufficiently many 
columns and rows,  we  improve the
 output accuracy and yield nearly optimal  error bound 
 such  that 
    \begin{equation}\label{eqrmmrel}
 ||W-CUR||_F\le (1+\epsilon)\tilde\sigma_{r+1}
 \end{equation}
 for $\tilde\sigma_{r+1}$ of (\ref{eqnsgmr0}) and any fixed
   positive $\epsilon$.
     
 Let us supply some details. 
 Let the columns of an $n\times r$ unitary matrix 
%\begin{equation}\label{eqlevsc0} 
 $V^{(r)T}=(v_{h,j}^{(r)})_{h,j=1}^{n,r}=
 ({\bf v}_{j}^{(r)})_{j=1}^{r}$
 %\end{equation}
 span
  the right singular subspace
 associated with the $r$ largest singular values of an input matrix $W$ of numerical rank $r$. 
 Fix  scalars
 $p_1,\dots,p_n$, and  $\beta$ such that
 \begin{equation}\label{eqlevsc}
 0<\beta\le 1,~ 
 p_j \ge (\beta/r)||{\bf v}_{j}^{(r)}||^2
~{\rm for}
~j=1,\dots,n,~{\rm and}~
\sum_{j=1}^np_j=1. 
 \end{equation}
%\begin{equation}\label{eqbetnrmz}
%0<\beta\le 1~{\rm and}~ 
%\sum_{j=1}^np_j=1
% \end{equation}
 Call the 
scalars $p_1,\dots,p_n$
  the {\em SVD-based
leverage scores} for the matrix $W$ (cf. (\ref{eqsmpl})). They stay invariant
   if we  post-multiply 
 the matrix $V^{(r)}$ by a unitary matrix.
 Furthermore  
 \begin{equation}\label{eqlevsc1}
p_j = {\bf v}_{j}^{(r)}~{\rm for}~j=1,\dots,n~{\rm and}~\beta=1. 
\end{equation}

For any $m\times n$ matrix $W$,  
  \cite[Algorithm 5.1]{HMT11} computes the matrix 
 $V^{(r)}$ and leverage scores 
 $p_1,\dots,p_n$ by using $mn$ memory units  and $O(mnr)$ flops.
 
 Given an integer parameter $l$, $1\le l\le n$, and
leverage scores $p_1,\dots,p_n$,  Algorithms \ref{algsmplex}  and \ref{algsmplexp}, reproduced from \cite{DMM08},                                                                                                                                                                                                                                                                                                                                                                                                                                                                                                                                                         compute auxiliary sampling  
 and rescaling matrices, $S=S_{W,l}$ and
$D=D_{W,l}$, respectively. (In particular Algorithms \ref{algsmplex}  and \ref{algsmplexp}  sample and rescale either exactly $l$ columns  of an input matrix $W$ 
or at most its $l$ columns in expectation
-- the $i$th column  with probability 
$p_i$ or $\min\{1, lp_i\}$,
respectively.)  Then \cite[Algorithms 1 and 2]{DMM08}
compute a CUR LRA of a matrix $W$ as follows.

%------------------------------------------------------------------------------

\begin{algorithm}\label{algcurlevsc} 
{\rm CUR LRA by using SVD-based leverage scores.} 

%------------------------------------------------------------------------------

\begin{description}
  
%------------------------------------------------------------------------------

\item[{\sc Input:}] A matrix
$W\in C^{m\times n}$ with
 $\nrank (W)=r>0$.

%------------------------------------------------------------------------------

\item[{\sc Initialization:}]
Choose integers $k$ and $l$ satisfying 
(\ref{eqklmnr}) and  $\beta$ 
and  $\bar \beta$  in the range $(0,1]$.

\item[{\sc Computations:}]

\begin{enumerate}
\item%1 
Compute the leverage scores
$p_1,\dots,p_n$ of (\ref{eqlevsc}).
\item%2
Compute  sampling   and  rescaling  
matrices $S$ and $D$ 
by applying Algorithm \ref{algsmplex}
or \ref{algsmplexp}. 
Compute and output a CUR factor 
$C:=WS$.
\item%3 
Compute 
leverage scores
  $\bar p_1,\dots,\bar p_m$
satisfying relationships (\ref{eqlevsc})
under the following replacements:
$\{p_1,\dots,p_n\}\leftarrow 
\{\bar p_1,\dots,\bar p_m\}$,
 $W\leftarrow (CD)^T$ and
$\beta \leftarrow\bar \beta$. 
\item%4
By applying Algorithm \ref{algsmplex}
or \ref{algsmplexp} 
to these 
leverage scores 
%$\bar p_1,\dots,\bar p_m$
compute  $k\times l$ sampling 
matrix $\bar S$ and 
 $k\times k$ rescaling matrix 
 $\bar D$.
 \item%5
Compute and output a CUR factor 
$R:=\bar S^TW$. 

\item%6
Compute and output a CUR factor
$U:=DM^+\bar D$ for 
$M:=\bar D\bar S^TWSD$.
\end{enumerate} 

%------------------------------------------------------------------------------

\end{description}
\end{algorithm}

%------------------------------------------------------------------------------
%------------------------------------------------------------------------------
 
{\bf Complexity estimates:}
Overall Algorithm \ref{algcurlevsc}
involves  $kn+ml+kl$ memory cells and 
$O((m+k)l^2+kn)$ flops 
in addition to $mn$ cells and $O(mnr)$  flops used for
computing SVD-based leverage scores
at stage 1. Except for that stage
the algorithm is superfast if 
$k+l^2\ll \min\{m,n\}$. 
 
Bound (\ref{eqrmmrel}) is expected to hold 
for the output of the algorithm
if we bound the integers $k$ and $l$
by combining \cite[Theorems 4 and 5]{DMM08} as follows. 

%------------------------------------------------------------------------------

\begin{theorem}\label{thdmm} 
Suppose that 

(i) 
$W\in C^{m\times n}$,
 $\nrank (W)=r>0$,  
 $\epsilon,\beta,\bar \beta\in(0,1]$, and $\bar c$ is a sufficiently large constant,

(ii) four integers $k$, $k_-$,
 $l$, and
 $l_-$ satisfy the bounds
\begin{equation}\label{eq3200}
0<l_-=3200 r^2/(\epsilon^2\beta)\le l\le n~{\rm and}~
0<k_-=3200 l^2/(\epsilon^2\bar\beta)\le k\le m
\end{equation}
or 
\begin{equation}\label{eqlog}
l_-=\bar c~r\log(r)/(\epsilon^2\beta)\le l\le n
~{\rm and}~
k_-=\bar c~l\log(l)/(\epsilon^2\bar\beta)
\le k\le m,
\end{equation} 

 (iii) we apply Algorithm \ref{algcurlevsc} 
  invoking at stages 2 and 4  either
 Algorithm \ref{algsmplex} 
 under (\ref{eq3200})
or  Algorithm \ref{algsmplexp}
 under (\ref{eqlog}).

Then
%in both cases  
%nearly optimal 
bound (\ref{eqrmmrel})
holds
 with a probability at least 0.7. 
\end{theorem}
   
%------------------------------------------------------------------------------

\begin{remark}\label{reepsbet}
The bounds $k_-\le m$ and $l_-\le n$
imply that either 
$\epsilon^6\ge
3200^3 r^4/(m\beta^2\bar\beta)$
 and  $\epsilon^2\ge
 3200 r/(n\beta)$ if
 Algorithm \ref{algsmplex} is applied
 or
 $\epsilon^4\ge 
\bar c^2r\log(r)\log(\bar c r\log(r)/
(\epsilon^2\beta))/(m\beta^2\bar\beta)$
and $\epsilon^2\ge \bar c r\log(r)/(n\beta)$
if
 Algorithm \ref{algsmplexp} is applied
 for a sufficiently large constant $\bar c$.
\end{remark}
   
%------------------------------------------------------------------------------
%------------------------------------------------------------------------------

\begin{remark}\label{rebet}
The  estimates 
$k_-$ and $l_-$ of (\ref{eq3200}) and
(\ref{eqlog}) are
minimized for 
  $\beta=\bar\beta=1$ and a fixed $\epsilon$.
   By decreasing the values of $\beta$ and $\bar\beta$
 we increase these two estimates
by factors of  $1/\beta$  and
$1/(\beta^2\bar\beta)$, respectively, and
for any 
values of the leverage scores $p_i$ in the ranges (\ref{eq3200}) and
(\ref{eqlog}) we can  ensure 
 randomized error 
bound (\ref{eqrmmrel}). 
\end{remark}
      
The following result implies that the leverage scores
 are stable in perturbation of 
 an input matrix: 
    
\begin{theorem}\label{thsngspc} (See \cite[Theorem 8.6.5]{GL13}.) 
  Suppose that
  
   $$g=:\sigma_{r}(M)-\sigma_{r+1}(M)>0~ {\rm and}~||E||_F\le 0.2g.$$
   Then there exist unitary matrix bases 
   $B_{r,\rm left}(M)$, 
    $B_{r,\rm right}(M)$
    $ B_{r,\rm left}(M+E)$, 
    and
   $B_{r,\rm right}(M+E)$ for the singular spaces associated with the $r$ largest singular values of the matrices 
   $M$ and $M+E$, respectively,
   such that
$$\max\{||B_{r,\rm left}(M+E)-
B_{r,\rm left}(M)||_F,||B_{r,\rm right}(M+E)-B_{r,\rm right}(M)||_F\}\le 
4\frac{||E||_F}{g}.$$
      \end{theorem}

For example, if 
$\sigma_{r}(M)\gg \sigma_{r+1}(M)$, which
implies that $g\approx \sigma_{r}(M)$, then
the upper bound on the right-hand side is  
approximately
$4||E||_F/\sigma_r(M)$.

%------------------------------------------------------------------------------

\begin{remark}\label{resmplfc}
At stage 6 of Algorithm \ref{algcurlevsc} we can alternatively apply the simpler expressions   
$U:=(\bar S^T WS)^+=
(S^TC)^+=
(RS)^+$, although this would a little weaken 
 numerical stability of the computation of a  nucleus of a  
 perturbed input matrix $W$.
\end{remark}  

%------------------------------------------------------------------------------

\subsection{Superfast LRA with 
leverage scores for random and average inputs}\label{slrasmpav} 
 
%------------------------------------------------------------------------------

We are going to  
apply
Algorithm \ref{algcurlevsc} to  perturbed 
 factor-Gaussian  matrices and  their averages
 (see Definitions
\ref{defgsfc} and \ref{deaverg}).
We begin with some definitions and  
Theorem \ref{thsngvgs} of independent interest.
 
%------------------------------------------

\begin{definition}\label{defxi} 
$\chi^2(s):=\sum_{i=1}^sg_i^2$ is the 
$\chi^2$-function 
for $s$ i.i.d. Gaussian variables
$g_1,\dots,g_s$.
\end{definition} 
 
%------------------------------------------

\begin{definition}\label{defgspl} 
Let $n\ge r$ and let  $n$ scalars
 $p_1,\dots,p_n$
satisfy (\ref{eqlevsc}) for 
a scalar $\beta$ in the range
 $0<\beta\le 1$ and for
the matrix $V^{(r)T}$ of the top $r$
right singular vectors of 
a $r\times n$ Gaussian matrix.
Then call these $n$ scalars
$\mathcal G_{n,r,\beta}$-{\em leverage scores}.
\end{definition}

 \begin{theorem}\label{thsngvgs} 
Let $W=GH$ for 
$G\in \mathbb C^{m\times r}$
and $H\in \mathbb C^{r\times n}$ and let
$r=\rank(G)=\rank(H)\le \min\{m,n\}$.
Then the matrices $W^T$  and $W$
share their SVD-based leverage scores with the matrices $G^T$ 
 and $H$, respectively,
 \end{theorem}

\begin{proof}
Let   
$G=S_G\Sigma_GT^*_G\in 
\mathbb C^{m\times r}$ 
and $H=S_H\Sigma_HT^*_H$
be SVDs.
 
Write $M:=\Sigma_GT^*_GS_H\Sigma_H$
and let $M=S_M\Sigma_MT^*_M$ be SVD.

Notice that 
$\Sigma_G$, $T^*_G$, $S_H$ and $\Sigma_H$
are $r\times r$ matrices. 

Consequently so are 
the matrices  $M$,
 $S_M$, $\Sigma_M$ and $T^*_M$.
 
Hence $W=\bar S_G\Sigma_M\bar T^*_H$
where $\bar S_G=S_GS_M$ and $\bar T^*_H=T^*_MT^*_H$ are unitary
matrices.

Therefore $W=\bar S_G\Sigma_M\bar T^*_H$
is SVD.

Hence the columns of the unitary matrices 
$\bar S_G$ and $\bar T^{*T}_H$
span  the $r$ top right singular spaces of the
matrices $W^T$ and $W$, respectively,
and so do the columns of the matrices 
$S_G$ and $T^{*T}_H$ as well because 
$\bar S_G=S_GS_M$ and $\bar T^*_H=T^*_MT^*_H$ where $S_M$ and $T^*_M$
are $r\times r$ unitary matrices.
This proves the theorem.
\end{proof}

 \begin{corollary}\label{cosngvgs}
 Under the assumptions of 
 Theorem \ref{thsngvgs}  the SVD-based leverage scores 
 for the matrix $W$ are
 $\mathcal G_{n,r,\beta}$-leverage scores if $H$ is a Gaussian matrix,
 while the SVD-based leverage scores  
 for the matrix $W^T$ 
 are $\mathcal G_{m,r,\beta}$-leverage scores if $G$ is a Gaussian matrix.
 \end{corollary}
 
Next assume that $n\gg r$ and recall from 
\cite[Theorem 7.3]{E89} that 
$\kappa(W)\rightarrow 1$ as 
$r/n\rightarrow 0$ if 
$W\in \mathcal G^{r\times n}$.
Therefore  for $r\ll n$ a matrix
$W\in \mathcal G^{r\times n}$ 
is 
close to a scaled unitary matrix, 
and hence within a constant 
factor is close to the unitary matrix of 
its right singular space.

Now observe that the norms of the column vectors of such a matrix $W$
are  i.i.d. random variables $\chi^2(r)$
and therefore are quite 
strongly concentrated in a reasonable range about the expected value of such a variable. 
Hence  we obtain reasonably good approximations to SVD-based  leverage scores for such a matrix 
by choosing the norms ${\bf v}_{j}^{(r)}$ equal to each other for all $j$,
and then we satisfy expressions (\ref{eqlevsc})  and consequently
bound (\ref{eqrmmrel}) by choosing a reasonably small value $\beta$. Let us supply some details. 

\begin{lemma}(Cf. \cite[Lemma 1]{LM00}.)\label{lmchitail}
Let $Z = \sum_{i=1}^{r} X_i^2$, where $X_1, X_2\dots X_r$ are i.i.d. standard Gaussian variables. Then for every $x>0$
it holds that
$$\mathrm{Probability}\{Z-r \ge 2\sqrt{rx} + 2rx\}
\le \exp(-x).$$
\end{lemma}
 
\begin{corollary} 
Let $w_j \sim \chi^2(r)$ be independent chi-square random variables
with $r$ degrees of freedom, for $j = 1, 2, \dots, n$. Fix  $0<\beta\le 1$
and write $c := \frac{1-\sqrt{2\beta - \beta^2}}{2\beta}$.
Then 
$$\mathrm{Probability}\Big \{ \frac{1}{n}>
\beta~\frac{w_j}{nr} ~
\textrm{for all}~j = 1, 2, \dots, n\Big \}
\ge 1 - \exp\Big(-c~ \frac{\ln{n}}{r} \Big).$$
\end{corollary}

\begin{proof}
Deduce from Lemma \ref{lmchitail}  that
\begin{align*}
\mathrm{Probability}\Big\{ \frac{1}{n} >
\beta~\frac{w_j}{nr} \Big\}
&= 1 - \mathrm{Probability}\Big\{ w_j - r \ge
\Big(\frac{1}{\beta}-1\Big)r \Big\}\ge 1 - \exp{(-c/r)}
\end{align*}
for $c = \frac{1-\sqrt{2\beta - \beta^2}}{2\beta}$ and  any $j \in\{1,2,\dots, n\}$. 
Furthermore the random variables  $w_j$
 are independent of each other, and hence
\begin{align*}
\mathrm{Probability}\Big\{ \frac{1}{n}>
\beta~\frac{w_j}{nr} 
\textrm{ for all}~j = 1, 2,\dots, n\Big \}
&\ge \big( 1 - \exp(-c/r) \big)^n \\
&\ge 1 - \exp\Big(-c~ \frac{\ln{n}}{r} \Big).
\end{align*}
\end{proof}

%By combining the probability density %function  
%of the $\chi^2$-function with Markov %inequality one can specify our argument %quantitatively.

 In combination with Corollary  \ref{cosngvgs} this enables us to bypass
 the bottleneck stage of
 the computation of 
leverage scores for 
 a perturbed diagonally scaled $m\times n$
factor-Gaussian matrix and of a  perturbed $m\times n$ right 
factor-Gaussian matrix; in both cases we assume  that $r\ll n$. 

Now let such a matrix $W$ be in the class of  perturbed left or right factor-Gaussian matrices with expected rank $r\ll n$,
but suppose that we do not know in which of the two classes. Then by assuming that both matrices 
$W$ and $W^T$ are  perturbed right factor-Gaussian
and applying Algorithm \ref{algcurlevsc}
to both of them, 
we can 
compute their 
 leverage scores
  superfast. 
We would expect to obtain an
accurate CUR LRA for at least  
one of them, and  would
verify correctness of the output by  
  estimating a posteriori error norm
  superfast.

It is not clear how much our argument can be extended if we relax the assumption that $r\ll n$. Probably not much, although
some hopes can be based on the following well-known results, which
 imply  that   
 the norms of the rows of the basis matrix
   $V^{(r)T}$ for a right singular space
   of an $n\times r$ Gaussian matrix $W$ are uniformly distributed.
  
 \begin{theorem}\label{thsngvgs1} 
Let $M=GH$ for $G\in \mathbb C^{m\times r}$
and $H\in \mathbb C^{r\times n}$
for $r\le \min\{m,n\}$. Then

(i) the rows 
of the matrix $M$ and its left singular vectors
 are uniformly distributed on the unit sphere if $G$ is a Gaussian matrix and 

(ii) the columns
of the matrix $M$ 
and its right singular vectors are uniformly distributed on the unit sphere
if $H$ is a Gaussian matrix. 
\end{theorem}

\begin{proof}
We prove claim (ii); 
its application to the matrix $M^T$ yields
claim (i). 

Let  $V$ be an $n\times n$ unitary matrix.
By virtue of   
Lemma \ref{lepr3} the matrix $HV$ is  Gaussian. Therefore   
the entries of the matrices $M$ and $MV$ 
as well as of the matrices $U_M$ 
and $U_{MV}$ of their right singular 
spaces have the same probability distribution
pairwise. In particular
the distribution of the entries of the unitary matrix $U_M$ is orthogonally invariant, and therefore these entries are  uniformly distributed over the set of unitary matrices.
\end{proof}
%\begin{remark}\label{resprs3} 
%Our results of this subsection readily %follows from Lemma \ref{lmchitail} and %thus can be readily extended to the case %of a little perturbed sparse 
%factor-Gaussian matrices of small rank 
%and hence to the sparse average matrices %allowing LRA. 
%\end{remark}

%------------------------------------------------------------------------------
%------------------------------------------

\subsection{Refinement of an LRA by using  
leverage scores}\label{sitreflra}

%------------------------------------------
 
Next suppose that $AB=W'\approx W$ is a reasonably close but still crude  LRA  for a matrix $W$ allowing a closer LRA and let us compute it. First
 approximate top SVD of the matrix $W'$
 by  applying to it Algorithm \ref{alglratpsvd}, then fix a positive value 
$\beta'\le 1$ and compute leverage scores  $p_1',\dots,p_n'$  of $W'$ by applying
(\ref{eqlevsc}) with
$\beta'$  replacing $\beta$.
Perform all these computations superfast.

By virtue of Theorem \ref{thsngspc}
the computed  values 
$p_1',\dots,p_n'$  approximate 
 the leverage scores $p_1,\dots,p_n$ of the matrix $W$, and so  
we satisfy (\ref{eqlevsc})
for an input matrix $W$ and 
properly chosen parameters
$\beta$ and $\bar\beta$. Then we 
arrive at a CUR LRA satisfying (\ref{eqrmmrel}) if we sample 
 $k$ rows and $l$ columns 
 within the bounds of Theorem  \ref{thdmm},
 which are inversely proportional
 to $\beta^2\bar\beta$ and $\beta$, respectively. 
         
%------------------------------------------------------------------------------

\subsection{A fast CUR LRA algorithm that samples fewer rows}\label{smd09}

%------------------------------------------------------------------------------
  
By virtue of
 Theorem \ref{thdmm} 
Algorithm \ref{algcurlevsc} 
outputs nearly optimal CUR LRA,
but the supporting estimates for the integers   
$l$ and particularly $k$ are
fairly large,
even for relatively large values of 
$\epsilon\le 1$ (cf. Remark \ref{reepsbet}). Such estimates for the integers $l$ and $k$,
however, are overly pessimistic:
 in \cite[Section 7]{DMM08} 
 Algorithm \ref{algcurlevsc} has computed accurate CUR LRAs of various real world inputs
  by using just reasonably large integers $l$ and $k$.

The upper bounds of  Theorem  \ref{thdmm}
on the parameters $k$ and $l$ have been decreased in the subsequent papers.
Already the algorithm of the paper \cite{MD09} extends the results 
of \cite{DMM08} by sampling as many columns and much fewer rows.
Namely the algorithm 
computes CUR factor $C$ and $R$ by
applying  
stages 1 and 2 of Algorithm \ref{algcurlevsc} to the matrices $W$ and $W^T$, respectively,  and
then computes
 a nucleus 
\begin{equation}\label{equc+wr+}
 U=C^+WR^+.
\end{equation}
Theorem \ref{thdmm} is readily extended if
under the same bounds on the integer $l$ we  choose $k=l$ and increase by twice the value  $\epsilon$ and the bound of the probability of failure.

The proof relies on the following simple estimates of independent interest:
$$|W-CUR|\le |W-CC^+W|+
|CC^+W-CC^+WR^+R|\le 
|W-CC^+W|+|W-WR^+R|.$$
The latter
inequality  follows because the multiplier $CC^+$ is a  projection matrix,  and so multiplication by it 
does not increase the spectral and Frobenius  norms. 

The inequality implies the following extension of Theorem
\ref{thcandec}.

 \begin{theorem}\label{thuc+wr+}
 Suppose that (\ref{eqrnkcureq}) and 
 (\ref{equc+wr+}) hold.
 Then equation (\ref{eqcureq}) holds.
  \end{theorem}
 The 
 nucleus $U$ of (\ref{equc+wr+}) is not canonical, and its computation is not superfast. Empirically one can try to ignore 
 (\ref{equc+wr+}),  compute  superfast the canonical nucleus  $U$ of (\ref{eqnuc+svd}), 
  and then estimate the output error norm of the resulting CUR LRA superfast.  
  Unless the norm $||U|||$ is large,  for a large class of inputs
  this superfast variation of the algorithms of \cite{MD09} output accurate 
  CUR LRA, but the upper estimates  of \cite{DMM08} for the output  errors 
  are not extended. Indeed for a fixed pair of CUR factors $C$ and $R$
 the nuclei of (\ref{eqnuc+svd}) and (\ref{equc+wr+}) define the same CUR LRA
 for a matrix $W$ under (\ref{eqrnkcureq}), by virtue of 
 Theorem \ref{thuc+wr+}, 
 but the perturbation of this matrix 
  tends to  affect the pseudo inverse of a CUR generator 
  $W_{k,l}$ stronger
  than the more rectangular  factors $C$ and $R$. 
          
%------------------------------------------------------------------------------

\subsection{Further progress and research challenges}\label{srprpr} 

%------------------------------------------------------------------------------

 The papers 
 \cite{WZ13}, \cite{BW14}, and
  \cite{BW17} decrease by order of magnitude the asymptotic upper bounds of \cite{DMM08}
  on the parameters $k$ and $l$  
  in terms of $m$, $n$, and $r$, although
    the overhead 
  constants of these bounds
  are hidden in the "O" and "poly" notation.  
  Similar comments apply to the paper  \cite{SWZ17}, which studies LRA
  in terms of $l_1$-norm.
  
 It may be tempting to turn the algorithms \cite{BW17} and \cite{SWZ17} into superfast algorithms that compute accurate LRAs of random and average matrices allowing close LRA, but this seems to be hard because he algorithms \cite{BW17} and \cite{SWZ17} are quite involved.

 Here are some other natural research challenges for the  computation of  LRA by means of
 statistical approach  involving leverage scores
 (also see Section \ref{srsrc}):
 
 (i) Devise random sampling 
 algorithms that output  CUR
 generators with fewer columns and rows.
 
(ii) Estimate the output accuracy of LRA  algorithms of \cite{DMM08} and their variations applied to
 perturbed factor-Gaussian inputs  and the average inputs allowing LRA. As in our current study  use uniform leverage scores and try to extend our study in Section \ref{slrasmpav}  to the case where the ratio 
 $r/\min\{m,n\}$ is just reasonably small.
 
 (iii)  Complement the formal analysis
 of LRA algorithms
 with test results on real world input data. In particular test our approach
 to refinement of a crude but reasonably accurate LRAs.
 
(iv) Estimate and decrease the overhead constants hidden in the "O" and "poly" notation for the complexity bounds of  \cite{BW17} and \cite{SWZ17}.
For example, one can compute the \textit{Lewis weights}, involved in  
 \cite{SWZ17}, superfast
  in the cases of perturbed factor-Gaussian inputs, the average inputs 
  allowing LRA, and  the inputs given with their crude but reasonably close LRAs. 
Namely, for a matrix $V$, let 
${\bf v}_i^T$ denote its $i$th row vector
and define Lewis $l_p$-weights $x_{i,p}$ for $1\le p<4$
as follows:
$x_{i,p}^{2/p}:=
{\bf v}_i^T(VX_p^{1-2/p}V)^{-1}{\bf v}$
for all $i$ (see \cite[Section 3.5]{SWZ17}).
In particular 
$x_{i,2}:=
{\bf v}_i^T(V^TV)^{-1}{\bf v}$
for all $i$, and so the superfast algorithms
for SVD-based
leverage scores 
of random and average matrices and for
matrices given with their crude but close LRAs
are readily extended. 

Moreover, we can heuristically extend the computations from $l_2$-weights
to  $l_1$-weights. Indeed $x_{i,1}^2:=
{\bf v}_i^T(VX_1^{-1}V)^{-1}{\bf v})$
for  all $i$, and
 we can write
$x_{i,1,j+1}^2:=
{\bf v}_i^T(VX_{1,j}^{-1}V)^{-1}{\bf v}$  for all $i$ and $j=0,1,\dots$.
Having computed Lewis weights $x_{i,2}$,
substitute the
matrix $X_{1,0}:=X_2$  on the right-hand side of these expressions for $j=0$ and  update the
 approximate weights $x_{i,1,j}$ by  recursively increasing 
$j$ by 1 and hoping for fast
convergence. 
We can increase the chances for it by applying homotopy  continuation, that is, making transition from $l_2$ to $l_1$ in $q$ steps, each 
of the size $1/q$ for a fixed integer $q>1$. 
              
%------------------------------------------------------------------------------
   
\section{CUR LRA with Randomized Pre-processing}\label{sextgqg}
 
%------------------------------------------------------------------------------
%------------------------------------------------------------------------------

In this section we seek CUR LRA of any matrix allowing LRA and pre-processed with random multipliers.
 Algorithms \ref{alglratpsvd1} and \ref{algrhtcur}  of this section implement two distinct approaches to the computation 
of CUR LRA
with randomized multiplicative 
pre-processing of an input matrix
allowing its accurate LRA.
 
%------------------------------------------------------------------------------
%------------------------------------------------------------------------------
  
\subsection{Pre-processing into factor-Gaussian matrices: a theorem}\label{sextg}

%------------------------------------------------------------------------------
 
Next we prove that 
pre-processing with Gaussian multipliers turns any matrix of numerical rank $r$ into a perturbed 
 factor-Gaussian matrix, for which  
we can superfast compute  CUR generators
that satisfy any or all of the three CUR criteria and hence define accurate CUR LRAs
(see Sections \ref{ssfstcurprpr},  
\ref{serralg1},
 \ref{scndlrgv},  and \ref{slrasmpav}).

\begin{theorem}\label{thquasi} 
%Suppose that 
For $k$, $l$, $m$, $n$, and  $r$
satisfying (\ref{eqklmnr}),
$G\in \mathcal G^{l\times m}$, 
$H\in \mathcal G^{n\times k}$,
 an $m\times n$ well-conditioned matrix
 $W$ 
of rank $r$
and $\nu_{p,q}$
and $\nu_{p,q}^+$ of Definition \ref{defnrm} it holds that

(i) $GW$ is a left factor-Gaussian matrix of expected rank $r$,
$||GW||\le||W||~\nu_{k,r}$,
and $||(GW)^+||\le||W^+||~\nu_{k,r}^+$,

(ii) $WH$ is a right factor-Gaussian matrix of expected rank $r$,
$||WH||\le ||W||~\nu_{r,l}$, and 
$||(WH)^+||\le ||W^+||~\nu_{r,l}^+$,
 and

(iii) $GWH$ is a diagonally scaled 
factor-Gaussian matrix of expected rank 
$r$, $||GWH||\le||W||~\nu_{k,r}\nu_{r,l}$,
and $||(GWH)^+||\le||W^+||~\nu_{k,r}^+\nu_{r,l}^+$.
\end{theorem}
\begin{proof}
Let $W=S_W\Sigma_WT^*_W$ be SVD
where $\Sigma_W$ is  the diagonal matrix of the singular values of $W$;
it is well-conditioned since so is 
the matrix $W$.
Then 
 
(i) $GW=\bar G\Sigma_WT^*$,

(ii) $WH=S_W\Sigma_W\bar H$, and

(iii) $GWH=\bar G\Sigma_W\bar H$  \\
where $\bar G:=GS_W$ and $\bar H:=T_W^*H$
are Gaussian matrices by virtue of  
 Lemma \ref{lepr3}    
because $S_W\in \mathbb C^{m\times r}$ and $T_W\in \mathbb C^{r\times n}$ are unitary  matrices, $r\le mn\{m,n\}$,
 $G\in \mathcal G^{m\times m}$, and
 $H\in \mathcal G^{n\times n}$.
 Combine these equations with Lemma \ref{lehg}.                                                       \end{proof}
 
%------------------------------------------------------------------------------

\subsection{Computation of a CUR LRA with Gaussian  multipliers and their pseudo inverses}\label{sgmpsd} 
 
%------------------------------------------------------------------------------
%------------------------------------------------------------------------------

The latter results motivate application of the following randomized algorithm. (We state it assuming that $m\le n$, but by applying it to the transposed matrix $W^T$ we can cover the case where $m\ge n$.)
\begin{algorithm}\label{alglratpsvd1}
{\em CUR LRA using a Gaussian multiplier of   a large size and its pseudo inverse.}
 
%------------------------------------------------------------------------------

\begin{description}

%------------------------------------------------------------------------------

\item[{\sc Input:}]
A matrix
$W\in \mathbb C^{m\times n}$
having numerical rank  $r$ such that $r\ll m\le n$.

%------------------------------------------------------------------------------

\item[{\sc Output:}] Whp a CUR LRA of the matrix $W$.                                                                                                                                                    

%------------------------------------------------------------------------------

\item[{\sc Initialization:}]
Generate an $n\times u$ Gaussian
 matrix $H$
for $u\ge n$.

%-------------------------------------------------------------------------------

\item[{\sc Computations:}]
\begin{enumerate}
\item
 Compute the  $m\times u$ matrix $WH$.
 \item
Compute its
  SVD-based leverage scores 
 $p_1,\dots,p_u$  satisfying (\ref{eqlevsc}) for  $n=u$,
$||{\bf v}_{1}^{(r)}||=\cdots=
||{\bf v}_{u}^{(r)}||$,
 and a sufficiently small positive $\beta$.
 \item
Compute a CUR LRA of the matrix 
$WH=CUR_H$  by applying to it Algorithm    \ref{algcurlevsc} for these  
leverage scores. (By virtue of 
Theorem \ref{thquasi} the matrix $WH$
lies near a factor-Gaussian matrix having expected rank $r$.)
 \item
Compute a CUR LRA of the matrix 
$W$ defined by the factors $C$
and $U$ above and by the factor 
 $R=R_HH^+$.
\end{enumerate}
%------------------------------------------------------------------------------
\end{description}
%------------------------------------------------------------------------------
\end{algorithm}

%------------------------------------------------------------------------------

Correctness of the algorithm is readily verified.

Its computations
are superfast except for its stages 1 and 4. We can  dramatically accelerate them by replacing the  Gaussian multiplier $H$ by sparse and structured  multipliers 
of Sections \ref{sabrhf} -- \ref{sgpbd}, although one can recall that
for the worst case input
and even for the input family of $\pm\delta$-matrices no algorithm can be both superfast
and accurate.  

%------------------------------------------------------------------------------

\begin{remark}\label{rerfn} 
Having performed Algorithm \ref{alglratpsvd1}, one can  superfast
estimate a posteriori errors for the computed CUR LRA of $W$ and then
possibly apply 
 refinement of  Section \ref{sitreflra}.
\end{remark}

%------------------------------------------------------------------------------

\begin{remark}\label{re2orddmn}
 If $m>n$, then instead of 
 pre-processing by means of
  post-multiplication one should apply 
   pre-multiplication by Gaussian or random sparse and structured multipliers.
   For $m=n$ one can apply pre-processing 
   in any or both of these ways.
   In the extension of our study to $d$-dimensional  tensors one can consider 
   pre-processing by using at least up to $d$ directions for application of random multipliers.
 \end{remark}

%------------------------------------------------------------------------------

\subsection{Computation of a CUR LRA with Gaussian sampling}\label{ssmplrcv}

%------------------------------------------------------------------------------
 
\cite[Sections 4, 10 and 11]{HMT11} survey, refine, and thoroughly analyze 
various algorithms that compute LRA by using random rectangular  multipliers,
which quite typically are tall-skinny.
 We can extend these algorithms to the
  computation of a CUR LRA
 by applying superfast Algorithms \ref{alglratpsvd} and \ref{algsvdtocur}
  or \ref{algsvdtocur}a. Next, however, we describe
 a distinct extension, 
  which we empirically perform at  superfast level.
 At first, by applying random $n\times l$ multipliers $H_l$ (say, those of \cite{HMT11} or 
 \cite{PZ16}), we sample a  
  column set. Then we superfast compute
  a row set 
 $\mathcal I$ for a
CUR LRA of the  matrix $WH_l$ by applying the algorithms of \cite{GE96} or our Remark  \ref{recurgge}.
Finally we reuse this set as
the row index set in a CUR LRA for the
input matrix $W$. Here is a formal
description of these computations.

%------------------------------------------------------------------------------

\begin{algorithm}\label{algrhtcur} 
{\em Computation of a CUR LRA 
with Gaussian sampling.} 
 
%------------------------------------------------------------------------------

\begin{description}

%------------------------------------------------------------------------------

\item[{\sc Input:}]
Five integers $k$, $l$,
$m$, $n$, and $r$
such that $\max\{k,l\}\ll \min\{m,n\}$
and (\ref{eqklmnr}) holds and an $m\times n$ 
matrix $W$ of numerical rank $r$. 
 
%------------------------------------------------------------------------------

\item[{\sc Output:}]
A CUR generator
$W_{\mathcal I,\mathcal J}
\in \mathbb C^{k\times l}$. 

%------------------------------------------------------------------------------

\item[{\sc Initialization:}] 
 Fix a sufficiently large integer $\bar l$
 such that 
 $l\le \bar l\le n$;
 fix an
 $n\times \bar l$
  Gaussian matrix $\bar H$.
  
%-------------------------------------------------------------------------------

\item[{\sc Computations:}]

\begin{enumerate}
\item %1   
Compute  $m\times \bar l$
matrix $W\bar H$  (cf. \cite{HMT11}). 
\item %2
 By applying to it the algorithms of \cite{GE96} compute its $k\times \bar l$ submatrix
$(W\bar H)_{\mathcal I,:}$.
\item %3
Compute a $k\times l$ submatrix 
$W_{\mathcal I,\mathcal J}$
of the matrix 
$W_{\mathcal I,:}$ by applying to it the algorithms of \cite{GE96}.
Output this submatrix.
\end{enumerate}

%------------------------------------------------------------------------------

\end{description}

%------------------------------------------------------------------------------

\end{algorithm}

{\bf The complexity estimates:} 
We need $\mathcal NZ_W+\bar l n\le(m+\bar l)n$ memory cell and less than $2 \bar l\cdot \mathcal NZ_W\le 2nm \bar l$ flops at stage 1,
$m \bar l$ cells and  $O(m\bar l^2)$ flops
at stage  2,
and $kn$ cells and  $O(nk^2)$ flops
at stage  3. 
Thus stages 2 and 3 are superfast.

Corollary 
\ref{cowc1} together with the following result bound the output errors of the algorithm. 

%------------------------------------------------------------------------------

\begin{theorem}\label{thnrm+} 
For the output matrix
 $W_{\mathcal I,\mathcal J}$ of Algorithm
\ref{algrhtcur}, 
$\nu_{p,q}^+$ of Definition \ref{defnrm},  where  
$\mathbb E(\nu_{p,q}^+)<e\sqrt p/|q-p|$ (see Theorem \ref{thsiguna}),
and $t_{p,r,h}$ of (\ref{eqtkrh}), where we can assume that 
 $h\approx 1$,
it holds that
%\begin{equation}\label{eqnrm+}
$$||W_{\mathcal I,\mathcal J}^+||\le
||W^+||~
t_{m,k,h}~t_{n,l,h}~\nu_{k,\bar l}^+~\nu_{r,\bar l}.$$
%\end{equation}
 \end{theorem} 
\begin{proof} 
First deduce from Corollary \ref{cocndprt}
  that
$$||W_{\mathcal I,\mathcal J}^+||\le 
 ||W_{\mathcal I,:}^+||~t_{n,l,h}~{\rm
and}~ 
||((W\bar H)_{\mathcal I,:})^+||\le 
 ||(W\bar H)^+||~t_{m,k,h}.$$ Next  obtain that
  \begin{equation}\label{eqw_ij}
  ||W_{\mathcal I,\mathcal J}^+||\le 
||(W\bar H)^+||~t_{m,k,h}~t_{n,l,h}~\nu_{r,\bar l}
  \end{equation}
 by complementing the above bounds with the following inequality:
  \begin{equation}\label{eqw-+}
  ||W_{\mathcal I,:}^+||\le 
 ||((W\bar H)_{\mathcal I,:})^+||~\nu_{r,\bar l}.
 \end{equation} 
Let us prove it.
Let $W_{\mathcal I,:}=S\Sigma T^*$ be SVD. Then 
$||W_{\mathcal I,:}^+||=||\Sigma^+||$. Furthermore
 $$(W\bar H)_{\mathcal I,:}=
 W_{\mathcal I,:} \bar H=
 S \Sigma T^*\bar H.$$    
 The $r\times \bar l$  matrix
 $T^*\bar H$ is Gaussian by virtue of Lemma \ref{lepr3}
 because  the matrix $T^*$ is unitary and  the matrix $\bar H$ is Gaussian.
 
 Write $G_{r,\bar l}:=T^*\bar H$,
 obtain that $\Sigma=S^*W_{\mathcal I,:} \bar HG_{r,\bar l}^+$~, 
 and deduce that
 $$||W_{\mathcal I,:}^+||=
 ||\Sigma^+||=
 ||((W\bar H)_{\mathcal I,:}G_{r,\bar l}^+)^+||.$$ Now apply Lemma \ref{lehg} and
 deduce that
 $$||W_{\mathcal I,:}^+||\le
 ||(W\bar H)_{\mathcal I,:}^+||
 ~||G_{r,\bar l}||= ||(W\bar H)_{\mathcal I,:}^+|| ~
 \nu_{r,\bar l}~,$$ and (\ref{eqw-+}) follows.
Finally deduce from Theorem \ref{thquasi}
 that 
$$||(W\bar H)^+||\le 
||\bar H^+||~||W^+||\le
||W^+||~\nu_{k,\bar l}^+~,$$ 
substitute
 this bound into (\ref{eqw_ij}),  and obtain the theorem.
 \end{proof}

%------------------------------------------------------------------------------
 
\begin{remark}\label{rerfnhmt}
The input matrix at stage 2 and the transpose of the input matrix 
at stage 3 are tall-skinny, and so by applying randomized techniques of Remark
 \ref{recurgge} we can decrease the error
 bound of Theorem \ref{thnrm+} by
 almost a factor of $t_{m,k,h}t_{n,l,h}$;
this would also decrease the flop bounds at stage  2 to  $O(m\bar l)$ flops and
at stage  3 to  $O(nk)$.
Seeking a further heuristic acceleration we can apply C-A iterations at stage 3 instead of the algorithms of \cite{GE96}. 
\end{remark}

%------------------------------------------------------------------------------

\begin{remark}\label{rerfnhmtcyn}
We can extend our
 study to supporting cynical algorithms
 applied to $q\times  s$ sketches 
 for $k\le q\le m$
 and $l\le s\le n$.
 \end{remark}

%------------------------------------------------------------------------------
%------------------------------------------------------------------------------
  
\subsection{Randomized Hadamard and Fourier multipliers}\label{srftrht}
 
%------------------------------------------------------------------------------

We can accelerate the bottleneck stages 1 and 4 of Algorithm \ref{alglratpsvd1} and stage 1 of
Algorithm \ref{algrhtcur}
if instead of Gaussian multipliers
 we apply
 the matrices $H_l$ of randomized Hadamard and Fourier transforms of Appendix \ref{sahaf} or their SRHT and SRFT modifications.
The following theorem from  
\cite{T11} and \cite[Theorem 11.1]{HMT11}
implies that the resulting  matrix product  $WH_l$ is well-conditioned whp.

\begin{theorem}\label{thrhtrft} 
Suppose that $T_r$ is a $r\times n$ unitary matrix, $H_l$ is an $n\times l$ 
matrix of SRHT or SRFT, and

\begin{equation}\label{eql}
4(\sqrt r+\sqrt{8\log(rn)}~)^2\log(r)
\le  l\le n.
\end{equation}
 Then with a failure probability 
 in $O(1/k)$ it holds that
$$0.40\le \sigma_r(T_rH_l)~~{\rm and}~~
 \sigma_1(T_rH_l)\le 1.48.$$ 
 \end{theorem}

\begin{corollary}\label{corhtrft} 
Under the assumptions of
Theorem \ref{thrhtrft}, let 
$W_r$ be a $r\times n$  matrix 
with singular values 
$\sigma_j=\sigma_j(W_r)$, $j=1,\dots,r$. Let
$\Sigma_r=\diag(\sigma_j)_{j=1}^r$,
$\sigma_r>0$ and let an integer $l$ 
satisfy (\ref{eql}).
Then
\begin{equation}\label{eqrhtrft} 
0.40/\sigma_r\le \sigma_r(W_rH_l)~~{\rm and}~~
 \sigma_1(W_rH_l)\le 1.48\sigma_1 
\end{equation} 
 with failure probability in $O(1/k)$.
\end{corollary}

\begin{proof}
Consider SVD $W_r=S_r\Sigma_rT_r$ with
$\Sigma_r=\diag(\sigma_j)_{j=1}^r$ 
and unitary   matrices
 $S_r$ of size $r\times r$ and $T_r$ 
 of size $r\times n$.
Deduce the corollary from 
Theorem \ref{thrhtrft}.
\end{proof}

Corollary \ref{corhtrft}  implies that  SRHT and SRFT pre-processing are still expected to output a well-conditioned matrix $WH_l$, although a little less 
well-conditioned and with a higher failure probability than in the case of Gaussian 
 pre-processing.
 
Pre-processing $W\rightarrow WH_l$
performed with SRHT or SRFT multiplier
$H_l$
 involves $2n+l$ random parameters,
$mn$ memory cells, and $O(mn\log (l)
+kn\log (n))$ flops. The computation is not superfast but is noticeably 
faster than Gaussian
pre-processing, at least in the case of dense matrices $W$.
  
 Having the well-conditioned matrix $WH_l$
 computed,  we can compute CUR LRA superfast in the same way as in the previous subsection.

Although  
   Gaussian pre-processing is slower
   than SHRT and SRFT pre-processing, it enables us to satisfy any of the three CUR criteria and with a higher probability,
  while with SHRT and SRFT pre-processing
we can   only support
   the first CUR criterion.

%------------------------------------------------------------------------------
  
\subsection{Pre-processing with matrices of abridged Hadamard and Fourier transforms}\label{sabrhf} 

%------------------------------------------------------------------------------

 We proved that the output LRAs of randomized
 fast Algorithms \ref{alglratpsvd1} and \ref{algrhtcur} are accurate whp and remain accurate if   we replace Gaussian multipliers with SRHT or SRFT multipliers. This substitution 
  accelerates Algorithm \ref{algrhtcur}, but do not make it superfast. We can, however, make them superfast  by applying various random sparse and structured 
multipliers of \cite{PZ16} instead
of Gaussian ones.

In order to accelerate  
Algorithm \ref{alglratpsvd1} 
to   superfast level as well, we must also perform its stage 4 superfast, and  
   we can achieve this   
 if among the latter multipliers we choose unitary ones, having their pseudo inverses  given by their Hermitian transposes.  
In particular we can use the unitary matrices of {\em  abridged randomized Hadamard and Fourier
transforms (ARHT and ARFT)}
of \cite{PZ16}, which we recall in Appendix \ref{sahaf}.

 $2^h\times 2^h$ 
Hadamard and Fourier matrices
are  generated
in 
$h$  recursive steps.
We  successively apply pre-processing with such multipliers generated in $j$ steps for $j=1,2,3,\dots$ until we succeed with CUR LRA. In the worst case we would perform $\log_2(n)$ recursive steps and arrive at fast SRHT or SRFT 
pre-processing, but in our tests with primitive, cynical,
 and C-A algorithms
we have consistently 
computed accurate CUR LRAs superfast -- by using 
ARHT and ARFT multipliers computed within 
a small number, typically at most three recursive steps. This can be well understood because, as we proved, the class of all hard inputs for these algorithms is narrow, and so randomized pre-processing of their inputs allowing LRA can hardly land in that class.

%------------------------------------------------------------------------------
  
\subsection{Pre-processing with 
quasi Gaussian matrices}\label{sqsgss}

%------------------------------------------------------------------------------
      
Generation, multiplication, and inversion of Gaussian matrices
are quite expensive. 
Towards 
simplification of these operations we
consider SRHT, SRFT, ARHT, and ARFT 
multipliers but also
(i) sparse 
Gaussian multipliers in 
$\mathcal G_{\mathcal NZ}$
and (ii) incompletely factored
Gaussian multipliers $G$.
In order to strengthen the chances for success of these heuristic recipes we can
apply multipliers on both sides of the input matrix (cf. Remark \ref{rerfn}).

Case (i).  We state this recipe  as  heuristic, but Theorem \ref{thNZnrm+}   indicates some potentials for its analytic support.  

Case (ii). Decompose a Gaussian matrix $G$ into a product of 
factors that can be readily inverted.

Our next decomposition of this kind and its nontrival formal support may be of independent interest:
we prove that the probability distribution of the product of random and randomly permuted 
cyclic bidiagonal matrices converges to a Gaussian matrix as the number of factors grows to  the infinity.

We proceed as follows: generate 
random  multipliers of the form 
$B_iP_i$ where
$B_i$  are bidiagonal matrices 
with ones on the diagonal and $\pm 1$ 
on the first subdiagonal for random choice of the signs $\pm$,
$P_i$ are random permutation matrices,
  $i=1,2,\dots,q$, and $q$ is the
  minimal integer for which our selected  superfast algorithm outputs accurate
  LRA of  the matrix 
  $WG_q:=W\prod_{i=1}^qB_iP_i$. 

Empirical convergence to Gaussian distribution is quite fast (see Section \ref{sfig}). 
A matrix that approximates a Gaussian matrix must be  dense, and using it  as a multiplier
would be expensive, but we can do
much  better by keeping it factored.
We can  hope 
to satisfy one, two, or all three
CUR criteria already for inputs pre-processed with the products of a  smaller number of bidiagonal factors such that their generation and application is still superfast.

%------------------------------------------------------------------------------

\subsection{Approximation of a Gaussian matrix by the product of  random bidiagonal 
 matrices}\label{sgpbd}

%------------------------------------------------------------------------------

Suppose that $n$ is a positive integer, $P$ is a random permutation matrix,
and define  $n\times n$ matrix 

\begin{equation}
\label{def11}
B: = \left[
\begin{array}{ccccc}
1 & & & & \pm 1\\
\pm 1& 1\\
	& \pm 1 & 1 \\
	&		& \cdots & \cdots\\
	&		&		 & \pm 1 & 1\\

\end{array} 
\right] P
\end{equation}
where each $\pm 1$ represents an independent Bernoulli random variable. 

\begin{theorem}
\label{thm1}
Let $B_0$, ..., $B_T$ be independent random matrices of the form \eqref{def11}. 
As $T\rightarrow \infty$, the  distributions of the matrices $G_T := \prod_{t=1}^{T}B_t$ 
converge to the  distribution of a Gaussian  matrix.
\end{theorem}

For demonstration of the approach  first consider its simplified version. 
Let

\begin{equation}
\label{def2}
A := \left[
\begin{array}{ccccc}
1/2 & & & & 1/2\\
 1/2& 1/2\\
	&  1/2 & 1/2 \\
	&		& \cdots & \cdots\\
	&		&		 &  & 1/2\\

\end{array}
\right] P
\end{equation}
where $P$ is a random permutation matrix. We are going to prove the following theorem.

\begin{theorem}
\label{thm2}
Let $A_0$, ..., $A_T$ be independent random matrices defined by \eqref{def2}. 
As $T\rightarrow \infty$, the  distributions of the matrices
$\Pi_T := \prod_{t=1}^{T}A_t$ converge to the distribution
of the matrix $$\left[
\begin{array}{cccc} 
1/n &\cdots &\cdots  & 1/n\\
\vdots &\ddots &\ddots & \vdots \\ 
\vdots &\ddots &\ddots & \vdots \\
1/n &\cdots &\cdots  & 1/n\\
\end{array}
\right].$$
\end{theorem}

\begin{proof}
First examine the effect of multiplying by such  a random matrix $A_i$:
let $$M :=( {\bf m}_1~ |~{\bf m}_2~ |~ ... ~ |~{\bf m}_n)$$ and let
the permutation matrix of $A_i$ defines a column permutation 
$\sigma: \{1,...,n\}\rightarrow \{1,...,n\}$. Then 
%\begin{equation}
$$MA_i = \Big(\frac{{\bf m}_{\sigma(1)} + {\bf m}_{\sigma(2)}}{2}~\Big|~
\frac{{\bf m}_{\sigma(2)} + {\bf m}_{\sigma(3)}}{2}~\Big|~\dots~
\Big|~\frac{{\bf m}_{\sigma(n)} + {\bf m}_{\sigma(1)}}{2}\Big).
$$
%\end{equation}
Here each new column is written as a linear combination of $\{{\bf m}_1, ..., {\bf m}_n\}$. More generally, if we consider $MA_0 A_1 \cdots A_t$, i.e., $M$ multiplied with $t$ matrices of the form \eqref{def2}, each new column has the following linear expression:
%\begin{equation}
$$MA_0 A_1 \cdots A_t = \Big(\sum_k \pi^t_{k1}{\bf m}_k~\Big|~
\sum_k \pi^t_{k2}{\bf m}_k~\Big|~\dots~\Big|~\sum_k \pi^t_{kn}{\bf m}_k\Big).
$$
%\end{equation}
Here $\pi^t_{ij}$ is the coefficient of ${\bf m}_i$ in the linear expression of the $j$th column of the product matrix.
Represent the column permutation defined by matrix of $A_t$ by the map 
$$\sigma_t: \{1,...,n\}\rightarrow \{1,...,n\}$$ 
and readily verify the following lemma.
\begin{lemma}
\label{lem1} It holds that
$\sum_j \pi^t_{ij} = 1~{\rm for~all}~i$
and
%\begin{equation}
$$\pi^{t+1}_{ij} = \frac{1}{2}(\pi^{t}_{i\sigma_{t}(j)} + \pi^{t}_{i\sigma_{t}(j+1)})~{\rm for~all~pairs~of}
~i~{\rm and}~j
%\end{equation}
%\begin{equation}
.$$
%\end{equation}
\end{lemma}

Next we prove the following result.
\begin{lemma}
\label{lem2}
For any triple of $i,j$ and $\epsilon>0$,
%\begin{equation} 
$$\lim_{T\rightarrow\infty} {\rm Probability}\Big(\Big|\pi^T_{ij} - \frac{1}{n}\Big| > \epsilon\Big) = 0.$$
%\end{equation}
\end{lemma}

{\it Proof.}
%Firstly notice that this is a monotonically decreasing sequence, since
%\begin{align*}
%Pr(|\pi^{t+1}_{ij} - 1|>\epsilon)
%&< Pr(\frac{1}{2}|\pi^{t}_{i\sigma(j)} - 1|>\frac{1}{2}\epsilon and 
%\frac{1}{2}|\pi^{t}_{i\sigma(j+1)} - 1|>\frac{1}{2}\epsilon)\\
%&< Pr(\frac{1}{2}|\pi^{t}_{i\sigma(j)} - 1|>\frac{1}{2}\epsilon) +
%Pr(\frac{1}{2}|\pi^{t}_{i\sigma(j+1)} - 1|>\frac{1}{2}\epsilon)
%\end{align*}
Fix $i$, define 
%\begin{equation}
$$\mathcal{F}^t := \sum_{j} \Big(\pi^t_{ij} - \frac{1}{n}\Big)^2.$$
%\end{equation}
Then 
\begin{align*}
\mathcal{F}^{t+1} - \mathcal{F}^t
&= \sum_{j} \Big(\pi^{t+1}_{ij} - \frac{1}{n}\Big)^2 - \sum_{j} \Big(\pi^t_{ij} - \frac{1}{n}\Big)^2\\
&= \sum_{j} \Big[ \Big(\Big(\frac{\pi^t_{i\sigma(j)}+\pi^t_{i\sigma(j+1)}}{2}) - \frac{1}{n}\Big)^2 - 
\frac{1}{2}\Big(\pi^t_{i\sigma(j)}-\frac{1}{n}\Big)^2 - \frac{1}{2}\Big(\pi^t_{i\sigma(j+1)}-\frac{1}{n}\Big)^2 \Big]\\
&= \sum_j \Big[ \Big(\frac{\pi^t_{i\sigma(j)}+\pi^t_{i\sigma(j+1)}}{2}\Big)^2 - \frac{1}{2}(\pi^t_{i\sigma(j)})^2 - \frac{1}{2}(\pi^t_{i\sigma(j+1)})^2 \Big]\\
&\leq \sum_j -\frac{1}{4}\big[ (\pi^t_{i\sigma(j)})^2 - 2\pi^t_{i\sigma(j)}\pi^t_{i\sigma(j+1)} + (\pi^t_{i\sigma(j+1)})^2\big]\\
&= -\frac{1}{4}\sum_j \big( \pi^t_{i\sigma(j)} - \pi^t_{i\sigma(j+1)}\big)^2\\
&\leq -\frac{1}{4n}\big(\sum_i |\pi^t_{i\sigma(j)} - \pi^t_{i\sigma(j+1)}|\big)^2\\
&= -\frac{1}{4n}\big(\pi^t_{max} - \pi^t_{min} \big)^2.
\end{align*}
Here $\pi^t_{max}: = \max_i\{\pi^t_{ij}\}$ and $\pi^t_{min}: = \min_i\{\pi^t_{ij}\}$.

Furthermore, since $\pi^t_{max}\ge \pi^t_{ij}, \forall j$ and $\frac{1}{n}\ge \pi^t_{min}\ge 0$, it follows that
\begin{align*}
\mathcal{F}^t
= \sum_j (\pi^t_{ij} - \frac{1}{n})^2
\le n(\pi^t_{max} - \pi^t_{min})^2.
\end{align*}

Therefore
%\begin{equation}
%\aligned
$$\mathcal{F}^{t+1} - \mathcal{F}^t 
\le  -\frac{1}{4n}\big(\pi^t_{max} - \pi^t_{min} \big)^2
\le -\frac{1}{4n^2}\mathcal{F}^t.$$
%\endaligned
%\end{equation}

Now our monotone decreasing sequence has the only stationary value 
 when all values $\pi^t_{ij}$ coincide with each other. 
Together with Lemma \ref{lem2} this implies
%\begin{equation}
$$\lim_{T\rightarrow\infty}  {\rm  Probability}\Big(\Big|\pi^T_{ij} - \frac{1}{n}\Big | > \epsilon\Big) = 0.$$
%\end{equation}
\end{proof}

Next we prove Theorem \ref{lem1}.
\begin{proof}
Let $S^t_{i}$ denote the values $\pm 1$ on each row. By definition 
 $S^t_{i'}$ and $S^t_i$ are independent for $i\neq i'$. 
Moreover the following lemma can be  readily verified.

\begin{lemma}
$\prod_{t=1}^T S^t_{i_t}$ and $\prod_{t=1}^T S^t_{i'_t}$ are independent 
as long as there is at least one index $t$ such that $i_t\neq i'_t$.
\end{lemma}
 
Write
%\begin{equation}
$$MB_0 B_1 \cdots B_t = 
\Big(\sum_k \gamma^t_{k1}{\bf m}_k~\Big|~\sum_k \gamma^t_{k2}{\bf m}_k~
\Big|~.\dots~\Big|~\sum_k \gamma^t_{kn}{\bf m}_k\Big)$$
%\end{equation}
and notice that each $\gamma^T_{ij}$ can be written as a sum of random values $\pm 1$ whose signs 
are determined by $\prod_{t=1}^T S^t_{i_t}$. Since different signs are independent, 
we can represent $\gamma^T_{ij}$ as the  difference of two positive integers $\alpha - \beta$
 whose sum is $2^T\pi^T_{ij}$. 

Theorem \ref{thm2} implies that
the sequence $\pi^T_{ij}$ converges to $\frac{1}{n}$ almost surely as $T\rightarrow \infty$.

Therefore $\gamma^T_{ij}/2^T$ converges to Gaussian distribution as $T\rightarrow\infty$. 
Together with independence of the random values $\gamma^T_{ij}$ for all pairs $i$ and $j$,
this implies that eventually the entire matrix 
 converges to a Gaussian  matrix (with i.i.d. entries).
\end{proof}
 
 The speed of the convergence to Gaussian distribution is determined by the speed of
 the convergence (i) of the values  
 $\pi^T_{ij}$  to 1 as 
 $T\rightarrow\infty$
and (ii) of the binomial distribution with the mean $\pi^t_{ij}$ 
 to the Gaussian distribution. For (i), we have the  following estimate: 
\begin{align*}
|\pi^t_{ij} - 1| \le& \mathcal{F}^t
\le \Big(1-\frac{1}{4n^2}\Big)^{t-1}\mathcal{F}^0;
\end{align*}
and for (ii) we have the following Berry--Esseen theorem (cf. \cite{B41}).
\begin{theorem}
Let $X_1$, $X_2$, $\dots$ be independent random variables with $E(X_i)=0$, $E(X_i^2)=\sigma_i^2>0$ and $E(|X_i|^3)=\rho_i<\infty$ for all $i$. Furthermore  let
\begin{align*}
S_n: = \frac{X_1 + \cdots + X_n}{\sqrt{\sigma_1^2 + \sigma_2^2 +\cdots + \sigma_n^2}}
\end{align*}
be a normalized $n$-th partial sum. 
Let $F_n$ and  $\Phi$ denote the cumulative distribution functions of $S_n$
  a Gaussian variable, respectively. 
Then for some constant $c$ and for all $n$, 

\begin{align*}
\sup_{x\in\mathbb{R}}|F_n(x) - \Phi(x)| \le c 
%\psi
%\end{align*}
%~~{\rm where}~~ 
%\begin{align*}
%\psi: = 
\cdot (\sum_{i=1}^{n}\sigma_i^2)^{-3/2} \max_{1\le i\le n}\rho_i.
\end{align*}
\end{theorem}

In our case, for any fixed pair of $i$ and $j$, write $N_t = 2^t\pi_{ij}^t$ and $\gamma_{ij}^t  = X_1 + \cdots + X_{N_t}$ where $X_i$ are i.i.d. $\pm 1$ variables. Then 
$E(X_i)=0$, $E(X_i^2)=1/4$,  $E(X_i^3)=1/8$, and 
\begin{align*}
&S_{N_t}=\frac{X_1 + \cdots + X_{N_t}}{\sqrt{\sigma_1^2 + \sigma_2^2 +\cdots + \sigma_{N_t}^2}}
%\\
= \frac{\gamma_{ij}^t}{\sqrt{N_t/4}}
%\\
= \frac{\sqrt{\pi_{ij}^t}\gamma_{ij}^t}{2^{t/2-1}}.
\end{align*}
Furthermore
\begin{align*}
\sup_{x\in\mathbb{R}}|F_N(x) - \Phi(x)| \le&  c\cdot\Big(\sum_{i=1}^{N_t} \sigma_i^2 \Big)^{-3/2} \max_{1\le i\le N_t}\rho_i\\
\le& \frac{1}{8}~c\cdot \Big (\frac{N_t}{4}\Big )^{-3/2} ~\rightarrow 0 ~{\rm as}~t\rightarrow \infty.
\end{align*}

\medskip
\medskip
\medskip

\noindent {\bf \Large PART IV.  EXTENSIONS, TESTS, AND SUMMARY}

%------------------------------------------------------------------------------

%------------------------------------------------------------------------------
\section{Some Implications and Extensions of Our Study}\label{sextpr} 
 
%------------------------------------------------------------------------------

Our progress can be extended to numerous important computational problems
linked to LRA. We refer the reader to
the variety of applications of fast LRA 
 in \cite{HMT11}, \cite{M11},  \cite{W14}, \cite{CBSW14}, \cite{BW17},
\cite{KS16}, \cite{SWZ17}, and the references 
therein, but more applications
should be inspired by our study.
 For example,
 computation of Tucker's, High Order SVD, and TT                                                                                                                                                                                                                                                                                                                                                                                                                                                                                                                                 decompositions of a tensor can be
reduced to the computation of LRA
of some associated matrices \cite{OST08}, \cite{MMD08}, \cite{CC10},
\cite{OT10}, \cite{VL16},  but one can also extend our analysis of superfast algorithms for random and average inputs (see \cite{LPSZa}). 

In this section we
accelerate the Fast Multipole Method (FMM)
  to supefast level.
  
%------------------------------------------------------------------------------

%\subsection{Superfast Multipole Method}%\label{ssfmm}

%------------------------------------------------------------------------------
%Namely 

FMM has been devised for  superfast multiplication by a vector of a 
special structured matrix,
called  HSS matrix,
provided that low rank generators are available for its off-diagonal blocks.
 Such generators are not available in some important applications, however (see, e.g, \cite{XXG12}, 
 \cite{XXCB14}, and \cite{P15}),
 and then their computation 
 by means of the known algorithms
 is not superfast. According to our study  C-A and some other algorithms  perform
 this stage superfast on the average input, thus turning FMM into {\em Superfast 
 Multipole Method}.
 
  Since the method is highly important
 we next include some details of this acceleration.
We first recall that
 HSS matrices
 naturally extend the class of banded matrices and their inverses, 
are closely linked to FMM,  
 and 
are increasingly popular
(see \cite{BGH03},  \cite{GH03}, \cite{MRT05},
 \cite{CGS07},
 \cite{VVGM05}, 
\cite{VVM07/08}, 
 \cite{B10}, \cite{X12},   \cite{XXG12},
\cite{EGH13}, \cite{X13},
 \cite{XXCB14},
  and the bibliography therein). 

\begin{definition}\label{defneut} (See 
\cite{MRT05}.) With each diagonal block of a block matrix 
associate 
its complement in its block column,
and call this complement a {\em neutered block column}.
\end{definition}

\begin{definition}\label{defqs} (See 
% \cite{GR87}, \cite{CGR88}, \cite{T00}, 
% \cite{BGH03},  \cite{GH03}, \cite{B10},
\cite{CGS07},
 \cite{X12},  \cite{X13}, \cite{XXCB14}.)
A block 
matrix $M$ of size  $m\times n$ is 
called a $r$-{\em HSS matrix}, for a positive integer $r$, 

(i) if all diagonal blocks of this matrix
consist of $O((m+n)r)$ entries overall
and
 
(ii) if $r$ is the maximal rank of its neutered block columns. 
\end{definition}

\begin{remark}\label{reqs}
Many authors work with $(l,u)$-HSS
(rather than $r$-HSS) matrices $M$ for which $l$ and $u$
are the maximal ranks of the sub- and super-diagonal blocks,
respectively.
The $(l,u)$-HSS and $r$-HSS matrices are closely related. 
If a neutered block column $N$
is the union of a sub-diagonal block $B_-$ and 
a super-diagonal block $B_+$,
then
 $\rank (N)\le \rank (B_-)+\rank (B_+)$,
 and so
an $(l,u)$-HSS matrix is a $r$-HSS matrix,
for $r\le l+u$,
while clearly a $r$-HSS matrix is  
a $(r,r)$-HSS matrix.
\end{remark}

The FMM exploits the $r$-HSS structure of a matrix as follows:

(i) Cover all off-block-diagonal entries
with a set of  non-overlapping neutered block columns.  

(ii) Express every neutered block column $N$ of this set
  as the product
  $FH$ of two  {\em generator
matrices}, $F$ of size $h\times r$
and $H$ of size $r\times k$. Call the 
pair $\{F,H\}$ a {\em length $r$ generator} of the 
neutered block column $N$. 

(iii)  Multiply 
the matrix $M$ by a vector by separately multiplying generators
and diagonal blocks by subvectors,  involving $O((m+n)r)$ flops
overall, and

(iv) in a more advanced application of  
 FMM solve a nonsingular $r$-HSS linear system of $n$
equations  by using
$O(nr\log^2(n))$ flops under some mild additional assumptions on  the input. 

This approach is readily extended to the same operations with 
$(r,\xi)$-{\em HSS matrices},
that is, matrices approximated by $r$-HSS matrices
within a perturbation norm bound $\xi$ where a  positive tolerance 
$\xi$ is small in context (for example, is the unit round-off).
 Likewise, one defines
an  $(r,\xi)$-{\em HSS representation} and 
$(r,\xi)$-{\em generators}.

$(r,\xi)$-HSS matrices (for $r$ small in context)
appear routinely in matrix computations,
and computations with such matrices are 
performed  efficiently by using the
above techniques.

%------------------------------------------------------------------------------

In some applications of the FMM (see  \cite{BGP05}, \cite{VVVF10})
stage (ii) is omitted because short generators for all 
neutered block columns are readily available,
 but this is not the case in a variety of other  important applications
 (see \cite{XXG12}, \cite{XXCB14}, and \cite{P15}). 
This stage of the computation of  generators is precisely 
 LRA of the neutered block
columns, which turns out to be 
the bottleneck stage of FMM in these applications, and superfast LRA algorithms  
provide a remedy. 

Indeed apply a fast  algorithm at this
 stage, e.g., the algorithm of \cite{HMT11}
 with a Gaussian multiplier.
Multiplication of a $q\times h$ matrix
by an $h\times r$ Gaussian matrix requires $(2h-1)qr$ flops,
while
standard HSS-representation of an $n\times n$ 
HSS matrix includes $q\times h$ neutered 
 block columns for $q\approx m/2$ and $h\approx n/2$. In this case 
the cost of computing an $r$-HSS representation of the
matrix $M$ is at least of order $mnr$.
For $r\ll \min\{m,n\}$, this
is much  greater than 
 $O((m+n)r^2)$ flops, used  at the other stages of 
the computations. 

Alternative customary techniques for LRA
rely on computing SVD
or rank-revealing factorization of an input matrix 
and are at least as costly as the computations by means of random sampling.

Can we alleviate such a problem?  
Yes, heuristically we can 
compute LRAs 
to  
$(r,\xi)$-generators
superfast
by applying C-A iterations or other superfast LRA algorithms that we studied and we can recall our formal support for the accuracy of the computed CUR LRAs 
for the average input
and whp for
a perturbed factor-Gaussian input.

%------------------------------------------------------------------------------

\section{Numerical Experiments}\label{sexp}
 
%------------------------------------------------------------------------------

\subsection{Test Overview}

%------------------------------------------------------------------------------

We tested CUR LRA algorithms on random input matrices and benchmark matrices of discretized Integral and Partial  Differential Equations (PDEs), running
 the tests   
 in the Graduate Center of the City University of New York 
 %on a Dell computer with the Intel Core 2 %2.50 GHz processor and 4G memory running 
%Windows 7 and 
and using  MATLAB. We applied its standard normal distribution function "randn()" in order to
generate Gaussian matrices and	 calculated 
 numerical ranks of the input matrices
 by using the MATLAB's function 
 "rank(-,1e-6)",
 which only counts singular values greater than $10^{-6}$. 
  
 Our tables display the mean
 spectral norm of the relative 
 output error over 1000 runs for every class of inputs
 as well as the standard deviation (std).
 
 In Section \ref{ststlvrg}  we present test results for the approximation of SVD-based leverage scores for selected matrices by the leverage scores
  for their LRAs. 
 
%------------------------------------------------------------------------------

\subsection{Four algorithms used}

%------------------------------------------------------------------------------

 In our tests 
we compared the following four algorithms
for computing CUR LRAs to input matrices $W$ having numerical rank $r$:
\begin{itemize}
\item
{\bf Tests 1 (a primitive algorithm):}
Randomly choose two index sets $\mathcal{I}$ and $\mathcal{J}$, both of cardinality $r$, then compute a nucleus 
$U=W_{\mathcal{I}, \mathcal{J}}^{-1}$ and define a CUR LRA
\begin{equation}\label{eqtsts1}
W':=CUR=W_{:, \mathcal{J}} \cdot  W_{\mathcal{I}, \mathcal{J}}^{-1} \cdot W_{\mathcal{I},\cdot}.
\end{equation}
%\item   
%{\bf Method 2}
%{\color{red} We need to drop this method, correct?}
%
%randomly choose index sets $\mathcal{I}$ and $\mathcal{J}$, then construct CUR decomposition from LRAs of $A$ with respect to $A_{\mathcal{I},\cdot}$ and $A_{:, \mathcal{J}}$. More specifically, compute compact QR factorizations
\item    
{\bf Tests 2 (Five loops of  C-A):}
Randomly choose an  initial row index set 
$\mathcal{I}_0$ of cardinality $r$, then perform five loops of C-A
  by applying Algorithm 1 of \cite{P00}
 as a subalgorithm  that
  produces $r\times r$ CUR generators.
  At the end compute a nucleus $U$ and define a 
CUR LRA  as
in Tests 1.
\end{itemize} 
 
\begin{itemize}
\item
{\bf Tests 3 (Cynical algorithm, for $k=l=4r$):}
Randomly choose a row index set $\mathcal{K}$ and a column index set $\mathcal{L}$,
both of cardinality $4r$, and then apply Algs. 1 and 2 from \cite{P00}
to compute a $r\times r$ submatrix $W_{\mathcal{I}, \mathcal{J}}$ of $W_{\mathcal{K}, \mathcal{L}}$  having locally maximal volume. Compute a nucleus and obtain a CUR LRA by applying 
equation (\ref{eqtsts1}).

\item
{\bf Tests 4 (Combination of a single 
 C-A loop with Tests 3):}
Randomly choose a  column index set 
$\mathcal{L}$
of cardinality $4r$; then
perform a single C-A
loop
(made up of  a single horizontal  step and  a single vertical step): First by applying Alg. 1 from \cite{P00} find an index set $\mathcal{K}'$ of cardinality $4r$ such that  $W_{\mathcal{K}', \mathcal{L}}$ has locally maximal volume in  $W_{:, \mathcal{L}}$, then by applying this algorithm to matrix  $W_{\mathcal{K}',:}$ find an index set 
$\mathcal{L}'$ of cardinality $4r$ such that  $W_{\mathcal{K}', \mathcal{L}'}$ has locally maximal volume in $W_{\mathcal{K}',:}$. Then proceed as in Tests 3 -- find an $r\times r$ submatrix $W_{\mathcal{I}, \mathcal{J}}$ having locally maximal volume in $W_{\mathcal{K}', \mathcal{L}'}$, compute a nucleus, and define a CUR LRA.
\end{itemize}

\subsection{CUR LRAs of random input matrices}
In the tests of this subsection 
we  
% set $\epsilon = 10^{-10}$ and 
computed CUR LRAs with random 
row- and column-selection for
a perturbed $n\times n$
factor-Gaussian matrices with expected 
 rank $r$, i.e., matrices $W$ in the form
%\begin{equation}
$$W = G_1 * G_2 + 10^{-10} G_3,$$
%\end{equation}
for three Gaussian matrices $G_1$
of size $n\times r$, $G_2$ of size $r\times n$,
and  $G_3$
 of size $n\times n$. 
%(Recall from Definition \ref{defgsfc} %that we call Gaussian a random matrix %whose entries are i.i.d. standard %Gaussian variables.)
Table \ref{tb_ranrc} shows the test results for all four test algorithms for $n =256, 512, 1024$ and $r = 8, 16, 32$. 
%The results suggest that Method 2 has smaller approximation error than Method 1, while Method 3 further decreases the computational error.
 
\begin{table}[ht]

\begin{center}
\begin{tabular}{|c c|c c|c c|c c|c c|}
\hline     
 &  &\multicolumn{2}{c|}{\bf Tests 1} & \multicolumn{2}{c|}{\bf Tests 2}&\multicolumn{2}{c|}{\bf Tests 3} & \multicolumn{2}{c|}{\bf Tests 4}\\ \hline
{\bf n} & {\bf r} & {\bf mean} & {\bf std}  & {\bf mean} & {\bf std}& {\bf mean} & {\bf std} & {\bf mean} & {\bf std}\\ \hline

%64 & 8 & 1.69e-06 & 1.99e-05 & 4.56e-07 & 1.14e-05\\ \hline
%64 & 16 & 2.85e-06 & 2.14e-05 & 1.31e-07 & 8.30e-07\\ \hline
%64 & 32 & 1.15e-04 & 3.11e-03 & 2.49e-07 & 2.42e-06\\ \hline
%128 & 8 & 5.10e-06 & 5.79e-05 & 1.73e-06 & 4.66e-05\\ \hline
%128 & 16 & 6.04e-06 & 5.10e-05 & 2.54e-07 & 1.26e-06\\ \hline
%128 & 32 & 1.70e-05 & 1.45e-04 & 4.56e-07 & 3.28e-06\\ \hline
256 & 8 & 1.51e-05 & 1.40e-04 & 5.39e-07 & 5.31e-06& 8.15e-06 & 6.11e-05 & 8.58e-06 & 1.12e-04\\ \hline
256 & 16 & 5.22e-05 & 8.49e-04 & 5.06e-07 & 1.38e-06& 1.52e-05 & 8.86e-05 & 1.38e-05 & 7.71e-05\\ \hline
256 & 32 & 2.86e-05 & 3.03e-04 & 1.29e-06 & 1.30e-05& 4.39e-05 & 3.22e-04 & 1.22e-04 & 9.30e-04\\ \hline
512 & 8 & 1.47e-05 & 1.36e-04 & 3.64e-06 & 8.56e-05&2.04e-05 & 2.77e-04 & 1.54e-05 & 7.43e-05\\ \hline
512 & 16 & 3.44e-05 & 3.96e-04 & 8.51e-06 & 1.92e-04&  2.46e-05 & 1.29e-04 & 1.92e-05 & 7.14e-05\\ \hline
512 & 32 & 8.83e-05 & 1.41e-03 & 2.27e-06 & 1.55e-05& 9.06e-05 & 1.06e-03 & 2.14e-05 & 3.98e-05\\ \hline
1024 & 8 & 3.11e-05 & 2.00e-04 & 4.21e-06 & 5.79e-05& 3.64e-05 & 2.06e-04 & 1.49e-04 & 1.34e-03\\ \hline
1024 & 16 & 1.60e-04 & 3.87e-03 & 4.57e-06 & 3.55e-05& 1.72e-04 & 3.54e-03 & 4.34e-05 & 1.11e-04\\ \hline
1024 & 32 & 1.72e-04 & 1.89e-03 & 3.20e-06 & 1.09e-05& 1.78e-04 & 1.68e-03 & 1.43e-04 & 6.51e-04\\ \hline
\end{tabular}
\caption{CUR LRA of random input matrices}
\label{tb_ranrc}
\end{center}
\end{table}

% - - - - - - - - - - - - - - - - - - - - - - - - - - - - - - - - - - - - -

\subsection{CUR LRA of   matrices of discretized Integral Equations}\label{sinteq}
Table  \ref{LowRkTest2} displays the results of Tests 2  applied to  
 $1,000\times 1,000$
 matrices from the Singular Matrix Database
of the San Jose 
University. (Tests 1   
 produced much less accurate 
CUR LRAs for the same input sets, and we do not display their results.)  We  
have tested dense  matrices with smaller ratios of "numerical rank/$\min(m, n)$"  from the built-in test problems in Regularization 
Tools.\footnote{See 
%database at 
 http://www.math.sjsu.edu/singular/matrices and 
  http://www2.imm.dtu.dk/$\sim$pch/Regutools 
  
For more details see Chapter 4 of the Regularization Tools Manual at \\
  http://www.imm.dtu.dk/$\sim$pcha/Regutools/RTv4manual.pdf } 
The matrices came from discretization (based on Galerkin or quadrature methods) of the Fredholm  Integral Equations of the first kind.

 We applied our tests to the following six input classes  from the Database:

\medskip

{\em baart:}       Fredholm Integral Equation of the first kind,

{\em shaw:}        one-dimensional image restoration model,
 
{\em gravity:}     1-D gravity surveying model problem,
 
%heat:        inverse heat equation.

%parallax:    Stellar parallax problem with 28 fixed, real observations.

%tomo:        a 2D tomography test problem. 

%ursell:      integral equation with no square integrable solution.

wing:        problem with a discontinuous
 solution,

{\em foxgood:}     severely ill-posed problem,
 
{\em inverse Laplace:}   inverse Laplace transformation.

%\medskip

% We tested three choices of the number 
% $r$ for row- and column-sampling:
%$$r: = \textit{numerical rank}-2, +0, +2.%$$ 
 
 \begin{table}[ht]
\begin{center}
\begin{tabular}{|c|c c|c c|}
\hline
 &  &  & \multicolumn{2}{c|}{\bf Tests 2}\\ \hline
{\bf Inputs}&{\bf m}& ${\bf r}$  & {\bf mean} & {\bf std}\\ \hline

\multirow{3}*{baart}
&1000 & 4 & 1.69e-04 & 2.63e-06 \\ \cline{2-5}
&1000 & 6 & 1.94e-07 & 3.57e-09 \\ \cline{2-5} 
&1000 & 8 & 2.42e-09 & 9.03e-10 \\ \hline
\multirow{3}*{shaw}
&1000 & 10 & 9.75e-06 & 3.12e-07 \\ \cline{2-5} 
&1000 & 12 & 3.02e-07 & 6.84e-09 \\ \cline{2-5} 
&1000 & 14 & 5.25e-09 & 3.02e-10 \\ \hline 
\multirow{3}*{gravity}
&1000& 23 & 1.32e-06 & 6.47e-07 \\ \cline{2-5} 
&1000 & 25 & 3.35e-07 & 1.97e-07 \\ \cline{2-5} 
&1000 & 27 & 9.08e-08 & 5.73e-08 \\ \hline
\multirow{3}*{wing}
&1000 & 2 & 9.23e-03 & 1.46e-04 \\ \cline{2-5} 
&1000 & 4 & 1.92e-06 & 8.78e-09 \\ \cline{2-5} 
&1000 & 6 &   8.24e-10 & 9.79e-11\\ \hline 
\multirow{3}*{foxgood}
&1000 & 8 & 2.54e-05 & 7.33e-06 \\ \cline{2-5} 
&1000 & 10 & 7.25e-06 & 1.09e-06 \\ \cline{2-5} 
&1000 & 12 & 1.57e-06 & 4.59e-07 \\ \hline 
\multirow{3}*{inverse Laplace}
&1000 & 23 & 1.04e-06 & 2.85e-07 \\ \cline{2-5} 
&1000 & 25 & 2.40e-07 & 6.88e-08 \\ \cline{2-5} 
&1000 & 27 & 5.53e-08 & 2.00e-08 \\ \hline 
\end{tabular}
\caption{CUR LRA of benchmark input matrices of discretized Integral Equations from the  San Jose University Singular Matrix Database}
\label{LowRkTest2}
\end{center}
\end{table}

% - - - - - - - - - - - - - - - - - - - - - - - - - - - - - - - - - - - - -

%\subsection{Tests for inputs generated via the discretization of a Laplacian operator
%and via the approximation of an inverse finite-difference operator}
\subsection{Testing benchmark input matrices with bidiagonal pre-processing}
\label{ststslo}

Tables 
\ref{tb_Lp}
 and \ref{tb_lr} 
display the  results of Tests 1, 3, and 4 applied
 to pre-processed matrices  of two kinds,
both replicated from \cite{HMT11}, namely, the matrices of
 
(i) the discretized single-layer Laplacian operator (see Table 3) and 

(ii) the approximation of the inverse of a finite-difference operator (see Table 4).

Application of Tests 1, 3, and 4
to these matrices without pre-processing tended to produce results with large errors, and so we pre-processed
   every input matrix by
   multiplying it by 20  matrices,
each obtained by means of
random column permutations 
of a random bidiagonal matrix (see Section \ref{sgpbd}).
%of (\ref{def11}). 
%For each setting we applied 10 different %random multipliers, and 
%then again we display the mean value and the %standard deviation of the results.
  
{\em Input matrices (i).} We considered the Laplacian operator from  \cite[Section 7.1]{HMT11}:  
$$[S\sigma](x): = \psi\int_{\Gamma_1}\log{|x-y|}\sigma(y)dy,~{\rm for}~x\in\Gamma_2,$$
for two circles $\Gamma_1: = C(0,1)$ and $\Gamma_2: = C(0,2)$  on the complex plane
with the center at the origin and radii 1 and 2, respectively.
Its dscretization defines an $n\times n$ matrix $W=(w_{ij})_{i,j=1}^n$  
for
$w_{i,j}: = \psi\int_{\Gamma_{1,j}}\log|2\omega^i-y|dy$,
a constant $\psi$ such that  $||W||=1$, and
the arc $\Gamma_{1,j}$  of the contour $\Gamma_1$ defined by
the angles in $[\frac{2j\pi}{n},\frac{2(j+1)\pi}{n}]$.

\begin{table}[ht]
\label{LowRktest3}
\begin{center}  
\begin{tabular}{|c c|c c|c c|}
\hline
 &  &\multicolumn{2}{c|}{\bf Tests 1}& \multicolumn{2}{c|}{\bf Tests 4}\\ \hline
{\bf n} & {\bf r}  & {\bf mean} & {\bf std} & {\bf mean} & {\bf std}\\ \hline

256 & 31 & 1.37e-04 & 2.43e-04& 9.46e-05 & 2.11e-04\\ \hline
256 & 35 & 5.45e-05 & 7.11e-05& 1.03e-05 & 1.08e-05\\ \hline
256 & 39 & 6.18e-06 & 6.32e-06& 1.24e-06 & 1.72e-06\\ \hline
512 & 31 & 7.80e-05 & 6.00e-05& 2.04e-05 & 1.52e-05\\ \hline
512 & 35 & 1.56e-04 & 1.53e-04& 6.74e-05 & 1.79e-04\\ \hline
512 & 39 & 5.91e-05 & 1.10e-04& 4.27e-05 & 1.20e-04\\ \hline
1024 & 31 & 9.91e-05 & 6.69e-05& 2.79e-05 & 3.13e-05\\ \hline
1024 & 35 & 4.87e-05 & 4.35e-05& 1.66e-05 & 1.50e-05\\ \hline
1024 & 39 & 6.11e-05 & 1.33e-04& 3.83e-06 & 5.77e-06\\ \hline
\end{tabular}
\caption{CUR LRA of Laplacian input matrices}
\label{tb_Lp}
\end{center}
\end{table}

{\em Inputs matrices (ii).} We similarly applied our
Tests 1, 3, and 4
with pre-processing to the input matrices $W$ 
being the inverses of large sparse matrices representing a finite-difference operator
 from  \cite[Section 7.2]{HMT11}, 
and we observed  similar results
with all structured  and Gaussian multipliers.

\begin{table}[ht]

\begin{center}
\begin{tabular}{|c c|c c|c c|c c|}
\hline
 &  &\multicolumn{2}{c|}{\bf Tests 1} & \multicolumn{2}{c|}{\bf Tests 3}& \multicolumn{2}{c|}{\bf Tests 4}\\ \hline
{\bf n} & {\bf r} & {\bf mean} & {\bf std} & {\bf mean} & {\bf std} & {\bf mean} & {\bf std}\\ \hline

800 & 78 & 4.85e-03 & 4.25e-03& 3.30e-03 & 8.95e-03& 3.71e-05 & 3.27e-05\\ \hline
800 & 82 & 2.67e-03 & 3.08e-03& 4.62e-04 & 6.12e-04& 2.23e-05 & 2.24e-05\\ \hline
800 & 86 & 2.14e-03 & 1.29e-03& 4.13e-04 & 8.45e-04& 6.73e-05 & 9.37e-05\\ \hline
1600 & 111 & 1.66e-01 & 4.71e-01& 1.11e-03 & 1.96e-03& 1.21e-04 & 1.17e-04\\ \hline
1600 & 115 & 3.75e-03 & 3.18e-03& 1.96e-03 & 3.93e-03& 4.03e-05 & 2.79e-05\\ \hline
1600 & 119 & 3.54e-03 & 2.27e-03& 5.56e-04 & 7.65e-04& 5.38e-05 & 8.49e-05\\ \hline
3200 & 152 & 1.87e-03 & 1.37e-03& 3.23e-03 & 3.12e-03& 1.68e-04 & 2.30e-04\\ \hline
3200 & 156 & 1.92e-03 & 8.61e-04& 1.66e-03 & 1.65e-03& 1.86e-04 & 1.17e-04\\ \hline
3200 & 160 & 2.43e-03 & 2.00e-03& 1.98e-03 & 3.32e-03& 1.35e-04 & 1.57e-04\\ \hline
\end{tabular}
\caption{CUR LRA of finite difference matrices}
\label{tb_lr}
\end{center}
\end{table}

\subsection{Tests with abridged  randomized Hadamard and Fourier pre-pro\-cess\-ing}
\label{tabahaf}

Table \ref{tb_fh} displays the results of our Tests 2 for  CUR LRA with abridged randomized
Hadamard and Fourier pre-processing (see Appendix \ref{sahaf} for definitions). We used the same input matrices as in previous two subsections. 
For these input matrices Tests 1 have no longer output stable accurate LRA. 

\begin{table}[ht]

\begin{center}
\begin{tabular}{|c |c c c|c c|c c|}
\hline
Multipliers & & & &\multicolumn{2}{c|}{\bf Hadamard} & \multicolumn{2}{c|}{\bf  Fourier}\\ \hline
{Input Matrix} &  {\bf m} & {\bf n} & {\bf r} & {\bf mean} & {\bf std}  & {\bf mean} & {\bf std}\\ \hline
gravity &   1000 & 1000 & 25 & 2.72e-07 & 3.95e-08 & 2.78e-07 & 4.06e-08 \\ \hline
wing  &   1000 & 1000& 4 & 1.22e-06 & 1.89e-08 & 1.22e-06 & 2.15e-08 \\ \hline
foxgood & 1000 & 1000& 10 & 4.49e-06 & 6.04e-07 & 4.50e-06 & 5.17e-07 \\ \hline
shaw &   1000 & 1000& 12 & 3.92e-07 & 2.88e-08 & 3.91e-07 & 2.98e-08 \\ \hline
bart & 1000 & 1000& 6 & 1.49e-07 & 1.37e-08 & 1.49e-07 & 1.33e-08 \\ \hline
inverse Laplace & 1000 & 1000& 25 & 3.62e-07 & 1.00e-07 & 3.45e-07 & 8.64e-08 \\ \hline
\multirow{3}*{Laplacian}  
&256 & 256 & 15 & 4.08e-03 & 1.14e-03 & 3.94e-03 & 5.21e-04 \\ \cline{2-8} 
&512 & 512 & 15 & 3.77e-03 & 1.34e-03 & 4.28e-03 & 6.07e-04 \\ \cline{2-8} 
&1024 & 1024 & 15 & 3.97e-03 & 1.22e-03 & 4.09e-03 & 4.47e-04 \\ \hline
\multirow{3}*{finite difference}  
&408 & 800 & 41 & 4.50e-03 & 1.12e-03 & 3.76e-03 & 8.36e-04 \\ \cline{2-8} 
&808 & 1600 & 59 & 4.01e-03 & 1.10e-03 &   3.80e-03 & 1.70e-03 \\ \cline{2-8} 
&1608 & 3200 & 80 & 4.60e-03 & 1.53e-03 & 3.85e-03 & 1.27e-03\\ \hline
\end{tabular}
\caption{Tests 2 for CUR LRA with ARFT/ARHT pre-processors}
\label{tb_fh}
\end{center}
\end{table} 
\subsection{Randomized factorization of Gaussian  matrices}\label{sfig}
 
We have tested multiplication of twenty inverse-bidiagonal matrices (with random column permutation). 
Figure \ref{LowRkTest1} shows the distribution of a single randomly chosen entry
 and the scattered plot of two such entries. The output $1024\times 1024$ matrices
were very close to Gaussian distribution and
have passed the Kolmogorov-Smirnov test for normality in all  tests repeated 1000 times.

\begin{figure}[htb] 
\centering
\includegraphics[scale=0.7] {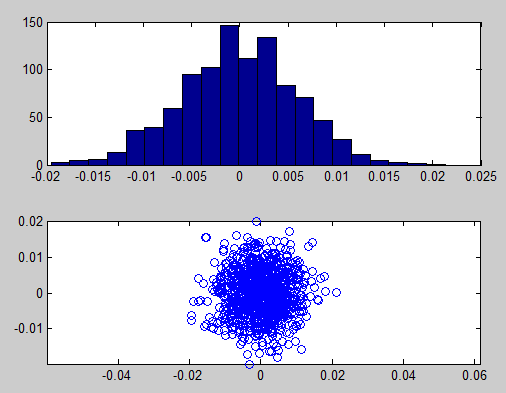}
\caption{Distribution of a randomly chosen entry}
\label{LowRkTest1}
\end{figure}

%------------------------------------------------------------------------------

\subsection{Testing C-A acceleration of the algorithms of \cite{DMM08}}
\label{ststc-alvrg}

%------------------------------------------------------------------------------
 
 Tables 7
 %\ref{tabcadmm1} 
 and  
 %\ref{tabcadmm}
 8 display the results of our tests  where we performed eight C-A iterations for the input matrices of Sections \ref{sinteq} and \ref{ststslo} by applying the  Algorithm 1 of \cite{DMM08} to all vertical and horizontal sketches
 (see the lines marked ``C-A") and for comparison computed LRA of the same matrices by applying to them Algorithm 2
 of \cite{DMM08} (see the lines marked ``CUR"). The columns of the tables marked with "nrank" display the numerical rank of an input matrix.
 The columns of the tables  marked with "$k=l$" show the number of rows and  columns in a square matrix of CUR generator.

%------------------------------------------------------------------------------

\subsection{Testing perturbation of leverage scores}
\label{ststlvrg}

%------------------------------------------------------------------------------

Table \ref{tb:lscore}
%\ref{lscore} 
shows the means and standard deviations of 
the norms of the relative errors of approximation of the input matrix $W$ and of its LRA $AB$ and similar data for the maximum difference between the SVD-based leverage scores of the pairs of these matrices. We also include  numerical ranks of the input matrices 
$W$ defined up to  tolerance  
$10^{-6}$.

In these tests we reused 
input matrices $W$ and their approximations $AB$ from our tests in Section \ref{sinteq} (using the Singular Matrix Database of San Jose University).

In addition, the last three lines of Table \ref{tb:lscore}
%\ref{lscore}
 show similar results for  perturbed
 diagonally scaled factor-Gaussian matrices $GH$ with expected numerical rank $r$
approximating input matrices $W$ up to small norm perturbations.

\begin{table}[ht]\label{lscore}
\centering
\begin{tabular}{|c|c c|c c|c c|}
\hline
 & & &\multicolumn{2}{c|}{LRA Rel Error} & \multicolumn{2}{c|}{Leverage Score Error}\\\hline
Input Matrix &	r &	rank &	mean &	std &	mean &	std\\\hline

baart &	4 &	6 &	6.57e-04 &	1.17e-03 &	1.57e-05 &	5.81e-05\\\hline
baart &	6 &	6 &	7.25e-07 &	9.32e-07 &	5.10e-06 &	3.32e-05\\\hline
baart &	8 &	6 &	7.74e-10 &	2.05e-09 &	1.15e-06 &	3.70e-06\\\hline
foxgood &	8 &	10 &	5.48e-05 &	5.70e-05 &	7.89e-03 &	7.04e-03\\\hline
foxgood &	10 &	10 &	9.09e-06 &	8.45e-06 &	1.06e-02 &	6.71e-03\\\hline
foxgood &	12 &	10 &	1.85e-06 &	1.68e-06 &	5.60e-03 &	3.42e-03\\\hline
gravity &	23 &	25 &	3.27e-06 &	1.82e-06 &	4.02e-04 &	3.30e-04\\\hline
gravity &	25 &	25 &	8.69e-07 &	7.03e-07 &	4.49e-04 &	3.24e-04\\\hline
gravity &	27 &	25 &	2.59e-07 &	2.88e-07 &	4.64e-04 &	3.61e-04\\\hline
laplace &	23 &	25 &	2.45e-05 &	9.40e-05 &	4.85e-04 &	3.03e-04\\\hline
laplace &	25 &	25 &	3.73e-06 &	1.30e-05 &	4.47e-04 &	2.78e-04\\\hline
laplace &	27 &	25 &	1.30e-06 &	4.67e-06 &	3.57e-04 &	2.24e-04\\\hline
shaw &	10 &	12 &	6.40e-05 &	1.16e-04 &	2.80e-04 &	5.17e-04\\\hline
shaw &	12 &	12 &	1.61e-06 &	1.60e-06 &	2.10e-04 &	2.70e-04\\\hline
shaw &	14 &	12 &	4.11e-08 &	1.00e-07 &	9.24e-05 &	2.01e-04\\\hline
wing &	2 &	4 &	1.99e-02 &	3.25e-02 &	5.17e-05 &	2.07e-04\\\hline
wing &	4 &	4 &	7.75e-06 &	1.59e-05 &	7.17e-06 &	2.30e-05\\\hline
wing &	6 &	4 &	2.57e-09 &	1.15e-08 &	9.84e-06 &	5.52e-05\\\hline
factor-Gaussian &	25 &	25 &	1.61e-05 &	3.19e-05 &	4.05e-08 &	8.34e-08\\\hline
factor-Gaussian &	50 &	50 &	2.29e-05 &	7.56e-05 &	2.88e-08 &	6.82e-08\\\hline
factor-Gaussian &	75 &	75 &	4.55e-05 &	1.90e-04 &	1.97e-08 &	2.67e-08\\\hline

\end{tabular}
\caption{Tests for the perturbation of leverage scores}
\label{tb:lscore}
\end{table}

%------------------------------------------------------------------------------

\begin{table}[ht]\label{tabcadmm1}
\centering
\begin{tabular}{|c|c|c|c|c|c|c|c|}
\hline
input & algorithm & m & n & nrank & k=l & mean & std \\ \hline
finite diff & C-A & 608 & 1200 & 94 & 376  & 6.74e-05 & 2.16e-05 \\ \hline
finite diff & CUR & 608 & 1200 & 94 & 376  & 6.68e-05 & 2.27e-05 \\ \hline
finite diff & C-A & 608 & 1200 & 94 & 188  & 1.42e-02 & 6.03e-02 \\ \hline
finite diff & CUR & 608 & 1200 & 94 & 188  & 1.95e-03 & 5.07e-03 \\ \hline
finite diff & C-A & 608 & 1200 & 94 & 94   & 3.21e+01 & 9.86e+01 \\ \hline
finite diff & CUR & 608 & 1200 & 94 & 94   & 3.42e+00 & 7.50e+00 \\ \hline

baart & C-A & 1000 & 1000 & 6 & 24 & 2.17e-03 & 6.46e-04 \\ \hline
baart & CUR & 1000 & 1000 & 6 & 24 & 1.98e-03 & 5.88e-04 \\ \hline
baart & C-A & 1000 & 1000 & 6 & 12 & 2.05e-03 & 1.71e-03 \\ \hline
baart & CUR & 1000 & 1000 & 6 & 12 & 1.26e-03 & 8.31e-04 \\ \hline
baart & C-A & 1000 & 1000 & 6 & 6  & 6.69e-05 & 2.72e-04 \\ \hline
baart & CUR & 1000 & 1000 & 6 & 6  & 9.33e-06 & 1.85e-05 \\ \hline

shaw & C-A & 1000 & 1000 & 12 & 48 & 7.16e-05 & 5.42e-05 \\ \hline
shaw & CUR & 1000 & 1000 & 12 & 48 & 5.73e-05 & 2.09e-05 \\ \hline
shaw & C-A & 1000 & 1000 & 12 & 24 & 6.11e-04 & 7.29e-04 \\ \hline
shaw & CUR & 1000 & 1000 & 12 & 24 & 2.62e-04 & 3.21e-04 \\ \hline
shaw & C-A & 1000 & 1000 & 12 & 12 & 6.13e-03 & 3.72e-02 \\ \hline
shaw & CUR & 1000 & 1000 & 12 & 12 & 2.22e-04 & 3.96e-04 \\ \hline
\end{tabular}
\caption{LRA errors of Cross-Approximation (C-A) tests  incorporating 
\cite[Algorithm 1]{DMM08} in comparison to stand-alone CUR tests of 
\cite[Algorithm 2]{DMM08} (for  three input classes from Sections \ref{sinteq} and \ref{ststslo}).}
\end{table}

\begin{table}[ht]\label{tabcadmm}
\centering
\begin{tabular}{|c|c|c|c|c|c|c|c|}
\hline
input & algorithm & m & n & nrank & $k=l$ & mean & std \\ \hline
foxgood & C-A & 1000 & 1000 & 10 & 40 & 3.05e-04 & 2.21e-04 \\\hline
foxgood & CUR & 1000 & 1000 & 10 & 40 & 2.39e-04 & 1.92e-04 \\ \hline
foxgood & C-A & 1000 & 1000 & 10 & 20 & 1.11e-02 & 4.28e-02 \\ \hline
foxgood & CUR & 1000 & 1000 & 10 & 20 & 1.87e-04 & 4.62e-04 \\ \hline
foxgood & C-A & 1000 & 1000 & 10 & 10 & 1.13e+02 & 1.11e+03 \\ \hline
foxgood & CUR & 1000 & 1000 & 10 & 10 & 6.07e-03 & 4.37e-02 \\ \hline
wing & C-A & 1000 & 1000 & 4 & 16 & 3.51e-04 & 7.76e-04 \\ \hline
wing & CUR & 1000 & 1000 & 4 & 16 & 2.47e-04 & 6.12e-04 \\ \hline
wing & C-A & 1000 & 1000 & 4 & 8  & 8.17e-04 & 1.82e-03 \\ \hline
wing & CUR & 1000 & 1000 & 4 & 8  & 2.43e-04 & 6.94e-04 \\ \hline
wing & C-A & 1000 & 1000 & 4 & 4  & 5.81e-05 & 1.28e-04 \\ \hline
wing & CUR & 1000 & 1000 & 4 & 4  & 1.48e-05 & 1.40e-05 \\ \hline

gravity & C-A & 1000 & 1000 & 25 & 100 & 1.14e-04 & 3.68e-05 \\ \hline
gravity & CUR & 1000 & 1000 & 25 & 100 & 1.41e-04 & 4.07e-05 \\ \hline
gravity & C-A & 1000 & 1000 & 25 & 50  & 7.86e-04 & 4.97e-03 \\ \hline
gravity & CUR & 1000 & 1000 & 25 & 50  & 2.22e-04 & 1.28e-04 \\ \hline
gravity & C-A & 1000 & 1000 & 25 & 25  & 4.01e+01 & 2.80e+02 \\ \hline
gravity & CUR & 1000 & 1000 & 25 & 25  & 4.14e-02 & 1.29e-01 \\ \hline
inverse Laplace & C-A & 1000 & 1000 & 25 & 100 & 4.15e-04 & 1.91e-03 \\ \hline
inverse Laplace & CUR & 1000 & 1000 & 25 & 100 & 5.54e-05 & 2.68e-05 \\ \hline
inverse Laplace & C-A & 1000 & 1000 & 25 & 50  & 3.67e-01 & 2.67e+00 \\ \hline
inverse Laplace & CUR & 1000 & 1000 & 25 & 50 &  2.35e-02 & 1.71e-01 \\ \hline
inverse Laplace & C-A & 1000 & 1000 & 25 & 25 &  7.56e+02 & 5.58e+03 \\ \hline
inverse Laplace & CUR & 1000 & 1000 & 25 & 25 &  1.26e+03 & 9.17e+03 \\ \hline
\end{tabular}
\caption{LRA errors of Cross-Approximation (C-A) tests  incorporating 
\cite[Algorithm 1]{DMM08}  in comparison to stand-alone CUR tests of 
\cite[Algorithm 2]{DMM08} (for four input classes Section \ref{sinteq}).} 
\end{table}

%------------------------------------------------------------------------------
%------------------------------------------------------------------------------
 
\section{Accuracy and Complexity of LRA Algorithms}\label{sacccmpl}    
  
%------------------------------------------------------------------------------
%---------------------------------------------- 
    
Next we survey the known complexity 
and accuracy estimates for LRA and then
display them in Table \ref{tb:complexity}.
See more in \cite[Section 2.1.2]{HMT11},
\cite{KS16},
 \cite{BW17} 
(in particular in
\cite[Table 2.1]{BW17}), \cite{SWZ17},
and the references therein.
The algorithms cited in that table
involve all entries of an input matrix and thus are not superfast, except for C-A iterations of \cite{GOSTZ10} and our Alg. \ref{algrhtcur}.

%------------------------------------------------------------------------------

\subsection{A brief survey of the accuracy and complexity estimates}\label{sacccmplsrv}

Classical algorithms compute
an {\em optimal 
LRA} of an $m\times n$  matrix given by its   
SVD-truncation. This deterministic computation
involves $mn$ memory cells and
 order of $(m+n)mn$  
flops.\footnote{We cite the estimate for the arithmetic cost of computing SVD 
from \cite[page 493]{GL13}.}

More recent  algorithms, based on
   QRP (orthogonal) or LUP (triangular)
 {\em rank-revealing factorization} of a  matrix, compute its LRA within a
   factor $f$ from optimal for  $f^2\le t_{p,r,h}$
  and $f\le t_{p,r,h}$, respectively,
 \begin{equation}\label{eqtkrh}
  t_{p,r,h}:=\sqrt{(p-r)r~h^2+1},
  \end{equation}
  $p=\max\{m,n\}$, and a real $h$ a little exceeding 1; one can assign  $h=1.1$ or $h=1.01$, say.
 For both QRP and LUP approaches the computations are deterministic and
 involve $O(mn\min\{m,n\})$ flops.
 See  \cite{GE96}, \cite{P00}, \cite[Section 5.4.6]{GL13}, and our Appendix \ref{scgrbcalg}.
  
  Researchers in Computer Science (CS) introduced various randomized LRA algorithms,
many of which have been
  surveyed, analyzed, and extended in
   \cite[Sections 10]{HMT11}. In particular one can first
 fix  oversampling  parameter $p=l-r$
 such that $4\le p\le n-r$, then  
 generate $nl$ i.i.d.
Gaussian variables, define
 an $n\times l$  Gaussian multiplier
 $H_l$, and successively
 compute the matrix $WH_l$, its  orthogonalization $A=Q$, and the matrix 
 $B=Q^*W$. With a probability close to 1
 this defines an LRA $AB$ of
 the matrix $W$ up to a factor 
 $f$  in the Frobenius norm expected to be
 strongly concentrated about its expected value 
 $\mathbb E(f)\le (1+\frac{r}{p-1})^{1/2}$
 (see detailed estimates in 
 \cite[Theorem 10.7]{HMT11}). 
For $p=4r+1$, say, this upper bound already 
decreases below 1.12. The algorithms use 
$\mathcal {NZ}_W+nl$ memory cells and
$O(l\cdot \mathcal {NZ}_W)$ flops.
   
 The alternative randomized
 algorithms of \cite[Section 11]{HMT11}
 generate  $2n+l$   
 random parameters defining
 an $n\times l$  SRFT or SRHT multiplier $H$ and then use $\mathcal {NZ}_W+n$ memory cells and
$O(\min\{l\cdot\mathcal \mathcal {NZ}_W, mn\log(n))$ flops for  computing the matrix $WH$. The subsequent computation of the matrices 
 $A=Q(H)$ and $B=A^*W$ still involves order of
 $l\cdot\mathcal {NZ}_W$ flops. 
 With a probability of failure of order $O(1/r)$
 this matrix 
 is an LRA 
  for a factor
$f=(1+ 7n/(r\log (r)))^{1/2}$  in the Frobenius norm and similarly in the spectral norm.
 
\begin{remark}\label{rehmtnz}
Our modification in Remark \ref{rerfnhmt}
  simplifies the transition from the factor $WH$ to an LRA: we involve just
  order of $(kn+lm)$ memory cells and
  $O(kn+lm\log(l))$ flops at that stage
  and (more importantly) yield optimal error bound up to a constant  factor $f$.
 \end{remark}
 
The Power Scheme variation of the above algorithms  computes the product 
$(W^*W)^qWH$ for a fixed  positive integer
$q$, rather than   $WH$. The computation of this product involves by $2q+1$ times more flops
than the original computation of the matrix $WH$ but is expected to output an  LRA with optimal spectral  error norm
  $\sigma_{r+1}$ up to a factor 
$$f\le \Big (1+\sqrt{\frac{r}{p-1}}+\frac{e\sqrt{r+p}}{p}\sqrt{n-r}\Big)^{1/(2q+1)}.$$
As $q$ increases, this bound converges very
fast  to
the optimal value  $1$ 
and nearly reaches this value already for 
$q$ of order $\log(\min\{m,n\})$ (see
\cite[Section 10.4]{HMT11}).
Recall, however, that already with an
$n\times (n+p)$ Gaussian multiplier 
for a reasonably large ratio $(p+1)/r$
we decrease the expected value of the error factor $f$  
to nearly 1 at the price of very minor 
increase of the computational cost. 

The C-A algorithm of 
\cite{GOSTZ10}  
 computes CUR LRA of an
$m\times n$ matrix of numerical rank $r$ within a factor of $\sqrt {mn}$
from optimal  by using $(m+n)r\alpha$ 
memory cells and $O(r^3)+(m+n-2r)r\alpha$ flops where the number $\alpha$ of C-A steps does not exceed $\begin{pmatrix}m\\r \end{pmatrix}\begin{pmatrix}n\\r \end{pmatrix}$ but empirically  is a small constant 
and is at most 2 on the average input
(see Section \ref{serralg1}).

Computation with leverage scores
in  \cite{FKW98}, \cite{DK03}, \cite{FKW04}, \cite{DKM06},  and
\cite{RV07} enabled  CUR LRA within  probabilistic error bounds
of the form
$$||W-CUR||_F^2\le \tilde\sigma_{r+1}^2 +\epsilon||W||_F^2~
{\rm for}~\min\{m,n\}\le p(r/\epsilon),$$
 involving $k=k(r,1/\epsilon,1/\delta)$  rows and $l=l(r,1/\epsilon,1/\delta)$ columns of an input matrix 
for  some polynomials $k(r,x,y)$ and
$l(r,x,y)$ of low degrees
and for a failure probability $\delta$.
This restricted  
the tolerance  $\epsilon$ in 
 (\ref{eqrmmrel}) from below because  
  $k\le m$
and $l\le n$.
 
Much stronger bound (\ref{eqrmmrel}) was  proved in  the paper \cite{DMM08},
 which  used  $mn$ memory cells and
$O(mnk+k^3)$ flops and
required to involve 
 $l\ge \psi r\log(r)/\epsilon^2$ columns and
 $k\ge \psi l\log(l)/\epsilon^2$ rows of 
 an input matrix $W$, respectively,
 for an unspecified constant $\psi$.
 This is inferior to the   
 estimates for the values $k$ and $l$ 
 of random row and column samples in \cite{HMT11}, but
in its tests the paper  \cite{DMM08}
reached accurate CUR LRAs  for 
$\epsilon=0.1$ and moderate
heuristic values of  $k$ and $l$.

Subsequent randomized  algorithms 
yielded  bound (\ref{eqrmmrel})
 under smaller asymptotic estimates
  for these parameters.
 In particular the randomized  algorithms   of \cite{BW14} and \cite{BW17} compute matrices $C$ and $R$ made up of
 $O(r/\epsilon)$ columns and
 $O(r/\epsilon)$ rows of $W$
 by using
$O(\mathcal {NZ}_W\log(n))+m~
{\rm poly}(r, n, 1/\epsilon)$ flops for $m\ge n$
and for ${\rm poly}(v)$  
denoting a polynomial in $v$.
By using  $O(mn^3r/\epsilon)$ flops overall
the authors make their CUR LRA algorithms deterministic.
Deterministic  algorithms of  
 \cite{GS12} and  \cite{BDM-I 14} 
 use a little fewer 
 flops for computing 
 LRAs, and  we can extend them to the
 computation of   CUR LRAs
   by applying our simple superfast Algorithms \ref{alglratpsvd} and 
  either \ref{algsvdtocur} or \ref{algsvdtocur}a. 

Randomized algorithms of \cite{SWZ17}
involve 
$\mathcal {NZ}_W+(m+n) {\rm poly} (r)$ flops, use matrices $C$ with 
$O(r\log (r))$ columns and  
$R$ with 
$O(r\log (r))$ rows, and
  achieve optimal output accuracy
  up to a factor of 
  $f=\log(n){\rm poly}(r)$ 
  in  norm (\ref{eqnrmswz}).
  For a constant $r$ and 
   poly$(m+n)$ flops,
  the authors decrease the output error norm factor $f$ to a constant.

   The cited algorithms 
  (except for  \cite{GOSTZ10} and stages 2 and 3 of Alg. \ref{algrhtcur})
   use at 
   least $\mathcal {NZ}_W$ memory cells and are not superfast for the worst case input.    
%------------------------------------------------------------------------------
They, however,  are superfast 
 as the subalgorithms of
primitive, cynical, or C-A iterations
applied to a 
$q\times s$ sketch for $q\ll m$
and $s\ll n$.

Likewise randomized Algorithm \ref{algrhtcur} becomes superfast
also at its stage 1 (not covered in Table \ref{tb:complexity}) if we perform it with   quasi Gaussian, Abridged Hadamard, or  Abridged 
Fourier multipliers instead of  Gaussian
ones; in this case we lose our formal estimates for the output accuracy of LRA of the worst case input.
  
%------------------------------------------------------------------------------

\subsection{A display of the accuracy and complexity estimates}\label{sacccmpldspl}

  Table \ref{tb:complexity} displays the  data for the LRA algorithms
   listed in the previous subsection and applied to $m\times n$ matrices of numerical rank $r$.
  We assume that $m\ge n$; 
  otherwise we could have applied  
  the same algorithms to the transposed input matrix.
   
 The column of the table
 marked by $f$ shows the known upper bound on the ratio of the Frobenius error norm of the output with the minimal error norm   $\tilde\sigma_{r+1}$, except that  the ratio is expressed  
  in  $l_1$-norm $||\cdot||_{l_1}$
 in the line of the paper \cite{SWZ17}.

       The column 
 marked by $k$ shows upper bound on the
 maximal numbers of rows or columns of the computed CUR generator.
  
   The column marked ``{\bf cells}" displays the estimated number of memory cells used (up to  smaller order terms). 
  
    The lines of the table representing the papers \cite{GL13}, \cite{HMT11}, \cite{GS12}, and \cite{BDM-I 14} cover the computation of LRAs that are not CUR LRAs, but our Algorithms \ref{alglratpsvd}  
    and  \ref{algsvdtocur} or \ref{algsvdtocur}a
    enable simple superfast transition to  CUR LRAs.
    
   The line  representing the survey paper \cite{HMT11}
 displays data for the expected value of the norm $||E||_F$ (see further probabilistic bounds on that norm in
\cite[Section 10.3]{HMT11}).

The line  representing Alg. \ref{algrhtcur} displays the data for its stages 2 and 3.

  The five papers  represented in the  last seven lines of the table cover papers originated in the Computer Science (CS) Community.
   The algorithms, techniques, and asymptotic cost  estimates   of these papers are important,  but the overhead constants hidden in the notation ``O" and ``poly" and an upper bound on the 
 positive scalar $\epsilon$ involved in
 the algorithms  are not explicitly specified so far (see Section \ref{srprpr}), although the overhead was reasonable in the tests reported in \cite{DMM08}. We also use ``O" notation 
 in the line of the table for our Alg. \ref{algrhtcur}, but only because we inherit this notation from \cite{GE96}.
In this case the overhead constants should be reasonable according to the data  about 
numerical tests in  \cite{GE96}; furthermore reasonable overhead constants are specified for the simialr algorithms of \cite{P00}. 
 
Let us list some assumptions, definitions,  and auxiliary bounds.
\begin{itemize}
\item
$t_{p,r,g}^2=(p-r)rg^2+1$ 
for $g\approx 1$ (cf. (\ref{eqtkrh})).
\item		
   In the line  representing the paper \cite{GOSTZ10}, $b$ is a 
  positive scalar of our choice -- 
   by increasing 
  its value we decrease the ratio $f$;
 the ratio  $||E||_F/||E||_C$ is bounded according to (\ref{eq0}) 
 and (\ref{eqnrmswz}) (also see 
 Remark \ref{rechb}), and 
   ITER denotes the number of iterations in the algorithm of \cite{GOSTZ10}:
in the worst case that number can grow large, but   empirically it is a constant, and is
   equal to 2
   on the average input.
\item 
We  can obtain the upper bound 1.12 on the factor $f$ in the line of the paper \cite{HMT11} by choosing the oversampling parameter $p:=4r+1$: 
 indeed
  $f\le(1+\frac{r}{p-1})^{1/2}=
  \frac{\sqrt 5}{2}<1.12$. 
   Notice that $f\ge 1,$
   $0\le p\le n-r$ and 
   $f\rightarrow 1$ as $p\rightarrow \infty$. 
\item
$\mathcal {NZ}_W$ denotes the number of nonzero entries of an input matrix $W$. 
\item
$\lg$ stands for $\log$.
\item
 $\epsilon\le 1$ is a positive parameter of our choice limited by the bound 
 $k\le m$. 
 \item
${\rm poly}(\lg(n),r,1/\epsilon)$, poly$(m)$, and poly$(r)$ denote some polynomials in the three variables
$\lg(n)$, $r$, and $1/\epsilon$ and in a single variable
$m$ or $r$, respectively.
    \end{itemize}
  
\begin{table}[ht]\label{tabacccmpl}
\centering
\begin{tabular}{|c|c |c|c |c|c|}
\hline
% & & &\multicolumn{2}{c|}{} & 
%\multicolumn{2}{c|}{} \\\hline
\bf{Source} &	$f$ & $k$	& {\bf cells} & \bf{flops}	\\\hline
\cite{GL13}& $1$ &$r$& $mn$ &
$O(m^2n)$\\\hline
\cite{GE96}& $t_{m,r,h}$&	$r$ & $mn$ &	$O(mn^2)$\\\hline
\cite{P00}& $t_{m,r,h}^2$&	$r$ & $mn$ &	$O(mn^2)$\\\hline
\cite{GOSTZ10}&
$\frac{b+1}{b}\frac{||E||_F}{||E||_C}$&
$(b+1)r-1$ &$mr\cdot$ {\rm ITER}  &$O(r^3)+(m-r)r\cdot {\rm ITER}$\\\hline
\cite{HMT11}
  &$(1+\frac{r}{k-r-1})^{1/2}$ &any $k>r+1$&
  $(m+k)n$&$16k\cdot\mathcal \mathcal{NZ}_W+O(r^2m)$\\\hline
%\cite{HMT11}&$\sqrt{1+\frac{7n}{r\lg(r)}}%$&$O((r+\lg(n)\lg(r))$&
% $O((r+\lg(n)\lg(r))$&	$O(mn\lg(n))$\\%\hline
  Alg. \ref{algrhtcur} & $O(1)$	&$O(r)$ &	$km$ &	$O(mk)$\\\hline
\cite{DMM08}& $1+\epsilon$ &$O(k\lg^2(k)/\epsilon^4)$&
$mn$& $O(mnk)$ \\\hline
\cite{MD09}& $1+\epsilon$ &$O(r\lg(r)/\epsilon^2)$&
$mn$& $O(mnk)$ \\\hline
%\cite{GS12}&$1+\epsilon$&$O(r/\epsilon)$ &	&$O(rmn^3\lg(n))$ &	no\\\hline
\cite{BDM-I 14}  &$2+\epsilon$&	$O(r/\epsilon)$ &$mn$&	$O(mn^3)$\\\hline
\cite{BW17}&$1+\epsilon$&$O(r/\epsilon)$ &	$mn$  &	$O(\mathcal {NZ}_W\lg(n))+$ \\
  &&	&	 	  &$m\cdot{\rm poly}(\lg(n,r,1/\epsilon)$\\\hline
\cite{BW17} &$1+\epsilon$&$O(r/\epsilon)$ &	$mn$  &	
$O(mn^3r/\epsilon)$\\\hline
\cite{SWZ17} &$O(\lg(n)\lg(r))$&$O(r\lg(r))$ &$mn$ & $\mathcal {NZ}_W+m\cdot {\rm poly}(r))$\\\hline
\cite{SWZ17} &$r\cdot{\rm poly}(\lg{r})$ &	 &	$mn$ &$n^{O(r}2^{O(r^2)}{\rm poly}(m)$\\\hline
% &	 &	 &	 &	 &	&	\\\hline
\end{tabular}
\caption{Accuracy and complexity
of some LRA algorithms (proven estimates for $m\ge n$)}
\label{tb:complexity}
\end{table}

%------------------------------------------------------------------------------

\section{Summary and Research Directions}\label{ssmr}

%------------------------------------------------------------------------------
%------------------------------------------------------------------------------

\subsection{Superfast LRAs: narrow classes of hard inputs and superfast computation of LRAs of random and the average inputs}\label{snrravr}

%------------------------------------------------------------------------------

We first recalled the LRA problem and then specified a family of $2mn$ matrices of size $m\times n$ (we call them $\pm\delta$
 matrices) 
for which no accurate LRA can be computed
unless the algorithm  
accesses  all $N$ nonzero entries of the input matrix. In that case, however, the algorithm involves at least
$N$ memory cells and 
at least $N/2$ flops.
 
Computational time and memory space used by some known {\em fast LRA algorithms} exceed but nearly match these lower bounds;
C-A  superfast  algorithms, however,  use much  fewer memory cells and  flops. We    provided formal support,
 so far missing, for the empirical observation that the C-A superfast  algorithms routinely output accurate
 LRAs in practice. 
 
 We first defined some natural randomization rules in the class of matrices allowing their  accurate LRAs and then proved that 
 C-A iterations and some other superfast algorithms output accurate LRAs of these  matrices whp;
 consequently they output accurate LRAs to the average matrix allowing LRA.  In other words, these algorithms fail to compute accurate LRAs only for a narrow subclass of the class of matrices allowing LRA.
 
   We computed CUR LRAs which approximate  a matrix in sublinear space. They are
defined by a CUR generators,
  which are small submatrices (sketches) of an input matrix.  
   Given an LRA, we 
   presented superfast algorithms for the
   computation of a CUR generator 
   and therefore a CUR LRA, as well as
 for a posteriori error estimation.  
 
%------------------------------------------------------------------------------

\subsection{Superfast approaches to LRA from NLA and CS and their synergy}\label{ssbspslsc}

%------------------------------------------

 Suppose that a CUR generator
shares with an input matrix
 numerical rank   or  
 has  maximal  volume
(up to a bounded factor)  among all its submatrices of the same size.
It has been proven that in both cases
the generator
defines  a close CUR LRA of a matrix.

Researchers in  Numerical
Linear Algebra (NLA) have been computing
CUR generators satisfying one or both of these sufficient CUR criteria, but at some point 
researchers in  Computer Science (CS)
proposed  dramatically different
 statistical approach. Namely they proposed 
to define CUR generators by randomly sampling
  sufficiently many rows and columns  
  of an input matrix according to  properly pre-computed leverage scores.
   
 The CS researchers proved 
 that such CUR generators  define  accurate CUR LRAs whp.  
In Section \ref{slrasmp} we studied 
such randomized CUR LRA algorithms of \cite{DMM08}, which compute CUR generators that whp satisfy this criterion  and thus define    
accurate (and even nearly optimal)
 CUR LRA. 
Those randomized  algorithms are superfast  except for their stage of computing leverage scores.
We  dropped this stage,
by choosing the uniform leverage scores,  and then proved that the output CUR LRA is  still accurate  whp
 for perturbed random   matrices allowing LRA and hence is  accurate  on  average matrices allowing LRA.  
  
 We demonstrated CS--NLA synergy
 by two  examples: 

(i) performing  C-A iterations we
 incorporated the randomized algorithms of \cite{DMM08}
 (rather than 
 deterministic ones of \cite{GE96}
 or  \cite{P00}) and 
 
 (ii) by using 
the algorithm of \cite{DMM08}, we refined a crude but reasonably close LRA,
say, supplied by a superfast algorithm from the NLA community.  
 We proposed to apply it
 to the original input matrix  by using the leverage scores computed for its crude LRA.

%------------------------------------------------------------------------------

\subsection{Superfast LRA computation based on randomized pre-processing}\label{srcps}

We 
 proved  that C-A iterations and some  other superfast algorithms 
 compute accurate CUR LRA to 
  a perturbed factor-Gaussian input,  but this proof does not apply to
 nonrandom inputs of the real world.
 So we randomized these inputs 
 by applying Gaussian, SRHT, or SRFT multipliers. Then
  we applied this 
  pre-processing to
  any matrix allowing LRA and proved that  
C-A iterations and some other superfast algorithms  output its accurate 
CUR LRA whp.
 
Neither Gaussian nor SRHT/SRFT
 pre-processing is superfast,
 but in our tests  
  superfast randomized pre-processing  with our sparse  multipliers was consistently as efficient
 as with Gaussian, SRHT, and SRFT multipliers.
  In particular we generated Abridged 
  SRHT and SRFT multipliers
 superfast --
 in our extensive tests we
 performed just  a few first
  recursive steps, out of
  $\log_2(n)$ steps routinely applied for generating $n\times n$
  SRHT and SRFT  multipliers for $n=2^k$.
  Then we pre-processed input matrices  superfast -- by using the resulting  Abridged  SRHT and SRFT  multipliers.  Finally we
consistently  computed accurate LRAs
by applying C-A iterations and some other superfast 
algorithms to such pre-processed matrices.
    
 Likewise, we 
 recursively factored a Gaussian multiplier
 into the product of random bidiagonal and random permutation matrices; we
 nontrivially proved that the  factorization converges,
 and in our tests it converged very  fast.
 
%------------------------------------------------------------------------------
 
\subsection{Superfast  algorithms for CUR LRA and related subjects}\label{sgrls}

%------------------------------------------

Superfast algorithms that 
empirically output accurate CUR LRAs
or are involved in their computations
 can be combined with  superfast a posteriori error
 estimation (see Sections \ref{spstr}
 and \ref{spstr1} and
\cite[Sections 4.3 and 4,4]{HMT11}).
We cover such estimation in Sections \ref{spstr}
 and \ref{spstr1}, extending the study in
\cite[Sections 4.3 and 4.4]{HMT11}.

Next we list some other superfast  algorithms
of this paper:
%\medskip 
%{\bf A superfast Procedure 1 for CUR LRA.}
\begin{enumerate}
\item
  primitive  algorithms,
\item
cynical algorithms,
\item
C-A iterations,
\item
the algorithms of \cite{DMM08}
provided that their original fast computation of SVD-based leverage scores is  accelerated to superfast level,
\item
pre-processing with sparse and structured multipliers such as quasi Gaussian,
Abridged Hadamard, and Abridged
Fourier multipliers,
\item
transition from any close LRA of
(\ref{eqlrk}) to CUR LRA,
\item 
calculation of  
SVD-based leverage scores for a low rank matrix of a large size
given with its CUR generator and the
extension to the refinement of
 a crude but reasonably close LRA. 
 \end{enumerate}
  
%------------------------------------------------------------------------------

\subsection{Our technical novelties,
some impacts, 
applications, and
 extensions}\label{simpapplext}

%------------------------------------------
 
Our study  provided some new insights into  LRA and CUR LRA, demonstrated synergy of 
  combining  the LRA methods proposed independently in the CS and NLA communities,  and  should embolden wider application   of our tools
  such as
    simplified heuristics,
 the average case analysis, 
 sketching methods,
    randomized sparse and structured
   pre-processing, and  recursive algorithms of C-A type.
      
Our progress should be interesting for a
wide range of  computations involving LRA,
but our  
 superfast CUR LRA algorithms enable new surprising applications.
 In  Section \ref{sextpr} we  
dramatically accelerated 
the construction of low rank generators at  the bottleneck stage of the Fast Multipole Method (FMM), which is among the ten top algorithms of the 20th century \cite{C00}. 
  
%------------------------------------------------------------------------------

%\subsection{Our Technical Novelties}%\label{slsttchnnv} 

%------------------------------------------

We achieved our
progress in a hot area of intensive research by applying relatively simple means.  
 Our study covered a large class of real world inputs, as our tests testified.
 Moreover our analysis supported {\em  superfast accurate computation of CUR LRA
 of any input allowing close LRA and
 pre-processed  with random multipliers, and our randomize pre-processing has led to new superfast LRA.} Moreover we covered superfast a posteriori error estimation and correctness verification and
  proposed  superfast refinement of any 
 matrix allowing its LRA and given together with its crude but reasonably close initial LRA. 

%Some  of our auxiliary techniques and %algorithms can be of independent interest.
 
 As by-product of our analysis of the known superfast LRA algorithms, we devised some new efficient ones, revealed and handled some  hidden delicate technicalities, and introduced 
  concepts and methods of independent interest. 
 
Here are some examples:
\begin{enumerate}
\item 
We introduced and explored the auxiliary concepts of factor-Gaussian matrices
(in Section \ref{srmcs}) 
and   $G_{n,r,\beta}$-leverage scores
(in Section \ref{slrasmpav}), 
  which enabled natural definitions of the average matrix of a low numerical rank.
We also defined the average sparse inputs allowing LRA.
\item
We formally supported the observed power
of superfast  LRA algorithms 
by studying them for random and  average inputs, both  in depth and from three distinct angles -- namely by
using error analysis, volume maximization, and leverage scores.
\item
We estimated the volume of a matrix in terms of its local volume,
estimated the volume of matrix products
in terms of the volumes of the factors,
and linked $r$-projective volume of 
a CUR generator to the volume of an associated submatrix of  full rank. 
\item 
We  empirically accelerated fast 
pre-processing for LRA to superfast level by applying quasi Gaussian, Abridged 
 Hadamard, and Abridged Fourier random multipliers.
 \item   
 We proposed memory efficient randomized factorization of a Gaussian multiplier into the product of random bidiagonal and random permutation matrices, 
 nontrivially proved that the product  converges to a Gaussian matrix as the number of multipliers grows large,
  and empirically verified that the convergence is fast.
 \item
 We proposed a superfast extension of any LRA to CUR LRA.
 \item  
We improved a decade-old estimate for the norm of the inverse of a Gaussian matrix and extended it to a sparse 
Gaussian matrix (see Remark \ref{resstimp} 
and Theorem \ref{thNZnrm+}). 
   \item 
By computing  LRAs superfast we  
 dramatically accelerated 
 the construction of low rank generators for the FMM, which is frequently the   bottleneck of this highly important and popular method.
  \item
  In Section \ref{spstr} we propose
   superfast a posteriori LRA error estimation when the input entries are 
 the observed  i.i.d. values of a single random variable; in this case we do not need to  know the input norm and the output error norms.
 \item 
Motivated by our proofs that C-A and some other superfast algorithms compute accurate LRAs of random input matrices, we proposed randomized multiplicative pre-processing and
 proved that whp  C-A and some other superfast algorithms compute close CUR LRAs when they are applied to any matrix allowing its close LRA and 
pre-processed by using Gaussian, SRHT, or SRFT multipliers. 
   \end{enumerate}
  
%------------------------------------------------------------------------------

\subsection{Further natural research directions}\label{srsrc}

%------------------------------------------
 
In Section  \ref{srprpr} we discussed some  research challenges motivated by the progress in statistical approach to study LRA. Here are  some other  research challenges.
 
 (i)  Try to deduce whp the second and the 
 third CUR criterion for a perturbed sparse factor-Gaussian matrix.
  
   (ii) How much can we extend our analysis  if the gaps between all pairs of
 the consecutive singular values of 
 an input matrix are not dramatic?
  
 (iii)  Given a crude LRA of a matrix 
     under the Frobenius or spectral norm,
    can we compute superfast 
    its refined LRA under $l_1$-norm?  
    
Further research challenges include
  \begin{enumerate} 
\item
 design, analysis,  and experimental study                                                                                                                                                      
 of new                                                                                                                                                                                                                                                                                                                                                                                                                                                            multipliers for  pre-processing  LRA,
\item
 the study of the benefits of the size variation of candidate CUR generators 
 when we alternate
expansion and compression of the input sketches of C-A iterations (see Appendix
\ref{scntrp}),
\item
further formal and experimental study  
of C-A techniques and their extension 
to tensor decomposition (cf. 
\cite{OST08}, \cite{MMD08}, \cite{CC10},
and our Remarks 
\ref{rec-atnsr} and \ref{re2orddmn}),
\item
 extension of the application area of  LRA where  we can  compute LRA superfast, 
\item
investigation of the links among the three sufficient CUR criteria,
  and
 \item
achieving  synergy by combining LRA algorithms and techniques proposed independently by researchers 
from the communities of  NLA and CS.
\end{enumerate}  

Finally, as we have already said earlier, our progress may suggest wider study of the efficiency of
    simplified heuristics,
 the average case analysis, 
% sketching methods,
    randomized sparse and structured
   pre-processing, and  recursive algorithms of C-A type
   in various computations of linear and multilinear algebra and beyond them.
   
\medskip

\medskip

%------------------------------------------------------------------------------
%------------------------------------------------------------------------------

{\bf \Large Appendix}

\appendix 

%\bigskip

%------------------------------------------------------------------------------

\section{Ranks and Norms of Random Matrices}\label{srnrmcs}

%------------------------------------------------------------------------------

\subsection{Ranks of random matrices}\label{srrmcs}
 
%------------------------------------------------------------------------------

\begin{theorem}\label{thrnd}  
Suppose that   $A$, 
$F$, and $H$ are  $m\times n$,
$r\times m$, and 
$n\times r$  matrices, respectively,
 the  entries of $F$ and $H$ are nonconstant linear combinations
of finitely many i.i.d. random variables $v_1,\dots,v_h$, and 
$\rank(A)\ge r$.

Then 
the matrices $F$, $FA$, $H$, and $AH$  
have full rank $r$ 

(i) with probability 1
if $v_1,\dots,v_h$ are Gaussian variables
and 

(ii) with a probability at least                                                                                                                                                                                                                                                                                                                  
$1-r/{\rm card}(\mathcal S)$
if they are random variables sampled under the
uniform probability distribution from 
a finite set $\mathcal S$ having cardinality 
${\rm card}(\mathcal S)$.  
\end{theorem}

%------------------------------------------------------------------------------
 
\begin{proof}
The determinant, $\det(B)$,
of any $r\times r$ block $B$ of a matrix 
 $F$, $FA$, $H$, or $AH$
is a polynomial of degree $r$ in the variables  $v_1,\dots,v_h$,
and so the equation $\det(B)=0$ 
  defines an algebraic variety of a lower
dimension in the linear space of these variables
(see \cite[Proposition 1]{BV88}). 
Clearly such a variety has Lebesgue  and
Gaussian measures 0, both being absolutely continuous
with respect to one another. This implies claim (i) of the theorem.
Derivation of claim (ii) from a celebrated lemma of \cite{DL78},
 also known from \cite{Z79} and \cite{S80},  
is a well-known pattern, specified in some detail
 in \cite{PW08}.
\end{proof}

%------------------------------------------------------------------------------

\begin{corollary}\label{cornd}
Suppose that $\mathcal{NZ}$-Gaussian matrix $W$ of size $m\times n$ has neither rows nor columns filled with zeros.
Then $\rank(W)=\min\{m,n\}$ with probability 1.
\end{corollary}

%------------------------------------------------------------------------------
%------------------------------------------------------------------------------

\subsection{Norms of Gaussian matrices and their pseudo inverses}\label{sngmcs}

%------------------------------------------------------------------------------

We state the following estimates for real Gaussian matrices,  
but similar estimates in the case of complex  
 matrices can be found in
 \cite{D88}, 
\cite{E88}, \cite{E89}, \cite{CD05}, 
and \cite{ES05}.

%-----------------------------------------------------------------------------

\begin{theorem}\label{thsignorm}
Let $p$ and $q$ be positive integers
and let $t\ge 0$.
%------------------------------------------------------------------------------ 
Then 

(i) {\rm Probability}$\{\nu_{p,q}>t+\sqrt p+\sqrt q\}\le
\exp(-t^2/2)$ (see \cite[Theorem II.7]{DS01}),  

(ii) $\mathbb E(\nu_{p,q})\le \sqrt p+\sqrt q$ (see \cite[Proposition 10.1]{HMT11}
for $S=T=I$),  and 

(iii) $\mathbb E(\nu_{p,q,C})=\mathbb E(\nu_{1,pq,C})\le f+\frac{1}{f}$, 
for $f: =\sqrt{2\ln(\max\{2,pq\})}$
(see \cite[Lemma A.3]{SST06}).
\end{theorem}   

%-----------------------------------------------------------------------------

\begin{remark}\label{recheb}
Given a $p\times q$ Gaussian matrix $W$,
apply Theorem \ref{thsignorm} and obtain that
$E(||W||)\le \sqrt {p}+\sqrt {q}$ and 
$E(||W||_C)\le f+\frac{1}{f}$ where $f\approx \sqrt{2\ln(pq)}$
for large integers $pq$; furthermore
both norms $||W||$ and $||W||_C$
deviate from their expected values by more than a factor $\zeta>1$
with a probability that decays exponentially fast as $\zeta$ grows to infinity. 
Notice that the theorem also implies estimates for
$||W||_F=\nu_{mn,1}$
where  $W\in \mathcal G^{m\times n}$. 
\end{remark}

%------------------------------------------------------------------------------

\begin{theorem}\label{thsiguna}   
Let $x>0$ 
%{\rm and}~\zeta(t)=
%\frac{\sqrt{2p}}{\Gamma(p/2)}(t\sqrt{p/2})^{p-1}\exp(-pt^2/2)=
%2~t^{p-1}(\frac{p}{2})^{p/2}\exp(-\frac{p}{2}t^2)/\Gamma(\frac{p}{2}).$$ 
and let $\Gamma(x)$
denote the Gamma function of (\ref{eqgmma}). Then

(i)  {\rm Probability} $\{\nu_{p,q}^+\ge px^2\}<\frac{x^{q-p-1}}{\Gamma(p-q+2)}$
for $p\ge q\ge 2$,

(ii)  $\mathbb E(\nu^+_{p,q})< e\sqrt{p}/|p-q|$
provided that $p\neq q>1$ and $e=2.71828\dots$ and
%Probability $\{||(G+A)^+||\ge 2.35x\sqrt {q}\}\le 1/x$  
%for any $m\times n$ matrix $A$

(iii) {\rm Probability} $\{\nu_{q,q}^+\ge 1/x\}\le  {\sqrt \frac{2}{\pi}}~x$ 
for $q\ge 2$.
\end{theorem}

%------------------------------------------------------------------------------

\begin{proof}
 See \cite[Proof of Lemma 4.1]{CD05} for claim (i), 
\cite[Proposition 10.2]{HMT11} for claim (ii) and
our Theorem \ref{thgssnrm} for claim (iii).
\end{proof}

%-----------------------------------------------------------------------------

\begin{remark}\label{resstimp}
Part (iii) improves the following estimate
of \cite[Theorem 3.3]{SST06}:
{\rm Probability} $\{\nu_{q,q}^+\ge x\}\le 2.35 {\sqrt q}/x$, 
for $q\ge 2$.
\end{remark}

%-----------------------------------------------------------------------------

\begin{remark}\label{regsswll}
The  probabilistic upper bounds of Theorem \ref{thsiguna}
on $\nu^+_{p,q}$ are quite reasonable even
where $p=q$,
but are strengthened very fast as the difference $|p-q|$ grows from 1.
Furthermore recall that 
$1/\nu_{p,q}\ge |\sqrt p-\sqrt q|$  
 whp (see \cite{RV09}, \cite{TV10}). In sum, all these bounds combined imply that a $p\times q$ Gaussian
matrix is well-conditioned  
unless the  integer $|p-q|$ is close to 0.
With some grain of salt 
we can consider  such a matrix  
well-conditioned  
 even
  where $p=q$.
\end{remark}

%------------------------------------------------------------------------------

\subsection{Norms of sparse  Gaussian matrices and their pseudo inverses}\label{snsgmcs}

%------------------------------------------------------------------------------

 The estimates of 
Theorem \ref{thsignorm}  cover
 sparse matrices as well,
 and we immediately bound the norm of 
 a sparse Gaussian matrix as follows.
 
%------------------------------------------------------------------------------

\begin{theorem}\label{thNZnrm}  
The squared Frobenius norm of a $\mathcal{NZ}$-Gaussian matrix $W$ is distributed according to the $\chi^2$ distribution with $\mathcal{NZ}$ degrees of freedom, \textit{i.e.} $||W||^2_F \sim
\chi^2(\mathcal {NZ}_W)$.
\end{theorem}
Next we estimate the norm of 
 the Moore-Penrose pseudo inverse
 of  a sparse Gaussian matrix.
We assume that 
the matrix has full rank with probability 1 (cf. Theorem \ref{thrnd}), and in our probabilistic analysis further assume that the matrix does have full rank.

%------------------------------------------------------------------------------

\begin{definition}\label{defsparse}  
Let  $\mathcal {NZ}_{i,:}(W)$ 
%and
%$\mathcal {NZ}_{:,j}(W)$)
denote the numbers of nonzero
entries in the
 $i$th row  
 %and  $j$th column  
 of a $p\times q$
matrix $W=(w_{i,j})_{i,j=1}^{p,q}$.
%respectively. 
%Write
%$\mathcal {NZ}_{+,:}(W)=
%\max_i\mathcal {NZ}_{i,:}(W)$,
%$\mathcal {NZ}_{-,:}(W)=
%\min_i\mathcal {NZ}_{i,:}(W)$,
%$\mathcal {NZ}_{:,+}(W)=
%\max_j\mathcal {NZ}_{:,j}(W)$
 and
%$\mathcal {NZ}_{:,-}(W)=
%\min_j\mathcal {NZ}_{:,j}(W)$.
let ``$\preceq$" stand for
``stochastically less than". 
\end{definition}
%------------------------------------------------------------------------------

%Recall that the norms
%$||{\bf u}||_1=\sum_{i=1}^p|u_i|$ 
%and 
%$||{\bf u}||_{\infty}=\max_{i=1}^p|u_i|$ 
%for a vector ${\bf u}
%=(u_i)_{i=1}^p
%$ of dimension $p$ satisfy 
%\begin{equation}\label{eqvnrm}
%||{\bf u}||\le ||{\bf u}||_1\le
%\sqrt p~||{\bf u}||,~
%%||{\bf u}||_{\infty}\le ||{\bf u}||\le
%\sqrt p~||{\bf u}||_{\infty}.
%\end{equation}
%Also recall that $\nu_{1,1}$ and $  %absolute value of a standard Gaussian %variable and its reciprocal, %respectively,
%\begin{equation}\label{eqnu11sq}
%{\rm Probability}~\{\nu_{1,1}\le x\}\le  
%\sqrt{\frac{2}{\pi}}~x,
%\end{equation}
% \begin{equation}\label{eqnu11}
% \mathbb E(\nu_{1,1})=
% \sqrt{\frac{2}{\pi}},~
% {\rm Var}~(\nu_{1,1})=1-2/\pi~
%{\rm and}~\mathbb E(\nu_{1,1}^2)=1.
%\end{equation}

%------------------------------------------------------------------------------

%\begin{theorem}\label{thNZlincmb}
%(Cf. \cite{C46}.)
%Suppose that $x_i$ are independent %standard Gaussian variables  and 
%$a_i$ are  real numbers, for $i=1,2,
%\dots,q$. Then 
%$s=\sum_{i=1}^q a_ix_i$ is a Gaussian %variable with expected value $\mathbb %E(s)=0$
%and variance ${}\rm Var(s)=\sum_{i=1}^q %a_i^2$.
%\end{theorem} 

\begin{lemma}\label{leinner} (See 
\cite[Lemma A.2]{SST06}.)
Consider the inner product
$p={\bf v}^T{\bf g}$ where
 $||{\bf v}||=1$ and ${\bf g}$ is a Gaussian vector.
 Then for any real value $t$ it holds that
 
 $${\rm Probability}\{|p-t|\le x\}\le \sqrt{\frac{2}{\pi}} x.$$
 \end{lemma}

%------------------------------------------------------------------------------
%------------------------------------------------------------------------------

\begin{theorem}\label{thNZnrm+}  
Let $W$ be a
$p\times q$  
$\mathcal {NZ}$-Gaussian
matrix of rank $q$.
Then
%\begin{equation}\label{eqnrmw_+prob}
$${\rm Probability}~\{||W^+||\ge 1/x\}\le  
\sqrt{\frac{2q}{\pi}}~x.$$
%\end{equation}
\end{theorem} 

%------------------------------------------------------------------------------

\begin{proof}
%Let $\sum_jv_j^2=1$.
Deduce from \cite[Theorem 8.1.2]{GL13}) that
%\begin{equation}\label{eq1/nrmw}
$$ 1/||W^+||=\sigma_{q}(W)=
\min_{||{\bf v}||=1}||W{\bf v}||.$$
%\end{equation}
Let ${\bf v}=(v_j)_{j=1}^q$ 
be a vector such that $||{\bf v}||=1$ and
$1/||W^+||=||W{\bf v}||$.

Notice that
$||W{\bf v}||^2= \sum_i\Big (\sum_jv_jw_{i,j}\Big )^2\ge 
(\sum_jv_jw_{i,j}\Big)^2,$
and so
$$1/||W^+||=||W{\bf v}||\ge 
\Big |\sum_jv_jw_{i,j}\Big |~
{\rm for~all}~i.$$

For every $i$ drop the terms of this sum
where $w_{i,j}=0$, that is, rewrite
the sum as  
$|\sum_{j\in \mathcal {NZ}_{i,:}}v_jw_{i,j}|$. Then rewrite 
it as the inner product
${\bf u}^{(i)T}{\bf g}^{(i)}$ 
where ${\bf g}^{(i)}=
(w_{i,j})_{j\in \mathcal {NZ}_{i,:}}$
is the Gaussian  vector 
 made up  of all nonzero components 
 of the vector 
$(\pm w_{i,j})_{j=1}^q$ and
 ${\bf u}^{(i)}=
 (\pm v_j)_{j\in \mathcal {NZ}_{i,:}}$ is a subvector of
 the vector $(v_j)_j$.

Notice that  $\max_j |v_j|\ge 1/\sqrt q$
and
let $|v_1|\ge  1/\sqrt q$, say. 
Then notice that 
 $\pm w_{i,1}$
is a Gaussian variable rather than 0
for some $i$, say,
for $i=1$.
(If $w_{i,1}=0$ for all $i$, then
 $\rank(W)<q$, contrary to
our assumption.)

Hence $|\sum_jv_jw_{1,j}|$
is the inner product of a Gaussian 
vector and a vector ${\bf u}^{(1)}$ 
with $||{\bf u}^{(1)}||\ge 1/\sqrt q$.
Apply 
 Lemma \ref{leinner} for $t=0$ to
 the inner product 
 ${\bf u}^{(1)T}{\bf g}^{(1)}/
 ||{\bf u}^{(1)}||$
  and deduce the theorem.
  %(\ref{eqnrmw_+prob}).
 \end{proof} 
 
%------------------------------------------------------------------------------

\begin{remark}\label{reNZdual}
If $\rank(W)=p$,
then apply   the theorem to the matrix $W^T$
replacing $W$,
recall that 
$||W^+||=||W^{+T}||$, and obtain the following dual counterparts to the theorem:
%\begin{equation}\label{eqnrmw_+prob}
$${\rm Probability}~\{||W^+||\ge 1/x\}\le  
\sqrt{\frac{2p}{\pi}}~x.$$
%\end{equation}
\end{remark} 

%------------------------------------------------------------------------------

\begin{remark}\label{reNZnrm}
The bound of  the theorem matches that of claim (iii) of Theorem \ref{thsiguna}
and is a little weaker than the bounds of claims (i) and (ii) of that  theorem.
\end{remark} 

%------------------------------------------------------------------------------

\begin{theorem}\label{thgssnrm}  
Let $W$ be a
$p\times q$  
dense Gaussian
matrix,
then
%\begin{equation}\label{eqnrmw_+probdense}
$${\rm Probability}~\{||W^+||\ge 1/x\}\le  
\sqrt{\frac{2}{\pi}}~x.$$
%\end{equation}
\end{theorem}

%------------------------------------------------------------------------------

\begin{proof}
Proceed as in the proof of 
Theorem \ref{thNZnrm+} applied to a $p\times q$ matrix $W$
and obtain
$$|\sum_jv_jw_{i,j}|=
(\pm v_j)_j^T(\pm w_{i,j})_{j=1}^q$$
where $(\pm w_{i,j})_{j=1}^q$ is a Gaussian vector of dimension $q$
and $(\pm v_j)_j$ is a unit vector.
\end{proof}

%-----------------------------------------------------------------------------

\section{Abridged Hadamard and Fourier Multipliers}\label{sahaf}
 
%------------------------------------------------------------------------------

  Recall the  following recursive  definition of  
dense and orthogonal (up to scaling by constants) $n\times n$ matrices $H_n$ 
of {\em Walsh-Hadamard transform} for $n=2^k$
(see   \cite[Section 3.1]{M11} and our Remark \ref{recmb}): 
\begin{equation}\label{eqrfd}
H_{2}: =
\big (\begin{smallmatrix} 1  & ~~1  \\  
1   & -1\end{smallmatrix}\big )~{\rm and}~H_{2q}: =\begin{pmatrix}
H_{q} & H_{q} \\
H_{q} & -H_{q}
  \end{pmatrix}
%H_{2^{k-d}=\begin{pmatrix}
%I_{2^{k-d} & I_{2^{k-d}  \\
%I_{2^{k-d} & -I_{2^{k-d} 
%\end{pmatrix},
\end{equation}
for $q=2^h$, $h=0,1,\dots,k-1$.  
Recursive representation (\ref{eqrfd}) enables  fast 
 multiplication of a matrix $H_{n}$ by a vector,  using $nk$ additions
and subtractions for $n=2^k$.

Now 
 shorten the recursive process  
by fixing a  {\em recursion depth} $d$, $1\le d<k$, and  applying equation (\ref{eqrfd}) where
$q=2^hs$, $h=k-d,k-d+1,\dots,k-1$,  and $H_{s}=I_s$ for 
$n=2^ds$.
%is the $n\times n$ Hadamard's  primitive matrix $H^{(2s)}$ of type 4.
For  two positive integers $d$ and $s$,
 denote the resulting $n\times n$ matrix
$H_{n,d}$ and call it the matrix of
 $d$-{\em abridged Hadamard
 (AH)
  transform} if $1\le d< k$. 
In particular 
\begin{equation}\label{eqbschf}
H_{n,1}: =\begin{pmatrix}
I_s  &  I_s  \\
I_s  & -I_s  
\end{pmatrix},
~{\rm for}~n=2s;~~
H_{n,2}: =\begin{pmatrix}
I_s  &  I_s & I_s  &  I_s \\
I_s  & -I_s  & I_s  &  -I_s \\ 
I_s  &  I_s & -I_s  &  -I_s \\
I_s  & -I_s  & -I_s  &  I_s
\end{pmatrix},~{\rm for}~n=4s.
\end{equation}

For a fixed $d$, the matrix  $H_{n,d}$
is still orthogonal up to scaling,
 has $q=2^d$ nonzero entries 
in every row and  column, and hence
 is sparse unless 
$k-d$ is a small integer.
Then again, by applying  recursive process (\ref{eqrfd}) we multiply such a  matrix by a vector fast, by using  just
  $dn$
additions/subtractions
and allowing
 efficient  
 parallel implementation 
(see 
Remark \ref{resprs0}).

Likewise recall a recursive process
of the generation of the $n\times n$ matrix  $\Omega_n$
of {\em discrete Fourier transform (DFT)} at $n$ points,  for $n=2^k$:
\begin{equation}\label{eqdft}
\Omega_n: =(\omega_{n}^{ij})_{i,j=0}^{n-1},~{\rm for}~n=2^k~{\rm and~a~primitive}~
n{\rm th~root~of~unity}~\omega_{n}: =\exp\Big (\frac{2\pi}{n} \sqrt {-1}\Big ).
\end{equation}

The matrix $\Omega_n$ is unitary up to scaling by $\frac{1}{\sqrt n}$.
We can multiply it by a vector 
by using $1.5nk$ flops and  can efficiently parallelize this computation
 if, instead of its representation by entries, we apply its recursive representation, called decimation in frequency (DIF) radix-2 representation
(see \cite[Section 2.3]{P01}  and our Remark \ref{recmb}):
\begin{equation}\label{eqfd}
\Omega_{2q}: =
\widehat P_{2q}
\begin{pmatrix}\Omega_{q}&~~\Omega_{q}\\ 
\Omega_{q}\widehat D_{q}&-\Omega_{q}\widehat D_{q}\end{pmatrix},~
\widehat D_{q}: =\diag(\omega_{n}^{i})_{i=0}^{n-1}.
\end{equation}
Here $\widehat P_{2q}$ is the matrix of odd/even permutations 
 such that 
$\widehat P_{2^{h}}({\bf u})={\bf v}$, ${\bf u}=(u_i)_{i=0}^{2^{h}-1}$, 
${\bf v}=(v_i)_{i=0}^{2^{h}-1}$, $v_j=u_{2j}$, $v_{j+2^{h-1}}=u_{2j+1}$, 
$j=0,1,\ldots,2^{h-1}-1$;
$q=2^h$, $h=0,1,\dots,k$,  and
$\Omega_{1}=(1)$ is the scalar 1.
%$H_{2}=\big (\begin{smallmatrix} 1  & ~~1  \\
%1   & -1\end{smallmatrix}\big ).$
  
Now shorten the recursive process  
by fixing a recursion depth $d$, $1\le d<k$,  
replacing $\Omega_{s}$ for $s=n/2^{d}$
by the identity matrix $I_{s}$,
and then  applying equation (\ref{eqfd}) for
$q=2^h$, $h=k-d,k-d+1,\dots,k-1$.
For $1\le d<k$ and $n=2^ds$,   
we denote the resulting $n\times n$ matrix
$\Omega_{n,d}$ and call it the matrix of
 $d$-{\em abridged   Fourier
 (AF) transform}. It is also unitary (up to scaling),
 has $q=2^d$ nonzero entries 
in every row and column, and thus is
  sparse unless 
$k-d$ is a small integer. 
We can represent such a matrix by its entries, 
but if we rely on recursive 
representation (\ref{eqfd}),
then again its
multiplication by a vector involves just  $1.5dn$
flops
and allows
highly efficient  
 parallel implementation. 

By complementing the above definitions of the matrices with randomization we define
the matrices $DH_{n}P$ and $DH_{n}P$ of {\em randomized Hadamard 
and Fourier transforms} or (by using acronyms) the matrices of {\em RHT} and {\em RFT}. 
Here $D$ is a random $n\times n$ diagonal matrix whose entries are independent
and uniformly distributed on the  complex unit circle $\{x:~||x||=1\}$ and $P$
is a random $n\times n$ permutation matrix.                                               
Suppose that we sample $l$ columns from such a matrix uniformly at random and without replacement and scale the resulting matrix by the constant $\sqrt {\frac{n}{l}}$.
Then such a scaled $n\times l$  matrix is called the matrix of                                                                                     
 a {\em subsampled randomized Hadamard or Fourier transform
(SRHT or
 SRFT)}. 
  
Likewise we define   
the matrices $DF_{n,d}P$ and $DH_{n,d}P$ of {\em abridged randomized Fourier and Hadamard's 
transforms}.

\begin{remark}\label{recmb}
The paper \cite{PZ16} defines the same                      matrices of abridged randomized    Hadamard and Fourier transforms
but calls them  matrices of 
abridged 
scaled and permuted
Hadamard and Fourier transforms.   
\end{remark}

%------------------------------------------------------------------------------

\begin{remark}\label{resprs0}
Other observations besides flop estimates  can  decide the efficiency of multipliers.
For example, a
special recursive structure of the
matrices $H_{2^{k},d}$ and 
 $\Omega_{2^{k},d}$ of abridged 
  Hadamard and Fourier transforms
allows
highly efficient  
 parallel implementation of  
their multiplication by a vector based on 
Application Specific Integrated Circuits (ASICs) and 
Field-Programmable Gate Arrays (FPGAs), incorporating Butterfly
Circuits \cite{DE}.
\end{remark}

%------------------------------------------------------------------------------

\section{Computation of Sampling and Rescaling Matrices}\label{ssrcs}
 
%------------------------------------------------------------------------------

We begin with the following simple computations.
Given an $n$ vectors 
${\bf v}_1,\dots,{\bf v}_n$
of dimension $l$, write
 $V=({\bf v}_i)_{i=1}^n$
 and compute $n$ leverage scores
 
%------------------------------------------------------------------------------

\begin{equation}\label{eqsmpl}
p_i={\bf v}_i^T{\bf v}_i/||V||^2_F,
i=1,\dots,n.
\end{equation}
Notice that $p_i\ge 0$ for all $i$ and
 $\sum_{i=1}^np_i = 1$.

%------------------------------------------------------------------------------

Next assume that some leverage scores $p_1,\dots,p_n$
are given to us and recall  \cite[Algorithms 4 and 5]{DMM08}.
For a fixed positive 
integer $l$ they sample  
either exactly $l$ columns 
 of an input matrix $W$ (the $i$-th column
 with probability $p_i$)
or at most $l$ its columns in expectation
(the $i$-th column
 with probability $\min\{1, lp_i\}$),
respectively.

%------------------------------------------------------------------------------

\begin{algorithm}\label{algsmplex} 
{\rm The Exactly($l$) sampling and rescaling}.

%------------------------------------------------------------------------------

\begin{description}
 
%------------------------------------------------------------------------------

\item[{\sc Input:}] 
Two integers $l$ and $n$ such that $1\le l\le n$ and $n$ nonnegative scalars $p_1,\dots,p_n$
such that $\sum_{i=1}^np_i = 1$.

%------------------------------------------------------------------------------

\item[{\sc Initialization:}]
Write $S:=O_{n,l}$ and $D:=O_{l,l}$. 

\item[{\sc Computations:}]
 
(1)  For $t = 1,\dots,l$ do

Pick $i_t\in \{1,\dots,n\}$ such that
 Probability$(i_t = i) = p_i$;

$s_{i_t,t} := 1$;

$d_{t,t} = 1/\sqrt{lp_{i_t}}$;

end

\medskip

(2) Write $s_{i,t}=0$ for all pairs of 
$i$ and  $t$
unless $i=i_t$.

%------------------------------------------------------------------------------

\item[{\sc Output:}] 
$n\times l$ sampling matrix $S=(s_i,t)_{i,t=1}^{n,l}$ and
$l\times l$ rescaling matrix 
$D=\diag(d_{t,t})_{t=1}^l$.

\end{description}
\end{algorithm}
The algorithm performs $l$ searches in the set $\{1,\dots,n\}$, $l$ multiplications,
$l$ divisions, and $l$ computations
of square roots.

%------------------------------------------------------------------------------

\begin{algorithm}\label{algsmplexp} 
{\rm The Expected($l$) sampling  and rescaling.} 

%------------------------------------------------------------------------------

\begin{description}
 
%------------------------------------------------------------------------------

\item[{\sc Input, Output and Initialization}] are as in Algorithm   \ref{algsmplex}. 
\item[{\sc Computations:}]
Write $t := 1$;

for $t = 1,\dots, l-1$ do

for $j = 1,\dots, n$ do

Pick $j$ with probability
$\min\{1, lp_j\}$;

if $j$ is picked, then

$s_{j,t}: = 1$;

$d_{t,t} := 1/\min\{1,\sqrt{lp_j}\}$;

$t := t + 1$;

end

end

%------------------------------------------------------------------------------

\end{description}

%------------------------------------------------------------------------------

\end{algorithm}

%Computations of  
Algorithm \ref{algsmplexp} involves  $ln$ memory cells and
$O((l+1)n+l\log(l))$ flops.
 
Obtain  the following results from 
 \cite[Lemmas 3.7 and 3.8]{BW17} (cf. \cite{RV07}).

\begin{theorem}\label{thsngvsmpl}
 Suppose that  
$n > r$,  $V\in \mathbb C^{n\times r}$ and $V^TV = I_r$.
Let $0<\delta\le 1$ and  
$4r \ln(2r/\delta)<l$. 
Define leverage scores by equations
(\ref{eqsmpl}) and then compute
the sampling and rescaling matrices $S$ and $D$ by 
applying Algorithm \ref{algsmplexp}.
Then, for all $i = 1,\dots, r$ and with a probability at least $1-\delta$ it holds that

%------------------------------------------------------------------------------

\begin{equation}\label{eqsngvsmpl}
 1-\sqrt{4r\ln (2r/\delta)/l}\le
\sigma_i^2(V^TSD)\le
1+\sqrt{4r\ln (2r/\delta)/l}.
\end{equation}
\end{theorem}

Notice that $||D^{-1}||_F\le \sqrt r$.

\begin{theorem}\label{thnrmsmpl}
Define the sampling and scaling matrices 
$S$ and $D$ as in Theorem \ref{thsngvsmpl}. 
Then for an $m\times n$ matrix $W$
it holds with a probability at least 0.9 that $||WSD||_F^2\le ||W||^2_F$.
\end{theorem}

%------------------------------------------------------------------------------
%------------------------------------------------------------------------------

\section{LUP- and QRP-based CUR LRA }\label{scntrp}

%------------------------------------------------------------------------------

\subsection{Computation of CUR LRA by means of rank-revaling factorization}\label{scgrbcalg}  
  
%------------------------------------------------------------------------------
  
  The algorithms  of
   \cite{GE96} and \cite{P00} are applied to a matrix $W$ of a  small numerical rank $r$ and 
  output its submatrices 
   with
locally maximal volumes and with $r$th largest singular values  
 bounded as follows.    
         
 \begin{theorem}\label{thcndprt} (See \cite{CI94}, \cite{GZT95}, \cite{GTZ97}, \cite{GTZ97a}, 
 \cite{GE96},  \cite{P00}.)
 
Suppose that $\min\{h,h'\}\ge 1$,  
$W_{\mathcal I,\mathcal J}\in 
\mathbb C^{k\times l}$ is a submatrix 
of a matrix $W\in \mathbb C^{m\times n}$, and
$t_{p,r,g}^2=(p-r)rg^2+1$ (cf. (\ref{eqtkrh})). Then
 $$t_{n,r,h}~\sigma_r(W_{\mathcal I,\mathcal J})\ge\sigma_r(W_{\mathcal I,:})$$
if $k=r\le l$ and if the volume  $v_2(W_{\mathcal I,\mathcal J})$
is locally column-wise  $h$-maximal  
 and
$$t_{m,r,h'}~\sigma_r(W_{\mathcal I,\mathcal J})\ge\sigma_r(W_{:,\mathcal J})$$
if $k\ge l=r$ and if this volume is locally  row-wise $h'$-maximal. 
\end{theorem}
Notice that 
$v_2(W_{\mathcal I,\mathcal J})=
v_{2,r}(W_{\mathcal I,\mathcal J})$
for the above matrices $W_{\mathcal I,\mathcal J}$ of sizes $r\times l$
and $k\times r$.
\begin{proof}
The theorem turns into \cite[Lemma 3.5]{P00} for $k=l=r$ and is extended to
the case where $r=\min\{k,l\}$ because no singular value of a  matrix 
increases in the transition to  its  submatrix.
\end{proof}

%------------------------------------------------------------------------------

 \begin{corollary}\label{cocndprt}
 The estimates of Theorem \ref{thcndprt} 
 hold for any $h\ge 1$ and  $k\times l$ matrix
$W_{\mathcal I,\mathcal J}$ output by 
\cite[Algorithm 4 for $f=h$]{GE96}
applied 
 to a matrix $W$ of size $m\times l$ 
 or $k\times n$.\footnote{ \cite[Algorithm 4 for $f=h$]{GE96} 
and similarly \cite[Algorithm 3 for  $\mu=h$]{P00} use parameters $f$ and $\mu$ a little exceeding 1 
in order to control the impact of rounding errors.} 
 \end{corollary}
 \begin{proof}
 The corollary holds if $k=l=r$ because 
 the output matrix $W_{\mathcal I,\mathcal J}$ has column-wise or row-wise locally 
 maximal volume in this case. 
  By applying  Algorithm \ref{algprjvlm}
 we extend this result to any $k$ and $l$ not exceeded by $r$.   
  \end{proof}
      
\begin{remark}\label{recurgge}
Suppose that 
an  $m\times l$ input matrix $W$ for $m\ge l$ has numerical rank 
$r$, that is, the ratio $||W||/\sigma_r(W)$ is not large, while  $\sigma_r(W)\gg \sigma_{r+1}(W)$. 
Then  
 Corollary \ref{cocndprt} defines a CUR generator $W_{k,l}=W_{\mathcal I,\mathcal J}$
such that $||U||\le t_{m,k,h}/||\tilde\sigma_{r+1}||$ for $\tilde\sigma_{r+1}$. Its computation involves  $ml$ memory cells and 
$O(ml^2)$ flops.
 By virtue of Claim (ii) of Corollary 
\ref{cowc1}
such a CUR generator defines a
close canonical CUR LRA of the matrix $W$.
If $W$ is a tall-skinny matrix, that is,
$m\gg l$ and if $k\ll m$,
then by applying the randomized algorithms of Theorems \ref{thsngvsmpl} and \ref{thnrmsmpl}
we can compute its $q\times n$
submatrix $\bar W$ 
%of the matrix $W$
for $q=4fk\ln(2k/\delta)$ that for sufficiently large factor $f$  
whp
has its singular values within a 
 factor $1+\phi$
 from those of the matrix $W$ for 
 $\phi\rightarrow 0$ as $q\rightarrow \infty$.
Then we can apply the algorithms of 
 Corollary \ref{cocndprt} to the matrix 
 $\bar W$ rather than $W$
 and decrease the upper bound on the 
 norm $||U||$ of the nucleus $U$
 by a factor of $t_{m,k,h}/t_{q,k,h}$
and also use by a factor of $l$ fewer flops.
 \end{remark}
                                                                                                                                                           
\begin{remark}\label{rep00}
\cite[Algorithm 4 for $f=h$]{GE96} relies on computing strong rank-revealing 
QR factorization of the matrix $W$. Alternative \cite[Algorithm 3 for  $\mu=h$]{P00} computes strong rank-revealing 
LU factorization. It also uses $O(mn\min\{m,n\})$ flops and ensures the upper bounds of Theorem \ref{thcndprt}
for $t_{n,r,h}$  replaced by $t_{n,r,h}^2$
and for $t_{m,r,h'}$ replaced by $t_{m,r,h'}^2$, but it is simpler for implementation, and we used it extensively in our experiments.
 \end{remark}

%------------------------------------------------------------------------------

\subsection{An iterative LUP-based
 C-A algorithm}\label{scrssalg}

%------------------------------------------------------------------------------
%------------------------------------------------------------------------------

The iterative C-A algorithm of \cite{GOSTZ10} seeks a $r\times r$ submatrix of locally maximal volume
in a $r\times n$ matrix.
 We first revisit and then a little modify 
  this algorithm.

%------------------------------------------------------------------------------

\begin{definition}\label{defgrdc}
Let $C_g$ for $g>1$ denote a $r\times g$ matrix and  
let $C_{g,j}$ denote its $r\times (g-1)$ submatrix obtained by removing
the $j$th column of $C_{g}$. 
If $v_2(C_{g,j'})=\max_{j=1}^g v_2(C_{g,j})$, then write $C_{g-1}=:C_{g,j'}$ 
and call the map $C_{g}\rightarrow C_{g-1}$ a {\em greedy local contraction}
or just a {\em greedy contraction}
of the matrix $C_g$.
% A {\em greedy reversed sequence} $C_{g-i},\dots,C_g$
%is defined by  $i$ recursive steps of greedy reversion beginning with $C_g$ for $i<g$. 
\end{definition}

%------------------------------------------------------------------------------
%------------------------------------------------------------------------------

\begin{definition}\label{defgrd}
Let $C_g$ be a $r\times g$ submatrix of a $r\times n$ 
matrix $W$,
append on its right
the $j$th column of the matrix $W$, and let $C_{g,j,+}$ denote the resulting 
$r\times (g+1)$ matrix.  
If $v_2(C_{g,j',+})=\max_{j=1}^n v_2(C_{g,j,+})$, 
then write 
$C_{g+1}=:C_{g,j',+}$ and
 call the map $C_{g}\rightarrow C_{g+1}$ a {\em greedy local expansion}
or just a {\em greedy  expansion}
of the matrix $C_g$.
\end{definition}

The algorithm of  \cite{GOSTZ10} outputs 
  a submatrix having locally maximal volume 
among the 
$r\times r$ submatrices
of a $r\times n$  matrix.
The algorithm  recursively performs a
greedy expansion followed by a greedy contraction of a fixed matrix. One 
continues the computations until the volume of a current matrix  becomes locally maximal.

Each recursive step involves  just $(n-r)r$ flops
and strictly increases the volume 
of the submatrix. Thus the algorithm can encounter no 
submatrix twice and has not more iterations than the
 exhaustive search. 
For the worst case input we have no better upper bound, but  
empirically much fewer iterations are
usually sufficient.
  
The paper \cite{GOSTZ10} analyzes  the algorithm by using  basic concept of a dominant $r\times r$
submatrix $B$ of a $r\times n$
matrix $A$ of full rank $r$. $B$ is said to be  {\em dominant} in $A$ if 
it is nonsingular and if 
$|c_{ij}|\le 1$ for every entry $c_{ij}$ of the matrix
$C=B^{-1}A$. In order to simplify the analysis 
assume that $C=I$. (For motivation recall that the volumes of all $r\times r$ submatrices
are only multiplied by $1/\det(B)$
in the transition from $A$ to $C$.)

Now notice that
  the volume of a dominant submatrix $B$ is 
locally maximal in $A$; furthermore, by virtue of Hadamard's bounds 
(\ref{eqhdm}),
it is $h$-maximal in $A$ for 
$h\le \max_{\mathcal J}\prod_{j\in \mathcal J}||{\bf a}_j||$
where maximization is over all subsets $\mathcal J$ of cardinality $r$ in the set $\{1,2,\dots,n\}$,
and so
$\max_{\mathcal J}\prod_{j\in \mathcal J}||{\bf a}_j||\le \max_{j=1}^n ||{\bf a}_j||^r$. 
This is at most $r^{r/2}$ because 
$|c_{ij}|\le 1$ since $B$ is a dominant
submatrix.

%------------------------------------------------------------------------------

\begin{algorithm}\label{algdomin} 
{\rm  Computing a dominant submatrix.} 

%------------------------------------------------------------------------------

\begin{description}

%------------------------------------------------------------------------------

\item[{\sc Input:}]
A $r\times n$ matrix $A$ of full rank $r$. 

%------------------------------------------------------------------------------

\item[{\sc Output:}] 
A $r\times r$ dominant submatrix $B$.

%------------------------------------------------------------------------------

\item[{\sc Initialization:}]
Fix a $r\times r$ nonsingular submatrix of $\bar B$ and place
 it in the leftmost position by 
 reordering the columns of the matrix $A$.
%and two real values $h_c\ge 1$ and $h_r\ge 1$. Set $s=0$.
%that produces a $r\times r$ submatrix of a given $r\times n$
%or $m\times r$ matrix satisfying a fixed stopping criterion. 

%-------------------------------------------------------------------------------

\item[{\sc Computations:}]

\begin{enumerate}
\item %1 
Compute the matrix $C=\bar B^{-1}A$ and
find its absolutely maximal entry $c_{ij}$.
If $|c_{ij}|\le 1$, 
then $\bar B$ is dominant; in this case 
output $B=:\bar B$ and stop.

\item %2
 Otherwise swap columns $i$ and $j$ of $C$.
(This strictly  increases the volumes of the leftmost $r\times r$
submatrices of $C$ and $A$.)
Denote by $\bar B$ the leftmost $r\times r$
submatrix of $A$ and go to stage 1.
\end{enumerate} 

%------------------------------------------------------------------------------

\end{description}

%------------------------------------------------------------------------------

\end{algorithm}

%------------------------------------------------------------------------------
%------------------------------------------------------------------------------

The algorithm uses 
 $i(r)=O(r^3)$ flops for inverting an $r\times r$
matrix at the initialization stage and then 
uses $(n-r)r$ flops per iteration (by exploiting a very special form of the matrix $\bar B$), that is,
 $(n-r)r\alpha+i(r)$ in $\alpha$ iterations.
%(As we said, no good upper bounds are known 
%on the number of iterations required
 %for the worst case input.) 

Next we specify the choice of the initial matrix $\bar B$ 
and express the 
algorithm via recursive LU factorization with column pivoting.

\begin{algorithm}\label{alggos} {\rm An iterative C-A algorithm by means of recursive LUP factorization.}

%------------------------------------------------------------------------------

\begin{description}

%------------------------------------------------------------------------------

\item[{\sc Input}~{\rm and}~{\sc Output}] 
as in Algorithm \ref{algdomin}.

%------------------------------------------------------------------------------

\item[{\sc Computations:}]

%------------------------------------------------------------------------------

\begin{enumerate}
\item %1
Compute 
LUP factorization of the matrix $A$,
where $L$ is a $r\times r$ lower triangular matrix,
$U=(U_0~|~U_1)$ is a $r\times n$ upper triangular matrix, 
$U_0$ is a nonsingular $r\times r$ matrix,
and $P$ is an $n\times n$ permutation matrix.

\item %2
Compute the matrices $\bar B=LU_0$ and
 $C=\bar B^{-1}AP^T=(I_r~|~C')$
 for 
 $C'\in \mathbb C^{r\times (n-r}$.  
($\det(C_{.,\mathcal J'})=\det(\bar B^{-1})\det((AP^T)_{.,\mathcal J'})$
for any $r$-tuple  of indices
$\mathcal J'$, and so the map $AP^T\rightarrow A'$
keeps pairwise order of
the volumes of  $r\times r$  submatrices.)

\item %3  
If $||A'_1||_C\le 1$,  
output the  submatrix 
$A_{.,\mathcal J}$ of the matrix
$A$ where the  set
 $\mathcal J$ 
is made up of the $r$ indices of  nonzero columns of the permutation matrix $P$.
(In this case $I_r$ is a dominant $r\times r$ submatrix of $A'$.)   
Otherwise write $C\rightarrow A$ and go to stage 1.
\end{enumerate}

%------------------------------------------------------------------------------

\end{description}

%------------------------------------------------------------------------------

\end{algorithm}

%------------------------------------------------------------------------------

\subsection{Computation of a contracted generator: basic techniques and results}\label{smtvbst}

%------------------------------------------------------------------------------

Osinsky's techniques of \cite{O16}  support
extension of the algorithm of \cite{GOSTZ10}
towards  the
maximization of the volume or $r$-projective volume of a $k\times l$ CUR generator
for $r<\min\{k,l\}$. 
Next we explore and extend some   recipes of \cite{O16} towards this goal by recursively applying QRP factorization. The
arithmetic cost of their iterative step   a little exceeds that of \cite{GOSTZ10},
but we decrease the output error bound, based on Theorem \ref{th3}.
  
We
begin with some  results supporting
greedy expansion/contraction.

Hereafter a matrix is said to be a {\em matrix basis} for its range, that is, for its column span.

The following theorems provide sufficient criterion 
for local maximality of
the volume of a fixed submatrix.

\begin{theorem}\label{thappnd} 
Let $B=(A~|~{\bf b})$ denote a $r\times (q+1)$ matrix for $r>q$
and let $U$ be a unitary matrix basis for the null space of $A^*$.
Then
\begin{equation}\label{eqbca}
v_2(B)=v_2(A)~||U^*{\bf b}||.
\end{equation}
\end{theorem}
\begin{proof}
Recall that $v_2^2(B)=\det(B^*B)$ and 
notice that 
$$B^*B=\begin{pmatrix} A^*A &  A^*{\bf b}  \\
{\bf b}^*A   &{\bf b}^*{\bf b}\end{pmatrix}
=\begin{pmatrix} I_r &  {\bf 0}  \\
{\bf x}^*  &  1 \end{pmatrix}\begin{pmatrix} A^*A &  {\bf 0}  \\
{\bf 0}^*  &  &{\bf b}^*S{\bf b} \end{pmatrix}\begin{pmatrix} I_r &  {\bf x}  \\
{\bf 0}^*  &  1 \end{pmatrix}$$ for  ${\bf x}: =(A^*A)^{-1}A^*{\bf b}$
and $S: =I_r-A(A^*A)^{-1}A^*$.
Hence
$v_2^2(B)=\det(B^*B)=\det(A^*A)\det({\bf b}^*(I_r-A(A^*A)^{-1}A^*){\bf b})$.
Recall that $S=I_r-A(A^*A)^{-1}A^*=UU^*$ where $U$ is a unitary matrix basis for 
the null space of the matrix $A$ (see \cite[Section 1.4.2]{S98}).
Substitute $\det(A^*A)=v_2^2(A)$ and $I_r-A(A^*A)^{-1}A=UU^*$ and deduce the theorem.
\end{proof}

\begin{corollary}\label{coorthap} 
A $r\times q$ submatrix $A$ of a $r\times n$ matrix $W:=(A~|~V)$ for $q<r\le n$
has locally $h$-maximal volume for $h$ equal to the maximal norm of the column 
of the matrix $U^*V$ 
where $U$ is a unitary matrix basis for the null space of the matrix $A$.
\end{corollary}
\begin{proof}
Suppose that $v_2(A)\le v_2(\bar A~|~{\bf b})$ where $\bar A$ is a $r\times (q-1)$ submatrix of $A$
and ${\bf b}$ is a column of $V$. Write $B=(A~|~{\bf b})$ as in  Theorem \ref{thappnd}
and deduce from this theorem that
$v_2(B)= v_2(A)||U^*{\bf b}||$. Hence $v_2(B)\le h~v_2(A)$ for $h=\max ||U^*{\bf b}||$. Now recall
that $v_2(B)\ge v_2((\bar A~|~{\bf b}))$
because $(\bar A~|~{\bf b})$ is a submatrix of $B$.
Thus $h v_2(A)\ge  v_2(B) \ge v_2((\bar A~|~{\bf b}))$.
\end{proof}
   
%-------------------------------------------------------------------------------

\begin{corollary}\label{coth2}
Remove the $j'$th column ${\bf b}_{j'}$, $1\le j'\le p$, from 
  a $p\times p$ upper triangular matrix $B=({\bf b}_j)_{j=1}^p$,
denote the resulting $p\times (p-1)$ submatrix $A=A_{j'}$, and  denote $b_{j'}$  the $j'$th 
coordinate of the vector ${\bf b}_{j'}$.
Then 
\begin{equation}\label{eqbtr}
v_2(B)=v_2(A)~||{\bf b}_{j'}||.
\end{equation}
\end{corollary}

%------------------------------------------------------------------------------

\begin{proof} 
Apply Theorem \ref{thappnd} and observe that in this case $U$ is the $j'$th coordinate vector,
$(0,\dots,0,1,0,\dots,0)^*$, filled with zeros except for the $j'$th coordinate 1.
\end{proof}

%-------------------------------------------------------------------------------

\begin{theorem}\label{th2} Let $B: =(A~|~{\bf b})$ denote a $p\times (q+1)$ matrix for $p\le q$.
Then\footnote{We slightly simplify the proof of  a similar result hidden in the proof of  \cite[claim 4 of Lemma 1]{O16}.}
\begin{equation}\label{eqa}
v_2^2(A)=v_2^2(B)(1-{\bf b}^*(BB^*)^{-1}{\bf b})=v_2^2(B)(1-||B^+{\bf b}||^2)
\end{equation}

and 
\begin{equation}\label{eqb}
v_2^2(B)=v_2^2(A)(1+{\bf b}^*(AA^*)^{-1}{\bf b})=v_2^2(A)(1+||A^+{\bf b}||^2).
\end{equation}
In particular if the matrix $B$ is unitary,
then 
\begin{equation}\label{eqaun}
v_2^2(A)=v_2^2(B)(1-||{\bf b}||^2),
\end{equation}
and  if the matrix $A$ is unitary,
then 
\begin{equation}\label{eqbun}
v_2^2(B)=v_2^2(A)(1+||{\bf b}||^2).
\end{equation}
\end{theorem}

%------------------------------------------------------------------------------

\begin{proof} 
Observe that
$$AA^*=BB^*-{\bf b}{\bf b}^*,$$
and so
$$\det(AA^*)=\det(BB^*-{\bf b}{\bf b}^*)=\det(BB^*)\det(I_p-(BB^*)^{-1}{\bf b}{\bf b}^*).$$

Substitute 
$\det(BB^*)=v_2^2(B)$,
$\det(I_p-(BB^*)^{-1}{\bf b}{\bf b}^*)=\det(I_1-{\bf b}^*(BB^*)^{-1}{\bf b})$, 
and $I_1=1$
and obtain equation (\ref{eqa}).
One can similarly prove equation  (\ref{eqb}).
\end{proof}

%------------------------------------------------------------------------------

Similarly to Corollary \ref{coorthap}
deduce the following result (use equation (\ref{eqb}) instead of Theorem \ref{thappnd}).

%------------------------------------------------------------------------------

\begin{corollary}\label{colclmx} 
Given a $r\times q$ submatrix $A$ of a $r\times n$ matrix $(A~|~V)$ for $r\le q\le n$,
let $b$ denote the maximal spectral norm of the columns of the matrix $A^+V$.
Then $v_2(A)$ is locally $(1+b^2)$-maximal.
\end{corollary}

%-------------------------------------------------------------------------------

\cite[Lemma 4]{O16} extends (\ref{eqb}) to a bound on the $r$-projective volume
as follows:
\begin{theorem}\label{th3pr}  
Suppose that again $B: =(A~|~{\bf b})$ denote a $p\times (q+1)$ matrix 
and let $v_{2,r}(A)$ be the maximal $r$-projective volume among 
all $p\times q$ submatrices of the matrix $B$.
Then
\begin{equation}\label{eqbpr}
 v_{2,r}^2(B)\ge v_{2,r}^2(A)(1+{\bf b}^*(A(r)A(r
)^*)^{-1}{\bf b})=v_{2,r}^2(A)(1+||A(r)^+{\bf b}||^2).
\end{equation}
\end{theorem}
%For two matrices $A$ of size $p\times q$ and 
%$B$ of size $q\times s$ it holds that
%$v_2(AB)=v_2(A)v_2(B)$ if $p=q\le s$ or $q=s\le p$.
%\end{theorem}

%------------------------------------------------------------------------------

The theorem  
does not imply extension of Corollary \ref{colclmx}, but one can maximize
 $r$-projective volume via maximization of volume by applying Algorithm  \ref{algprjvlm}. 

%------------------------------------------------------------------------------

\subsection{QRP-based greedy iterative search for  locally maximal volume}\label{scrsorth}

%------------------------------------------------------------------------------

Corollaries \ref{coorthap} and    \ref{colclmx} motivate recursive application of greedy expansion and contraction based on
 Theorems \ref{thappnd} and  \ref{th2}.  
Suppose that we seek a $r\times l$ submatrix of a $r\times n$ matrix $W$, for $\max\{r,l\}\le n$.
 We can fix its column having the largest norm,
denote it $C_1$, recursively apply greedy expansion
based on  Theorems \ref{thappnd} and  \ref{th2}, until we arrive at a $r\times l$
submatrix $C_l$ of $W$. Then we can check
whether the sufficient criteria of Corollaries \ref{coorthap} and  \ref{colclmx} 
 for local maximization of the volume $v_2(C_l)$
are satisfied, and if not, apply a sequence of expansions/contractions to the matrix $C_l$
similarly to Algorithm \ref{algdomin}.  
Next we describe two 
implementations of this outline.

%------------------------------------------------------------------------------

\begin{algorithm}\label{alg1tor} 
{\rm  Greedy  expansions
from a vector up to a square matrix.} 

%------------------------------------------------------------------------------

\begin{description}
 
%------------------------------------------------------------------------------

\item[{\sc Input:}] 
A $r\times n$ matrix  $W$ and an integer
$q$ such that $q\le r\le n$.
  
%------------------------------------------------------------------------------

\item[{\sc Output:}] An $n\times n$ permutation matrix $P$ such that  
every leftmost   $r\times g$ submatrices $C_g$ of $WP$ is a greedy expansion 
of its $r\times (g-1)$  leftmost  predecessor $C_{g-1}$ for $g=2,\dots,q$.

%------------------------------------------------------------------------------

\item[{\sc Initialization:}] 
Write $W_0=V_0: =W$, $U_0: =I_r$, $P_0: =I_n$ and let $C_0$ denote
the empty $r\times 0$ matrix.

\item[{\sc Computations:}] 
 
For $g=0,1,\dots,q-1$, recursively proceed as follows:

%------------------------------------------------------------------------------

\begin{enumerate}
\item%1. 
Given a $r\times n$ matrix $W_g: =U_iWP_g=(C_g~|~V_g)$,
for a unitary matrix $U_g$, a permutation matrix $P_g$, a $r\times g$
matrix $C_g: =\begin{pmatrix} R_g\\ O_{g,r-g} \end{pmatrix}$,
and an $g\times g$ 
 upper triangular matrix  $R_g$ having all its diagonal entries 1,
fix a column vector ${\bf b}_g$ whose subvector made up of its last
$r-g$ coordinates has maximal
spectral norm among all such subvectors of
the column vectors of the matrix $V_g$.
Form the matrix  $B_g: =(C_g~|~{\bf b}_g)$.
\item%2.
Move the vector ${\bf b}_g$ from its position in the matrix $W_g$  into the
 $(g+1)$st column of the new matrix, $W_g': =W_gP_{g+1}=(C_g'~|~V_g')$ where 
$C_g': =(C_g~|~{\bf b}_g)$.
\item%3.
Define a Householder reflection matrix $H'_g$
 such that every vector ${\bf v}$ shares its first $g$ coordinates with the vector
  $H_g{\bf v}$ for $H_g=\diag(I_g,H'_g)$,
  while
  the vector
${\bf b}_{g+1}: =H_g{\bf b}_g$ 
shares  shares its remaining $r-g$ coordinates with the 
coordinate vector $(1,0,\dots,0)^T$ of dimension $r-g$ (see \cite[Section 5.1]{GL13}). 
\item%4.
Compute the matrix $W_{g+1}: =H_gW_g'$ and let $C_{g+1}$ denote its 
$r\times (g+1)$ leftmost submatrix, such that
$C_{g+1}: =H_gC_{g}'=\begin{pmatrix} R_{g+1}\\ O_{g+1,r-g-1} \end{pmatrix}$
and $R_{g+1}: =(R_g~|~{\bf b}_{g+1})$ is a $r\times (g+1)$
 upper triangular matrix with diagonal entries 1.
\item%5.
Compute and output the permutation matrix $P: =P_0P_1\cdots P_{q-1}$.
\end{enumerate} 

%------------------------------------------------------------------------------

\end{description}

%------------------------------------------------------------------------------

\end{algorithm}

%------------------------------------------------------------------------------
 
{\em Correctness} of the algorithm follows because the matrices $U_g=\diag(O_{g,g},I_{r-g})$
form unitary matrix bases for the null spaces of the matrices $C_g$ for $g=0,\dots,q-1$.

{\em Computational cost}. The algorithm computes the squares of the $l_2$-norms of
$n-g$ vectors of dimensions   $r-g$ for $g=0,\dots,q-1$ by using $\sum_{g=0}^{q-1}(n-g)(2r-2g-1)<2qrn$ flops.
 It   multiplies  $(r-g)\times (r-g)$ Householder reflection matrices by
$(r-g)\times (n-g)$ matrices by using at most $\sum_{g=0}^{q-1} 6(n-g)(r-g))<6qrn$ flops, hence less than $8qrn$ flops overall.

%------------------------------------------------------------------------------

\begin{algorithm}\label{algrup} 
{\rm Greedy expansions from a $r\times r$ submatrix to a $r\times q$ submatrix for $q>r$.} 

%------------------------------------------------------------------------------

\begin{description}
 
%------------------------------------------------------------------------------

\item[{\sc Input:}] 
An $r\times n$ matrix  $W$, with 
a $r\times r$ leftmost submatrix $C_r$, 
% have been  computed by means of greedy expansions with Algorithm \ref{alg1tor},
and an integer $q$, $r<q\le n$.
%(W.l.o.g. assume that they are the leftmost submatrices of the matrix $W$.)
  
%------------------------------------------------------------------------------

\item[{\sc Output:}] An $n\times n$ permutation matrix $P$ 
such that every    
 $r\times (g+1)$  leftmost submatrix $C_{g+1}$ of 
 the matrix $WP$ for $g=r,\dots,q-1$ has been computed by means of 
 greedy expansion of its preceding $r\times g$ leftmost
  submatrix $C_{g}$.

%------------------------------------------------------------------------------

\item[{\sc Initialization:}] Write $W=W_r: =(C_r~|~V_r)$.

\item[{\sc Computations:}] 

 Recursively, for $g=r,r+1,\dots,q-1,$  proceed as follows:

%------------------------------------------------------------------------------

\begin{enumerate}
\item%1. 
Pre-multiply the matrix $W_g$ by a $r\times r$ matrix $R_g$ that orthogonalizes
the submatrix $C_g$. 
%Re-denote $R_gW_g\rightarrow W_g$.
\item%2. 
Among the columns of the matrix $R_gV_g$, 
select a column  vector ${\bf b}_{g+1}$
 having the maximal spectral norm.  
\item%3. 
Move this vector into the $(g+1)$st position in the matrix $R_gW_g$,
thus turning $R_gW_g$ into the  matrix $W_{g+1}: =R_gW_gP_g$ for a permutation matrix $P_g$.
 If $g=q-1$, stop and output the permutation matrix $P: =\prod_{g=r}^{q-1}P_g$. 
\end{enumerate} 

%------------------------------------------------------------------------------

\end{description}

%------------------------------------------------------------------------------

\end{algorithm}

%------------------------------------------------------------------------------

Equation (\ref{eqbun}) implies that  the $g$th loop of invocation of stages 1--3,
 for $g=r,\dots,q-1$, 
appends to the matrix $C_g$
a column that maximally increases its volume; then {\em correctness}
of the algorithm follows.

The {\em computational cost}  of performing the algorithm 
amounts to the cost of $q-r-1$ orthogonalizations after 
$q-r-1$ movements of  columns; this involves  $O(rn)$ flops after each movement
(see \cite[Section 6.5.2]{GL13}), that is, $O((q-r-1)rn)$ flops overall.

Now suppose that for a fixed $h\ge 1$ we seek a  submatrix having locally $h$-maximal
volume among $r\times l$ submatrices of a  $r\times n$ matrix $W$.
Then, by applying Algorithms \ref{alg1tor} and \ref{algrup},
we  compute a greedy sequence of $l$ submatrices $C_g$, $g=1,\dots,l$.
Unless Corollary \ref{coorthap} or \ref{colclmx} implies that submatrix $C_l$ is locally $h$-optimal, we extend the greedy sequence by $p$ additional applications of  Algorithms \ref{alg1tor} and \ref{algrup}, for a fixed number $p$. Then we contract the sequence  back to the length $l$
by reversing these algorithms. We can successively apply this recipe for $p=1$; then  for $p=2$, and so on, until it finally works.

We can reverse Algorithm \ref{alg1tor} by applying  (\ref{eqbtr}) and reverse
Algorithm \ref{algrup} by applying equations 
(\ref{eqbun}) or (\ref{eqb}) instead of equation (\ref{eqa}).
%------------------------------------------------------------------------------
In the latter case, equations (\ref{eqbtr}),
(\ref{eqaun}), and (\ref{eqbun}) imply that 
at each cycle of expansion and contraction
 the volume of the input matrix $A$ cannot exceed 
that of the  output matrix $A$.

\medskip

%------------------------------------------------------------------------------

\noindent {\bf Acknowledgements:}
Our research has been supported by NSF Grants CCF--1116736 
and CCF--1563942 and PSC CUNY Award  68862--00 46.
Victor Pan is also grateful 
 to
A. Osinsky and N. L. Zamarashkin
for  a preprint of the paper \cite{O16}
and ample verbal introduction to it, 
to S. A. Goreinov for reprints of his papers, to I. V. Oseledets and
A. V. Knyazev for helpful pointers to the bibliography,
 to E. E. Tyrtyshnikov for both
such pointers and the challenge of
 formally supporting fast empirical convergence of C-A algorithms,
and to V. Mehrmann
 for his advice for in depth study of the subject after the 
submission to LAA  in the summer of 2016 of the first draft of this paper.

%------------------------------------------------------------------------------
%------------------------------------------------------------------------------

%1------------------------------------------------------------------------------

\end{document}